\documentclass[12pt]{amsart}

\usepackage{amssymb}
\usepackage{hyperref} 
\usepackage{amsmath}
\usepackage{latexsym}
\usepackage{extarrows}
\usepackage{enumerate}
\usepackage{mathtools}
\usepackage{bm}
\usepackage[all]{xy}
\usepackage{tikz-cd}
\usepackage{xcolor}
\usepackage{mathrsfs}
\usepackage{euscript}
\usepackage{csquotes}
\usepackage{comment}
\usepackage{bbm}
\usepackage{cleveref}

\makeatletter
\@namedef{subjclassname@2020}{%
  \textup{2020} Mathematics Subject Classification}
\makeatother

\usepackage[T1]{fontenc}

\calclayout

\theoremstyle{plain}

\newtheorem{theorem}{Theorem}[section]
\newtheorem{proposition}[theorem]{Proposition}

\newtheorem{lemma}[theorem]{Lemma}
\newtheorem{corollary}[theorem]{Corollary}

\theoremstyle{definition}

\newtheorem{definition}[theorem]{Definition}

\newtheorem{remark}[theorem]{Remark}
\newtheorem{remarks}[theorem]{Remarks}
\newtheorem{example}[theorem]{Example}
\newtheorem{examples}[theorem]{Examples}

\newcommand\Z{\mathbb{Z}}
\newcommand\R{\mathbb{R}}

\newcommand\T{\mathbb{T}}
\newcommand\C{\mathbb{C}}
\newcommand\N{\mathbb{N}}

\renewcommand\subset{\subseteq}
\renewcommand\phi{\varphi}

\newcommand\euE{\EuScript{E}}

\newcommand\euS{\EuScript{S}}
\newcommand\euT{\EuScript{T}}

\newcommand\uX{\mathrm{X}}
\newcommand\uY{\mathrm{Y}}

\newcommand\fix{\operatorname{fix}}
\newcommand\lin{\operatorname{lin}}

\begin{document}


\baselineskip=17pt
\setlength\parindent{0pt}


\title[Relatively ergodic compact extensions]{A Peter--Weyl theorem for compact group bundles and the geometric representation of relatively ergodic compact extensions}
\author[N. Edeko]{Nikolai Edeko}
\address{Nikolai Edeko, Institut für Mathematik, Universität Zürich, Winterthurerstrasse 190, CH-8057 Zürich, Switzerland}
\email{nikolai.edeko@math.uzh.ch}

\author[A. Jamneshan]{Asgar Jamneshan}
\address{Asgar Jamneshan, Mathematical Institute, University of Bonn, Endenicher Allee 60, D-53115 Bonn, Germany}
\email{ajamnesh@math.uni-bonn.de}

\author[H. Kreidler]{Henrik Kreidler}
\address{Henrik Kreidler, Fachgruppe für Mathematik und Informatik, Bergische Universität Wuppertal, Gaussstrasse 20, D-42119 Wuppertal, Germany}
\email{kreidler@uni-wuppertal.de}

\date{\today}

\begin{abstract}  
We show that a relatively ergodic extension of measure-preserving dynamical systems has relative discrete spectrum if and only if it can be represented as a skew-product by a bundle of compact homogeneous spaces. Our result holds without restrictions on the acting group or the underlying probability spaces. This generalizes previous work by Mackey, Zimmer, Ellis, Austin, and the second author and Tao, and is inspired by the Furstenberg--Zimmer and Host--Kra structure theories for actions of uncountable groups. Our approach translates the ergodic-theoretic question into topological dynamics, where we establish a corresponding classification: an extension in topological dynamics has relative discrete spectrum precisely when it admits a skew-product representation by bundles of compact homogeneous spaces. A key step in our argument is establishing a Peter--Weyl-type theorem for bundles of compact groups which might be of independent interest.
\end{abstract}



\maketitle

\section{Introduction}

There are numerous interactions between topological dynamics and ergodic theory, some heuristic and others explicit (see, e.g., \cite{ellis,furstenberg2014recurrence,gw,ggy,hkm,linden}). This paper strengthens one of these structural analogies into a precise correspondence. Specifically, we develop a geometric representation theory for extensions of topological dynamical systems with relative discrete spectrum. Our work yields a geometric representation of measure-preserving dynamical systems with relative discrete spectrum in settings that are non-ergodic and involve uncountable groups and inseparable probability spaces, aligning with recent efforts to extend ergodic-theoretic structural results to such settings (see, e.g., \cite{EHK2021,jt19,jt20,jamneshan2019fz}).

One motivation for our study comes from seeking uncountable and non-ergodic versions of the Furstenberg--Zimmer \cite{furstenberg1977ergodic,zimmer1976ergodic,zimmer1976extension} and Host--Kra structure theories \cite{host2005nonconventional,ziegler2007universal,bergelson2010inverse,jst}, which are relevant, for example, in combinatorial and ergodic applications involving ultraproducts (see, e.g., \cite{cz,jt21-1,szegedy-inv2,tz-finite,minghao}) or Stone-\v{C}ech compactifications (see, e.g., \cite{hs,hs1,MRR2019}) that naturally yield uncountable groups and inseparable spaces. 

Before detailing our contributions, we first summarize prior results and methodologies related to the Mackey--Zimmer representation theorem.

\subsection{Related literature}

For the sake of the discussion, let us recall the classical Mackey--Zimmer representation theorem \cite{Mackey1966,zimmer1976extension}. Let $G$ be a (discrete) countable group, let $\uX=(X,\Sigma_X,\mu_X)$ be a Lebesgue probability space, and let $\varphi_X\colon G\times X\to X, \, (t,x) \mapsto (\varphi_X)_t(x)$ be an ergodic measure-preserving near-action, i.e., $(\varphi_X)_1 = \mathrm{id}_X$ almost everywhere and $(\varphi_X)_{st} = (\varphi_X)_{s} \circ (\varphi_X)_t$ almost everywhere for $s,t \in G$ (see \cite[Definition 3.1]{zimmer1976extension}, and see \cite{Einsiedler2011} for a general introduction.). An extension $q \colon (X,\Sigma_X,\mu_X,\varphi_X) \rightarrow (Y,\Sigma_Y,\mu_Y,\varphi_Y)$ of such systems is said to have \emph{relative discrete spectrum} if $\mathrm{L}^2(\uX)$ is the closure of all finitely generated invariant $\mathrm{L}^\infty(\uY)$-submodules of $\mathrm{L}^2(\uX)$, where $\mathrm{L}^\infty(\uY)$ embedded in $\mathrm{L}^2(\uX)$ by pullback acts on it by multiplication\footnote{Other authors say that the extension is \emph{compact} or \emph{relatively compact}, e.g., see Furstenberg in \cite[Chapter 6]{furstenberg2014recurrence}. Our terminology is  borrowed from Zimmer \cite{zimmer1976extension}.}. 

The Mackey--Zimmer theorem now states that any ergodic extension with relative discrete spectrum has the following geometric representation. There are a compact group $H$, a closed subgroup $U \subset H$ and a cocycle $c \colon G\times Y \rightarrow H$ such that $q$ is isomorphic (as extension of measure-preserving systems) to the skew-product extension $q'\colon \uY \rtimes_c H/U \rightarrow (Y,\Sigma_Y,\mu_Y, \phi_Y)$, see \cite[\S 3.5]{glasner2015ergodic} for notation and \cite[Theorem 9.14]{glasner2015ergodic} for a textbook proof, see also \cite{Tao-mz} for a streamlined proof. The converse statement, namely that a skew-product extension  has relative discrete spectrum, follows from the Peter--Weyl theorem (for this direction, no ergodicity assumptions are needed). 

We now discuss two extensions of the Mackey--Zimmer representation theorem that are relevant to this work.   
Ellis \cite{ellis} extends the theorem to ergodic actions of arbitrary (not necessarily countable) discrete groups on general (not necessarily separable) probability spaces by passing to a topological model and using tools from topological dynamics, see also \cite{jt20,jamneshan2019fz} for an alternative measure-theoretic proof. 
 In another direction, Austin \cite{austin-fund} relaxes the ergodicity assumption by extending the theorem to the relatively ergodic case, that is, when the invariant sets of $\uX$ are (up to null sets) given by the preimages of invariant sets in $\uY$.  More precisely, he establishes that a relatively ergodic extension with relative discrete spectrum is isomorphic to a skew-product, where the homogeneous space is allowed to vary measurably along the fibers in the ergodic decomposition of the base space. Similarly to the original result of Mackey and Zimmer, his theorem holds for actions of countable groups on standard Lebesgue spaces. 

A main goal of this article is to join both these extensions and prove a geometric representation of relatively ergodic extensions with relative discrete spectrum for actions of arbitrary groups on general probability spaces.  A relatively ergodic version of the classical Halmos--von Neumann theorem \cite{HaNe1942}, a precursor of the Mackey--Zimmer theorem, was established by the first author in \cite{edeko}.

The formal statement of our extension of the Mackey--Zimmer theorem is inspired by Austin's statement, although our methods are different. Indeed, Austin's proof relies on measurable selection theorems and the ergodic decomposition both of which break down for actions of uncountable groups on inseparable probability spaces. Our approach to pass to a natural topological model and studying the problem in topological dynamics is inspired by Ellis' investigation of the Mackey--Zimmer theorem. 

For the action of uncountable groups on inseparable spaces a \emph{canonical disintegration} is available after passing to a suitable topological model, see \cite[Theorem 2.3]{ellis} and see \cite[\S 7]{jt-foundational} for background and proof. However this disintegration is in general \emph{not} an ergodic decomposition, see the counterexample constructed in \cite[Appendix B]{jt-foundational}. Therefore, the  relatively ergodic case imposes serious formal and technical difficulties. We suggest a framework, where we can work \emph{simultaneously} in all the \enquote{fibers} over the invariant factor. To this end, it is essential to revisit the classical Peter--Weyl theorem and develop an analogous representation theorem for compact group bundles which may be of independent interest.  This along with the geometric representation theory of relatively ergodic isometric extensions in topological dynamics form the main results of this paper. Next, we state precisely these results and discuss our innovations. 

\subsection{Innovations and main results}

Our general approach to the classification result is the following: Instead of working directly within the setting of measure-preserving transformations, we use topological models (see Section \ref{measpres}) to transform the classification problem for extensions with relative discrete spectrum into a question of topological dynamics, i.e., the study of continuous actions $\varphi\colon G \times K \rightarrow K, \, (t,x) \mapsto \varphi_t(x)$ of a topological group $G$ on compact spaces $K$ (here, $G$ will be discrete). This idea is quite old and already applied effectively by  Ellis \cite{ellis} (see also  the proof of the Halmos-von Neumann theorem given in \cite[Theorem VIII.4]{DNP}, as well as \cite[Section 3]{kreidler-hermle}). However, we highlight that even in the case of ergodic systems considered by Ellis, our methods yield a new proof of the classification result avoiding the Choquet theoretic techniques used in \cite{ellis}. 

Due to this bridge between topological dynamics and ergodic theory, we first study extensions $q \colon (K;\varphi) \rightarrow (L;\psi)$ between topological dynamical systems which are \enquote{relatively ergodic} and \enquote{have relative discrete spectrum} in a topological sense.
We also have the crucial additional conditions (coming from the topological models) that $q$ is an open map and that $L$ is a Stonean  (i.e., extremally disconnected compact) space. 

Similar to the measure-theoretic situation described above, a typical example for an extension with relative discrete spectrum in this topological framework is a \emph{skew-product extension} $q_c \colon (L \times H/U;\varphi_c) \rightarrow (L;\psi)$ constructed from the following data: a topological dynamical system $(L;\psi)$, a compact group $H$, a closed subgroup $U \subset H$, and a \emph{continuous} cocycle $G \times L \rightarrow H$ into $H$. The action of $G$ on $L \times H/U$  is then given by $(\varphi_c)_t(l,xU) \coloneqq (\psi_t(l),c(t,l)xU)$ for $(l,xU) \in L \times H/U$ and $t \in G$.

We generalize this construction to obtain a \emph{relative skew-product extension} $q_c \colon (L \times_{L_{\fix}} H/U;\varphi_c) \rightarrow (L;\psi)$ by considering \enquote{relativized} data with respect to the invariant factor $L_{\fix}$ of $(L;\psi)$, , i.e., the maximal factor on which the dynamics is trivial. These consist of a \emph{compact group bundle} $p \colon H \rightarrow L_{\fix}$, a \emph{closed subgroup bundle} $p|_U \colon U \rightarrow L_{\fix}$, and a \emph{relative cocycle} $c \colon G \times L \rightarrow H$ into $H$ (see Section \ref{relskew} for the definition).

Theorems \ref{reldiscsp} and \ref{reptheorem1} then provide the following result on topological extensions with relative discrete spectrum which depends on certain openness conditions for the involved bundles (cf. Remark \ref{gap}).
\begin{theorem}\label{intro1}
    For an open relatively topologically ergodic extension $q \colon (K;\varphi) \rightarrow (L;\psi)$ with $L$ being Stonean consider the following assertions. 
        \begin{enumerate}[(a)]
            \item $q$ has relative discrete spectrum.\label{intro1a}
            \item There are 
            \begin{enumerate}[(i)]
                \item an open compact group bundle $p \colon H \rightarrow L_{\fix}$,
                \item a closed subgroup bundle $p|_U \colon U \rightarrow L_{\fix}$ of $p$, and
                \item a relative cocycle $c \colon G \times L \rightarrow H$, 
        \end{enumerate}
        such that the induced relative skew-product extension $q_c$ is isomorphic to $q$.\label{intro1b}
        \end{enumerate}
    Then (a) implies (b). Conversely, if (b) holds and the subgroup bundle $p|_U\colon U \rightarrow L_{\fix}$ in (ii) can also be chosen to be open, then (a) holds.
\end{theorem}

In fact, assuming (\ref{intro1a}), one can even achieve that the closed subgroup bundle $p|_U \colon U \rightarrow L_{\fix}$ and the relative cocycle $c \colon G \times L \rightarrow H$ in (\ref{intro1b}) satisfy certain irreducibility conditions, see Theorem \ref{reptheorem1} below for the details.

A relative skew-product extension induced by an open compact group bundle $p \colon H \rightarrow L_{\fix}$ admits a canonical unique and fully supported relatively invariant measure induced by the Haar bundle of $p \colon H \rightarrow L_{\fix}$. This allows to \enquote{lift} any invariant measure of $(L;\psi)$. Thus, topological skew-product extensions also give rise to measure-preserving ones. It is the main result of this article that, up to an isomorphism, the relative skew-product extensions (induced by open bundles) are precisely the relatively ergodic extensions with relative discrete spectrum (see Theorems \ref{skewdiscr} and \ref{mainres}). Here, measure-preserving systems on general probability spaces and their extensions are defined operator theoretically, or equivalently in terms of homomorphisms of measure algebras, see Section \ref{measpres} below. 

\begin{theorem}\label{intro2}
    For a relatively ergodic extension $J \colon (\uY;S) \rightarrow (\uX;T)$ of measure-preserving systems the following assertions are equivalent.
        \begin{enumerate}[(a)]
            \item $J$ has relative discrete spectrum.\label{intro2a}
            \item There are
        \begin{enumerate}[(i)]
            \item a topological dynamical system $(L;\psi)$,
            \item an open compact group bundle $p \colon H \rightarrow L_{\fix}$,
            \item an open closed subgroup bundle $p|_U \colon U \rightarrow L_{\fix}$ of $p$, 
            \item a relative cocycle $c \colon G \times L \rightarrow H$, and
            \item a fully supported invariant measure $\mu_L \in \mathrm{P}^\psi(L)$ on $L$,
        \end{enumerate}
            such that the induced relative skew-product extension $J_{q_c}$ is isomorphic to $J$.\label{intro2b}
        \end{enumerate}
\end{theorem}

Again, one can achieve that the closed subgroup bundle and the relative cocycle in  (\ref{intro2b}) are irreducible, and in addition the measure $\mu_L$ can be chosen to be normal (meaning that $L$ is the Stone space of the induced measure space), see Theorem \ref{mainres} below.

The proofs of Theorems \ref{intro1} and \ref{intro2} require relative versions of several known concepts and results which we now briefly discuss.

As indicated above, the proof that a skew-product extension has relative discrete spectrum rests on the Peter--Weyl theorem stating, in particular, that for a compact group $H$ the space $\mathrm{C}(H)$ of continuous functions on $H$ is generated by the (finite-dimensional) subspaces of representative functions $\mathrm{R}_{[\pi]}$ of $[\pi] \in \hat{H}$, where $\hat{H}$ is the unitary dual of $H$ (see, e.g., \cite[Theorem 5.12]{folland-harmonic-analysis}). To prove that every topological, and then also measure-preserving, relative skew-product extension (induced by open compact group bundles) has relative discrete spectrum, we therefore first establish a Peter--Weyl theorem for open compact group bundles as well as for the induced bundles of homogeneous spaces. 

As a first step, we introduce the dual $\hat{p} \colon \hat{H} \rightarrow M$ of such an open compact group bundle $p \colon H \rightarrow M$. Inspired by the classical connection between representations of locally compact groups and the associated C*-algebras (as discussed, e.g., in Dixmier's book \cite{Dixm1977}), as well as earlier work of Muhly, Renault and Williams (see \cite{MRW1996} and also \cite{Rena2021}) and Goehle (see \cite{Goeh2008}), we do so via the \emph{group bundle C*-algebra}\footnote{This is a special case of the more general construction of groupoid C*-algebras, see \cite{Rena1980}.} $\mathrm{C}^*(H)$ associated to $p$. In this way we obtain $\hat{H}$ as a locally compact (Hausdorff) space (see Proposition \ref{proptop}) fibered over $M$ which admits \enquote{many} local continuous sections $\hat{\sigma} \colon O \rightarrow \hat{H}$ (see Theorem \ref{section}). We can then define corresponding \emph{bundles of representative functions} $\mathrm{R}_{\hat{\sigma}} \rightarrow M$ which turn out to be locally trivial vector bundles over $M$ (see Theorem \ref{loctrivrep}).
By replacing the space of continuous functions on a compact group from the classical Peter--Weyl theorem by the \emph{Banach bundle of continuous fiber maps} $\mathrm{C}_p(H) \coloneqq \bigcup_{m \in M} \mathrm{C}(p^{-1}(\{m\})) \rightarrow M$, we then prove the following (see Theorem \ref{pw1} for a more detailed formulation).

\begin{theorem}
    Let $p \colon H \rightarrow M$ be an open compact group bundle. Then $\mathrm{C}_p(H)  \rightarrow M$ is generated by the locally trivial vector bundles  $\mathrm{R}_{\hat{\sigma}} \rightarrow O$ defined by all local continuous sections $\hat{\sigma} \colon  O \rightarrow \hat{H}$.
\end{theorem}

This and similar results (see Theorem \ref{pw}, as well as Theorems \ref{pw2} and \ref{pw3} for bundles of homogeneous spaces) are purely abstract harmonic analytic statements and are of independent interest. However, applied to topological dynamics and ergodic theory, they show that relative skew-product extensions (induced by open group bundles) have relative discrete spectrum.

In order to go in the converse direction and represent topological extensions via relative skew-products, we rely on the concept of uniform enveloping semigroupoid introduced by the first and the third author in \cite{EdKr2021}. They are inspired by the classical notion of enveloping semigroups due to Ellis (see \cite{Ellis1960}). For a topological dynamical system $(K;\varphi)$  with discrete spectrum it is known that the \emph{uniform enveloping semigroup} given by the closure $\mathrm{E}_{\mathrm{u}}(K;\varphi) \coloneqq \overline{\{\varphi_t \mid t \in G\}} \subset \mathrm{C}(K,K)$ with respect to the topology of uniform convergence (i.e., the compact-open topology) is a compact group. In the situation of extensions $q \colon (K;\varphi) \rightarrow (L;\psi)$ we obtain a similar statement if we consider the so-called \emph{uniform enveloping semigroupoid} $\euE_{\mathrm{u}}(q)$ which is constructed in the space of continuous fiber maps $\mathrm{C}_q^q(K,K) \coloneqq \bigcup_{m \in M} \mathrm{C}(p^{-1}(\{m\}),p^{-1}(\{m\})) $. This provides a canonical way to construct compact groups from extensions with relative discrete spectrum. Combined with a continuous selection theorem due to Gleason (see \cite{gleason}), the uniform enveloping semigroupoid is the key for representing extensions via relative skew-products.

The last essential ingredient in the proof of Theorem \ref{intro2} is a thorough inspection of relatively invariant measures and their connections to usual invariant measures. Given an extension of topological dynamical systems $q \colon (K;\varphi) \rightarrow (L;\psi)$ and an invariant normal measure $\mu_L \in \mathrm{P}^\psi(L)$ we show that the measures $\mu_K \in \mathrm{P}^\varphi(K)$ with $q_*\mu_K =\mu_L$ are in one-to-one correspondence with the relatively invariant measures of $q$ (see Proposition \ref{dynamicaldisint}). This observation rests on the previously mentioned canonical disintegration result already used by Ellis in \cite{ellis} which works without any countability assumptions. With this bijection, one can readily check that the topological model of a relatively ergodic extension with relative discrete spectrum is not only topologically a relative skew-product extension, but also measure-theoretically. This is the main step for deducing Theorem \ref{intro2} from Theorem \ref{intro1}.



\subsection*{Organization of this paper} Our article is arranged into four parts: a preliminary section on topological bundles, a section on compact group bundles and their representations, a section on structured extensions in topological dynamics over Stonean spaces, and finally a section on structured extensions in ergodic theory. We briefly give an outline of these different parts.

In Section \ref{topb1} we first discuss some basic concepts concerning topological bundles, i.e., continuous surjections $p \colon \Omega \rightarrow M$ from a topological space $\Omega$ (usually carrying additional algebraic structure) to a topological (usually compact or locally compact) space $M$ (see Section \ref{basicconc}). To study structured extensions in ergodic theory and topological dynamics, we then \enquote{relativize} notions of continuous functions (Subsection \ref{contfib}) and probability measures (see Section \ref{secrelmeas}). Finally, the important case of so-called Banach bundles (i.e., a relative notion of a Banach space) is considered in Subsection \ref{bbundl}. This first section is both an introduction to topological bundles and a toolbox for subsequent sections.

Section \ref{seccompgrp} is devoted to an extensive study of compact group bundles and their representation theory. We first recall the concepts of compact group bundles and their associated homogeneous spaces (Subsection \ref{compgroupbundles}) and the induced C*-algebras (see Section \ref{staralg}). We then develop the representation theory of compact group bundles involving the definition of their duals and a study of their representative functions in Subsections \ref{dual}, \ref{posdefsec}, and \ref{dualtop}. This culminates in Peter--Weyl type theorems for compact group bundles and bundles of homogeneous spaces (see Sections \ref{secpw} and \ref{secpw2}).

The results of the second part are applied in Section \ref{sectop} to structured extensions in topological dynamics. We recall basic concepts and set up the notation in Subsection \ref{sectopdyn}. Relative skew-product extensions are defined in Subsection \ref{relskew} and then shown to have relative discrete spectrum (see Section \ref{secStone}). The main result of the second part is a converse result: A topological Mackey--Zimmer representation theorem for structured extensions over Stonean spaces (see Section \ref{topmz}). This requires the concepts of uniform enveloping semigroupoids and relatively invariant measures, which are treated in Subsections \ref{uniformenv} and \ref{RelInvMeas}, respectively.

Finally, in Section \ref{measure} the topological results of the fourth section are transferred to ergodic theory. To do so, we recall in Subsection \ref{measpres} (among other basic concepts for measure-preserving systems) the important notion of topological models. In Subsection \ref{measdiscr} we establish that every topological skew-product extension gives rise to a measure-preserving extension with relative discrete spectrum. To show our main result, representing every structured ergodic extension as a relative skew-product, we first construct a canonical model of an extension in Subsection \ref{canmod}. We can then apply the topological results of the previous part and obtain the Mackey--Zimmer-type representation theorem in Subsection \ref{mainres}.

\subsection{Convention and notation} We assume that all compact and locally compact spaces are Hausdorff and all vector spaces are complex. Moreover, we write $\mathscr{L}(E,F)$ for the space of all bounded linear operators between normed spaces $E$ and $F$ and abbreviate $\mathscr{L}(E) \coloneqq \mathscr{L}(E,E)$ and $E' \coloneqq \mathscr{L}(E, \mathbb{C})$. 

\subsection*{Acknowledgements} The authors are grateful to Markus Haase, Rainer Nagel, and Jean Renault for inspiring discussions. They also are thankful for many helpful suggestions by the anonymous reviewers of an earlier version of this article. Asgar Jamneshan acknowledges funding by the Deutsche Forschungsgemeinschaft (DFG, German Research Foundation) - 547294463. 
Henrik Kreidler acknowledges the financial support from the DFG (project number 451698284). Some of the work was conceived during his time as a scholarship holder of the Friedrich-Ebert-Stiftung and he also expresses his gratitude towards the FES. 

\section{Topological bundles}\label{topb1}

\subsection{Basic concepts}\label{basicconc}
We start by recalling the concept of topological bundle: A \emph{topological bundle} is a continuous surjection $p \colon \Omega \rightarrow M$ between topological spaces $\Omega$ and $M$, called the \emph{total space} and \emph{base space}, respectively. The base space $M$ will often be locally compact. We sometimes refer to $\Omega$ as the bundle and suppress the map $p$ if it is understood from the context. Usually, we consider \emph{open} bundles, i.e., the map $p$ is open. Moreover, we write $\Omega_{m} \coloneqq p^{-1}(\{m\})$ for the \emph{fiber} over a point $m \in M$.

With this notation, we state the following  characterization of openness in terms of an approximation property. (see \cite[Theorems 17.7 and 17.19]{AliBor}). This will be used at many occasions throughout the article. 

\begin{lemma}\label{openchar}
    For a continuous surjection $p \colon \Omega \rightarrow M$ between topological spaces the following assertions are equivalent.
            \begin{enumerate}[(a)]
                \item $p$ is open.
                \item Whenever $(m_\alpha)_{\alpha \in A}$ is a net in $M$ converging to some $m \in M$ and $x \in \Omega_m$ is in the fiber over $m$, there is a subnet of $(m_\beta)_{\beta \in B}$ of $(m_\alpha)_{\alpha \in A}$ and elements $x_\beta \in \Omega_{m_\beta}$ such that the net $(x_\beta)_{\beta \in B}$ converges to $x$.
            \end{enumerate}
\end{lemma}

A \emph{(global) continuous section} of $p\colon \Omega \to M$ is a continuous map $\sigma \colon M \rightarrow \Omega$ such that $p \circ \sigma = \mathrm{id}_{M}$. If $O \subset M$ is any open subset, then a continuous section for the restricted bundle $p|_{p^{-1}(O)}\colon p^{-1}(O) \rightarrow O$ is called a \emph{local continuous section} of the original bundle $p \colon \Omega \rightarrow M$. Suppressing $p$ we write $\Gamma(E)$ for the \emph{set of continuous sections} and $\Gamma_{\mathrm{loc}}(E)$ for the \emph{set of all local continuous sections} of $p$.

Finally, for two bundles $p \colon \Omega \rightarrow M$ and $p' \colon \Omega' \rightarrow M$, we write 
    \begin{align*} 
        \Omega \times_{p,p'} \Omega' \coloneqq \{(x,x') \in \Omega \times \Omega'\mid p(x) = p'(x')\} \subset \Omega \times \Omega'
    \end{align*}
for the corresponding \emph{fiber product} which is itself a bundle over $M$ in a canonical way. We also use the notation $\Omega \times_{M} \Omega'$ for the fiber product if the maps $p$ and $p'$ are understood from the context.

\subsection{Continuous fiber maps}\label{contfib}
It is convenient to work with a space of continuous maps reflecting the fibered structure of a topological bundle. We recall the following construction from \cite[Definition 2.1]{EdKr2021}.

\begin{definition}\label{defcontfib}
    Let $p \colon \Omega \rightarrow M$ and $p' \colon \Omega' \rightarrow M'$ be continuous surjections from topological spaces $\Omega$ and $\Omega'$ to compact spaces $M$ and $M'$. The set
        \begin{align*}
		    \mathrm{C}_{p}^{p'}(\Omega,\Omega') \coloneqq \bigcup_{m \in M,\, m' \in M'} \mathrm{C}(\Omega_m, \Omega'_{m'})
	    \end{align*}
    is the \emph{space of continuous fiber maps}. We define $s(\vartheta) \coloneqq m$ and $r(\vartheta) \coloneqq m'$ for $\vartheta \in \mathrm{C}(\Omega_m,\Omega'_{m'}) \subset \mathrm{C}_p^{p'}(\Omega,\Omega')$, where $m \in M$, $m' \in M'$, and call $s\colon \mathrm{C}_{p}^{p'}(\Omega,\Omega') \rightarrow M$ and $r\colon \mathrm{C}_{p}^{p'}(\Omega,\Omega') \rightarrow M'$ the \emph{source map} and \emph{range map}, respectively.
    If $\Omega' = \C$ and $M' = \{\mathrm{pt}\}$ is a singleton set, we abbreviate $\mathrm{C}_p(\Omega) \coloneqq \mathrm{C}_p^{p'}(\Omega,\Omega')$.
\end{definition}

In analogy to the space of continuous maps between two topological spaces, spaces of continuous fiber maps can be equipped with a canonical topology (see  \cite[Definition 2.3]{EdKr2021}).

\begin{definition}
    For continuous surjections $p \colon \Omega \rightarrow M$ and $p' \colon \Omega'  \rightarrow M'$ from topological spaces $\Omega$ and $\Omega'$ to compact spaces $M$ and $M'$, we call the topology on $\mathrm{C}_{p}^{p'}(\Omega,\Omega')$ generated by the sets 
	\begin{align*}
		\mathrm{W}(C,O) \coloneqq \{\vartheta \in \mathrm{C}_p^{p'}(\Omega,\Omega') \mid \vartheta(C \cap \Omega_{s(\vartheta)}) \subset O \cap \Omega'_{r(\vartheta)} \}
	\end{align*}
where $C \subset \Omega$ is compact and $O \subset \Omega'$ is open the \emph{compact-open topology} on $\mathrm{C}_{p}^{p'}(\Omega,\Omega')$.
\end{definition}
From now on, we equip any space of continuous fiber maps with the compact-open topology. We recall the following characterization of convergence with respect to the compact-open topology from \cite[Proposition 2.4]{EdKr2021}.\footnote{In the statement of \cite[Proposition 2.4]{EdKr2021} only local compactness of $\Omega$ is assumed, but the proof uses compactness.}
	\begin{proposition}\label{descriptiontop}
	    Let $p \colon \Omega \rightarrow M$ and $p' \colon \Omega' \rightarrow M'$ be continuous surjections from topological spaces $\Omega$ and $\Omega'$ to compact spaces $M$ and $M'$. Assume that $\Omega$ is compact. For a net $(\vartheta_\alpha)_{\alpha \in A}$ in $\mathrm{C}_{p}^{p'}(\Omega,\Omega')$ and an element $\vartheta \in \mathrm{C}_{p}^{p'}(\Omega,\Omega')$ the following assertions are equivalent.
			\begin{enumerate}[(a)]
				\item The net $(\vartheta_\alpha)_{\alpha \in A}$ converges to $\vartheta$.
				\item The following two conditions are satisfied.
					\begin{enumerate}[(i)]
						\item $\lim_{\alpha} s(\vartheta_\alpha) = s(\vartheta)$.
						\item Whenever $(\vartheta_\beta)_{\beta \in B}$ is a subnet of $(\vartheta_\alpha)_{\alpha \in A}$, and $(x_\beta)_{\beta \in B}$ is a net in $\Omega$ with $p(x_\beta) = s(\vartheta_\beta)$ for every $\beta \in B$ converging to some $x \in \Omega$, then $\lim_\beta \vartheta_\beta(x_\beta) = \vartheta(x)$.
					\end{enumerate}
			\end{enumerate}
	\end{proposition}
  
    We give a simple example to illustrate this kind of convergence.
    \begin{example}\label{simpleex}
        Let $a \in [0,1]$ be fixed and consider the closed subset $\Omega = \{0\} \times [0,1]  \cup (0,1]  \times [0,a]$ of the unit square $[0,1]^2$. Projecting onto the first component yields a continuous surjection $p \colon \Omega \rightarrow M$, where $M = [0,1]$. The zero functions $0_m \colon \{m\} \times [0,a] \rightarrow \C, \, (m,x) \mapsto 0$ for $m\in(0,1]$ converge for $m \rightarrow 0$ in the space $\mathrm{C}_p(\Omega)$ to any continuous function $f \colon \{0\} \times [0,1] \rightarrow \C$ which vanishes on $\{0\} \times [0,a]$. In particular, the limit is unique (and then given by the zero function on $\{0\} \times [0,1]$) if and only if $a = 1$. This is precisely the case when $p$ is an open map.
    \end{example}
      We obtain the following important corollary (cf. \cite[Remark 2.5]{EdKr2021}) which directly follows from Proposition \ref{descriptiontop} via Lemma \ref{openchar}.
    \begin{corollary}
         Let $p \colon \Omega \rightarrow M$ and $p' \colon \Omega' \rightarrow M'$ be continuous surjections from topological spaces $\Omega$ and $\Omega'$ to compact spaces $M$ and $M'$. Assume that $\Omega$ is compact and $p$ is open, then $\mathrm{C}_{p}^{p'}(\Omega,\Omega')$ is a Hausdorff space.
    \end{corollary}

\subsection{Fiber measures and relative measures}\label{secrelmeas}
We need a \emph{relative} version of measure theory for bundles for the study of extensions of dynamical systems. 
We denote by $\mathrm{P}(\Omega)$ the space of regular Borel probability measures on a compact space $\Omega$. By the Riesz representation theorem (see, e.g.,  \cite[Appendix E]{EFHN} or \cite[Section 5]{jt-foundational}) the map
        \begin{align*}
            \mathrm{P}(\Omega) \rightarrow \mathrm{C}(\Omega)', \quad\mu \mapsto \left[f \mapsto \int_\Omega f\, \mathrm{d}\mu\right] 
        \end{align*}
is an affine bijection onto a convex subset of the dual space $\mathrm{C}(\Omega)'$ which is compact with respect to the weak$^*$-topology. We will often identify $\mathrm{P}(\Omega)$ with its image in $\mathrm{C}(\Omega)'$ and write $\mu(f) \coloneqq \int_\Omega f \, \mathrm{d}\mu$ for $f \in \mathrm{C}(\Omega)$ and $\mu \in \mathrm{P}(\Omega)$.
\begin{definition}
    Let $p \colon \Omega \rightarrow M$ be a continuous surjection between compact spaces. A measure $\mu \in \mathrm{P}(\Omega)$ is called a \emph{fiber measure} if its pushforward measure $p_*\mu$ with respect to $p$ is a point measure on $M$. We write $\mathrm{P}_p(\Omega) \subset \mathrm{P}(\Omega)$ for the weak* closed subspace of all \emph{fiber measures}.
\end{definition}
\begin{remark} 
    An equivalent condition for $\mu \in \mathrm{P}(\Omega)$ to be a fiber measure is that $\mu$ is supported in a fiber $\Omega_m$ for some $m \in M$, see the proof of \cite[Lemma V.3.20]{deVr1993} or \cite[Lemma 2.16]{edeko}.
\end{remark}
The push-forward of measures then defines a continuous surjection $p_* \colon \mathrm{P}_p(\Omega) \rightarrow M$, therefore $\mathrm{P}_p(\Omega)$ is a \enquote{topological bundle} over $M$.\medskip\\
The following result is a straightforward generalization of \cite[Lemma 5.5]{EdKr2021} and will be useful later.
	
\begin{proposition}\label{weakstaconvmeas}
    Let $p \colon \Omega \rightarrow M$ be a continuous surjection between compact spaces and let $(m_{\alpha})_{\alpha \in A}$ a net in $M$ converging to some $m \in M$. Assume that $\mu_{\alpha} \in \mathrm{C}(\Omega_{m_{\alpha}})'$ for each $\alpha \in A$ with $\limsup_{\alpha \in A} \|\mu_{\alpha}\| < \infty$. Then the following assertions are equivalent.
        \begin{enumerate}[(a)]
            \item $\lim_{\alpha} \mu_{\alpha}(f|_{\Omega_{m_{\alpha}}}) = \mu(f|_{\Omega_m})$ for each $f \in \mathrm{C}(\Omega)$.
            \item $\lim_{\alpha} \mu_{\alpha}(f_\alpha) = \mu(f)$ whenever $(f_{\alpha})_{\alpha \in A}$ is a net in $\mathrm{C}_p(\Omega)$ with $s(f_{\alpha}) = m_{\alpha}$ for each $\alpha \in A$ converging to some $f \in \mathrm{C}_p(\Omega)$.
        \end{enumerate}
\end{proposition}

\begin{proof}
    The implication \enquote{(b) $\Rightarrow$ (a)} is easy. For the converse implication  \enquote{(a) $\Rightarrow$ (b)} let $(f_{\alpha})_{\alpha \in A}$ be a net in $\mathrm{C}_p(\Omega)$ such that $s(f_{\alpha}) = m_{\alpha}$ for each $\alpha \in A$ converging to some $f \in \mathrm{C}_p(\Omega)$. Choose $F \in \mathrm{C}(\Omega)$ with $F|_{\Omega_m} = f$. For every $\alpha \in A$ there is $x_{\alpha} \in \Omega_{m_{\alpha}}$ with 
       \begin{align*}
           |f_{\alpha}(x_{\alpha}) - F(x_{\alpha})| = \max_{y \in \Omega_{m_{\alpha}}} |f_{\alpha}(y) - F(y)|.
        \end{align*}
   Since it suffices to show that every subnet of $(\mu_{\alpha}(f_{\alpha}))_{\alpha \in A}$ has a subnet converging to $\mu(f)$, we may assume by compactness of $\Omega$ that the limit $x \coloneqq \lim_{\alpha} x_{\alpha} \in \Omega_{m}$ exists. Hence, 
        \begin{align*}
            \lim_{\alpha} \max_{x \in \Omega_{m_{\alpha}}} |f_{\alpha}(y) - F(y)| = \lim_{\alpha} |f_{\alpha}(x_{\alpha}) - F(x_{\alpha})| = |f(x) - F(x)| = 0.
        \end{align*}
    Since $\limsup_{\alpha} \|\mu_{\alpha}\| < \infty$, we conclude that
        \begin{align*}
            \lim_{i} \mu_{\alpha}(f_{\alpha} -F|_{\Omega_{m_\alpha}}) = 0.
        \end{align*}
    On the other hand, $\lim_{\alpha} \mu_{\alpha}(F|_{\Omega_{m_\alpha}}) = \mu(F) = \mu(f)$. In combination we obtain $\lim_{\alpha} \mu_{\alpha}(f_{\alpha}) = \mu(f)$ as desired.
\end{proof}

In particular, we obtain the following consequence of Proposition \ref{weakstaconvmeas}.  Recall that here, $\mathrm{C}_p(\Omega)$ is a bundle over $M$ via the source map $s$ (cf. Definition \ref{defcontfib}) and hence we can form the fiber product $\mathrm{P}_p(\Omega) \times_{p^*,s} \mathrm{C}_p(\Omega)$ (see Section \ref{basicconc}).

\begin{corollary}\label{convlemma}
    Let $p \colon \Omega \rightarrow M$ be a continuous surjection between compact spaces. Then the map
        \begin{align*}
            \mathrm{P}_p(\Omega) \times_{p^*,s} \mathrm{C}_p(\Omega) \rightarrow \C, \quad (\mu,f) \mapsto \mu(f)
        \end{align*}
    is continuous.
\end{corollary}

We also need the following concept of relative measure. This is the non-dynamical version of the well-known concept of relatively invariant measure, see, e.g., \cite{Glas1975} and Subsection \ref{RelInvMeas} below.

\begin{definition}
    Let $p \colon \Omega \rightarrow M$ be a continuous surjection between compact spaces.  A \emph{relative measure} is a continuous section $\nu \colon M \rightarrow \mathrm{P}_p(\Omega), \, m \mapsto \nu_m$ of $p_* \colon \mathrm{P}_p(\Omega) \rightarrow M$. It is \emph{fully supported} if $\mathrm{supp}\, \nu_m = \Omega_m$ for every $m \in M$. We denote the \emph{space of all relative measures} of $p$ by $\mathrm{RM}(p)$.
\end{definition}

We again provide a simple example.

\begin{example}
    In the case $a = 1$ of Example \ref{simpleex}, the Lebesgue measure on the unit interval induces a fiber measure on every fiber $\Omega_m = \{m\} \times [0,1]$ for $m \in M = [0,1]$ and these form a relative measure for the continuous surjection $p\colon \Omega \rightarrow M$.
\end{example}

Note that, in general, relative measures do not need to exist (consider, e.g., the map $p \colon [0,1] \rightarrow [0,1]/\sim$ identifying the two endpoints $0$ and $1$).\medskip\\

We note the following lemma for later use.

\begin{lemma}\label{lemmaforlowersemi}
	Let $p \colon \Omega \rightarrow M$ be a continuous surjection between compact spaces and $\nu \colon M \rightarrow \mathrm{P}_{p}(\Omega), \, m \mapsto \nu_m$ a fully supported relative measure. Let further $r \in [1,\infty)$ and $m \in M$. Then every $g \in \mathrm{C}(\Omega_m)$ has an extension $h \in \mathrm{C}(\Omega)$ with 
        \begin{align*}
            \max_{m' \in M} \|h|_{\Omega_{m'}}\|_{\mathrm{L}^r(\Omega_{m'},\nu_{m'})} = \|g\|_{\mathrm{L}^r(\Omega_m,\nu_m)}.
        \end{align*}
\end{lemma}
\begin{proof}
    We may assume that $c \coloneqq \|g\|_{\mathrm{L}^r(\Omega_m,\nu_m)} \neq 0$. First, pick any continuous extension $g \in \mathrm{C}(\Omega)$. By the continuity of the relative measure, we find some neighborhood $O$ of $m$ with $\|g\|_{\mathrm{L}^r(\Omega_{m'},\nu_{m'})} > \frac{c}{2}$ for all $m' \in O$. Let $f \in \mathrm{C}(M)$ with $f(m) = 1$, $0 \leq f \leq \mathbbm{1}$, and support in $O$. Then
        \begin{align*}
            h(x) \coloneqq \begin{cases}
                f(p(x)) \frac{g(x)}{\|g\|_{\mathrm{L}^r(\Omega_{p(x)},\nu_{p(x)})}}, & \textrm{ if } x \in p^{-1}(O),\\
                0 & \textrm{ else},
                \end{cases}
        \end{align*}
    defines an extension $h \in \mathrm{C}(\Omega)$ with the desired property.
\end{proof}

The next result establishes a connection between relative and ordinary measures.
Given a continuous surjection $p \colon \Omega \rightarrow M$ between compact spaces, relative measures for $p$ allow to \enquote{lift} measures $\mu_M \in \mathrm{P}(M)$ to measures $\mu_\Omega$ on $\Omega$ with $p_*\mu_\Omega = \mu_M$. More is true if $\mu_M$ is \emph{normal} meaning that the canonical map $\mathrm{C}(M) \rightarrow \mathrm{L}^\infty(M,\mu_M)$ is a bijection (hence $M$ is isomorphic to the Stone space of the measure space $(M,\mu_M)$, and in particular, a Stonean, i.e., an extremally disconnected and compact, space, see \cite[Section 9]{jt-foundational} and \cite[Section 12.4]{EFHN}). In this case, any measure $\mu_\Omega \in \mathrm{P}(\Omega)$ with $p_*\mu_\Omega = \mu_M$ can be disintegrated to obtain a relative measure (cf. \cite[Section 2]{ellis} and \cite[Section 8]{jt20}).
\begin{proposition}\label{disintegration}
	Let $p \colon \Omega \rightarrow M$ be a continuous surjection between compact spaces and $\mu_M \in \mathrm{P}(M)$ fully supported. The canonical map
				\begin{align*}
					\int_M \mathrm{d}\mu_M \colon \mathrm{RM}(p) \rightarrow \{\mu_\Omega \in \mathrm{P}(\Omega)\mid p_*\mu_\Omega = \mu_M\}, \quad \nu \mapsto \int_M \nu\,\mathrm{d}\mu_M
				\end{align*}
			with $(\int_M \nu\,\mathrm{d}\mu_M)(f) \coloneqq \int_M \nu_m(f) \, \mathrm{d}\mu_M(m)$ for $f \in \mathrm{C}(\Omega)$ and $\nu \in \mathrm{RM}(p)$ is
			\begin{enumerate}[(i)]
				\item injective, and\label{disintegration1}
				\item bijective if $\mu_M$ is normal.\label{disintegration2}
			\end{enumerate}
		\end{proposition}
		\begin{proof}
		    It is clear that the map $\int_M \mathrm{d}\mu_M$ is well-defined.
			For (\ref{disintegration1}) take $\nu,\varrho \in \mathrm{RM}(q) $ with $\int_M \nu \,\mathrm{d}\mu_M = \int_L \varrho \,\mathrm{d}\mu_M$ and $f \in \mathrm{C}(\Omega)$. Then
				\begin{align*}
					\int_M g(m) \nu_m(f)\, \mathrm{d}\mu_M(m) &= \int_M \nu_m((g \circ p) \cdot f) \,\mathrm{d}\mu_M(m) = \int_M \varrho_m((g \circ p) \cdot f) \,\mathrm{d}\mu_M(m)\\
					&=\int_M g(m) \varrho_m(f)\, \mathrm{d}\mu_M(m)
				\end{align*}
			for every $g \in \mathrm{C}(M)$. This implies $\nu_m(f) = \varrho_m(f)$ for almost every $m \in M$. Since $\mu_M$ has full support, this equality even holds for every $m \in M$. We conclude that $\nu_m = \varrho_m$ for every $m \in M$ and thus $\nu = \varrho$.
			
			For (\ref{disintegration2}) we assume that $\mu_M$ is normal and take a measure $\mu_\Omega \in \mathrm{P}(\Omega)$ with $p_*\mu_\Omega = \mu_M$. Consider the corresponding conditional expectation operator $\mathbb{E}_{(M,\mu_M)}\colon \mathrm{L}^\infty(\Omega,\mu_\Omega) \rightarrow \mathrm{L}^\infty(M,\mu_M)$.
			Composing this with the inverse of the canonical isomorphism $J \colon \mathrm{C}(M) \rightarrow \mathrm{L}^\infty(M,\mu_M)$ we obtain a bounded linear map $J^{-1}\mathbb{E}_{(M,\mu_M)} \colon \mathrm{L}^\infty(\Omega,\mu_\Omega) \rightarrow \mathrm{C}(M)$ which induces an operator $P \colon \mathrm{C}(\Omega) \rightarrow \mathrm{C}(M)$. We set $\nu_m \coloneqq \delta_m \circ P \in \mathrm{P}(\Omega)$ for the point measure $\delta_m \in \mathrm{P}(M)$ and $m \in M$. One can readily check that this defines an element $\nu \in \mathrm{RM}(p)$. Moreover,
				\begin{align*}
					\int_\Omega f \,\mathrm{d}\mu_\Omega = \int_{M} Pf(m)\,\mathrm{d}\mu_{M}(m) = \int_M \delta_m(Pf)\, \mathrm{d}\mu_M(m) = \int_M \nu_m(f) \,\mathrm{d}\mu_M(m)
				\end{align*}
			for every $f \in \mathrm{C}(\Omega)$ as desired.
		\end{proof}
        \begin{remark}
          he proof of part Proposition \ref{disintegration} (\ref{disintegration2}) is based on an equivalent description of relative measures via conditional expectation operators, see \cite[Proposition 3.1]{Glas1975}.
        \end{remark}
\subsection{Banach bundles}\label{bbundl}
An important type of bundles are so-called Banach bundles. We recall the definition (see \cite[Definition 1.1]{DuGi1983}, \cite[Section 1 and Theorem 3.2]{Gierz1982} and \cite[Section 4]{EdKr2021}).

\begin{definition}
    A continuous and open surjection $s \colon E \rightarrow M$ from a topological space $E$ to a locally compact space $M$ is a \emph{continuous Banach bundle} over $M$ if
        \begin{enumerate}[(i)]
					\item every fiber $E_m$ for $m \in M$ is a Banach space,
					\item the maps 
						\begin{align*}
							&E \times_M E \rightarrow E, \quad (v,w) \mapsto v + w,\\
							&\C \times E \rightarrow E, \quad (\lambda, v) \mapsto \lambda v\, \textrm{ and}\\
							&E \rightarrow \R_+, \quad v \mapsto \|v\|_{s(x)}
						\end{align*}
						are continuous, and
					\item for every $m \in M$, the sets
					    \begin{align*}
					        V(O,\varepsilon) \coloneqq \{v \in s^{-1}(O)\mid \|v\|_{s(v)}< \varepsilon\}
					    \end{align*}
					    for $\varepsilon > 0$ and neighborhoods $O$ of $m$ in $M$ define a neighborhood base of the zero vector $0_{E_m} \in E_m$.
				\end{enumerate}
	The definition of a \emph{continuous Hilbert bundle} is analogous (every fiber is a Hilbert space and the corresponding inner products on the fibers define a continuous map on the fiber product).
\end{definition}

    The following is the most important example of a continuous Banach bundle in this article.

\begin{example}\label{extensionbundle}
    Let $p \colon \Omega \rightarrow M$ be an open continuous map between compact spaces $\Omega$ and $M$. Then the space of continuous fiber maps $\mathrm{C}_p(\Omega)$ with the compact-open topology defines a continuous Banach bundle $s \colon \mathrm{C}_p(\Omega) \rightarrow M$ over $M$, see \cite[Example 4.5]{EdKr2021}. 
\end{example}

    Continuous Banach bundles $s \colon E \rightarrow M$ always admit many continuous sections: For every $v \in E$ there is $\sigma \in \Gamma(E)$ with $\sigma(s(v)) = v$ (see, e.g., \cite[Theorem 3.2]{Gierz1982}). Moreover, the space of all bounded continuous sections 
        \begin{align*}
            \Gamma_{\mathrm{b}}(E) \coloneqq \left\{\sigma \in \Gamma(E)\mid\sup_{m \in M} \|\sigma(m)\| < \infty\right\}
        \end{align*}
    is canonically a module over the algebra $\mathrm{C}_{\mathrm{b}}(M)$ of bounded continuous functions on $M$ by defining multiplication with functions in $\mathrm{C}_{\mathrm{b}}(M)$ fiberwise. Together with the complete norm defined by $\|\sigma\| \coloneqq \sup_{m \in M} \|\sigma(m)\|_{E_m}$ for $\sigma \in \Gamma_{\mathrm{b}}(E)$ the space $\Gamma_{\mathrm{b}}(E)$ is a Banach module, i.e., $\|f\sigma\|\leq \|f\| \cdot \|\sigma\|$ for all $f \in \mathrm{C}_{\mathrm{b}}(M)$ and $\sigma \in \Gamma_{\mathrm{b}}(E)$, see  \cite[Remark 4.3(i)]{EdKr2021} and \cite[Chapter 2]{DuGi1983}. In fact, there is a correspondence between continuous Banach bundles and certain Banach modules (see, e.g., \cite{HoKe1977} and \cite[Theorem 2.6]{DuGi1983}).
    
    However, here we only need the following Stone-Weierstrass theorem for Banach bundles  (see \cite[Corollary 4.3]{Gierz1982}). For a continuous Banach bundle $s \colon E \rightarrow M$ over a locally compact space $M$ we call a subset $D \subset E$ \emph{fiberwise dense} (resp. \emph{fiberwise total}) if $D \cap E_m$ is dense (resp. total) in $E_m$ for every $m \in M$ (see \cite[Remarks before Theorem 4.8]{EdKr2021}).
    \begin{proposition}\label{swbundles}
        Let $s \colon E \rightarrow M$ be a continuous Banach bundle over a compact space $M$ and $\Gamma' \subset \Gamma(E)$ a $\mathrm{C}(M)$-submodule such that $\{\sigma(m) \mid m \in M, \sigma \in \Gamma\}$ is fiberwise dense in $E$. Then $\Gamma'$ is dense in $\Gamma(E)$.
    \end{proposition}

In the case of Example \ref{extensionbundle}, the space of continuous sections can be represented as follows (see again \cite[Example 4.5]{EdKr2021}). 

\begin{example}\label{extensionbundle2}
    Let $p \colon \Omega \rightarrow M$ be an open continuous map between compact spaces $\Omega$ and $M$. Then the map
        \begin{align*}
            \mathrm{C}(\Omega) \rightarrow \Gamma(\mathrm{C}_p(\Omega)), \quad f \mapsto [m \mapsto f|_{\Omega_m}]
        \end{align*}
    is an isometric $\mathrm{C}(M)$-linear bijection, where $\mathrm{C}(\Omega)$ is equipped with the multiplication $f \cdot g \coloneqq (f \circ p) g$ for $f \in \mathrm{C}(M)$ and $g \in \mathrm{C}(\Omega)$.
\end{example}

We also need the following concept of continuous Banach subbundle.
\begin{definition}\label{subbundledef}
    Let $s \colon E \rightarrow M$ be a continuous Banach bundle over a locally compact space $M$. A restriction $s|_F \colon F \rightarrow M$ to a subset $F \subset E$ is a \emph{continuous Banach subbundle} if $s|_F$ is still open and $F_m \coloneqq E_m \cap F$ is a closed linear subspace of $E_m$ for every $m \in M$. 
    
    Moreover, if $O \subset M$ is an open subset, then we call a continuous Banach subbundle of the restriction $s|_{s^{-1}(O)} \colon s^{-1}(O)  \subset E \rightarrow O$ a \emph{local continuous Banach subbundle of $s$ (over $O$)}.
\end{definition}
A continuous Banach subbundle of a continuous Banach bundle is then indeed a Banach bundle in its own right (see \cite[Section 8]{Gierz1982}).

We can also extend a local continuous Banach subbundle over an open subset to a continuous subbundle over the whole base space:

\begin{lemma}\label{remsub}
    Consider a continuous Banach bundle $s \colon E \rightarrow M$ over a locally compact space $M$ and a local continuous Banach subbundle $s|_F \colon F \rightarrow O$ of $s$ over an open subset $O \subseteq M$. Set
        \begin{align*}
            F_+ \coloneqq F \cup \bigcup_{m \in M \setminus O} \{0_m\} \subseteq E.
        \end{align*}
    Then $s|_{F_+} \colon F_+ \rightarrow M$ is a continuous Banach subbundle of $s$.
\end{lemma}
\begin{proof}
    The only non-trivial part is to show that $s|_{F_+} \colon F_+ \rightarrow M$ is an open map. So take $v \in F_+$ and $V \subseteq E$ a neighborhood of $v \in E$. We show that $s(V \cap F_+)$ is then a neighborhood of $s(v)$ in $M$.
    
    If $v \in F$, then, since $s|_F \colon F \rightarrow O$ is a localy continuous Banach subbundle, the set $s(F \cap V)$ is a neighborhood of $s(v)$ in $O$, and hence in $M$. Since $s(F \cap V) \subseteq s(F_+ \cap V)$, $s(F_+ \cap V)$ is a neighborhood of $s(v)$ as desired.
    
    If $v \notin F$, then we find some $m \in M \setminus O$ with $v = 0_m$. By the definition of a Banach bundle, we may therefore assume that $V = V(O',\varepsilon) = \{w \in s^{-1}(O') \mid \|w\|_{s(w)} < \varepsilon\}$ for some neighborhood $O'$ of $m = s(v)$ and $\varepsilon > 0$. But then $s(F_+ \cap V) = O'$.
\end{proof}

 
In the remainder of this section we will deal with Banach and Hilbert bundles which are locally trivial. Recall first the concept of locally trivial vector bundle (see, e.g., \cite[Chapter I]{Karo1978}, \cite[Chapter 3]{Huse1994}).

\begin{definition}\label{defloctrivvb}
    A continuous surjection $s \colon E \rightarrow M$ from a topological space $E$ to a locally compact space $M$ is a \emph{locally trivial vector bundle} over $M$ if every fiber $E_m$ for $m \in M$ is a vector space and for every point $m \in M$ there are a compact neighborhood $C \subset M$ of $m$, some $d \in \N_0$ and a homeomorphism $\Phi \colon s^{-1}(C) \rightarrow C \times \C^d$ such that 
        \begin{enumerate}[(i)]
            \item with the projection $\mathrm{pr}_1 \colon C \times \C^d \rightarrow C$ onto the first component, the diagram
                	\[
		                \xymatrix{
			                s^{-1}(C) \ar[rr]^{\Phi}  \ar[dr]_s & & C \times \C^d \ar[ld]^{\mathrm{pr}_1}\\
			                 & C & 
		                }
	                \]	
	           commutes, and
            \item the restriction $\Phi|_{E_{m'}} \colon E_{m'} \rightarrow \{m'\} \times \C^d$ is a vector space isomorphism for every $m' \in C$.
        \end{enumerate}
\end{definition}

Notice that here the condition of \enquote{local triviality} automatically implies that the map $s \colon E \rightarrow M$ is open.

We now define locally trivial Banach bundles.

\begin{definition}
    A continuous Banach bundle $s \colon E \rightarrow M$ over a locally compact space $M$ is a \emph{locally trivial Banach bundle} if $s$ (with the given vector space structure on the fibers) is a locally trivial vector bundle in the sense of Definition \ref{defloctrivvb}.
\end{definition}

The next result is a consequence of \cite[Proposition 10.9]{Gierz1982} and yields an automatic compatibility with the norm. It shows that our definition of locally trivial continuous Banach bundles is equivalent to that of  \cite[Section 17]{Gierz1982} and \cite[Definition 4.2]{EdKr2021}). 

\begin{lemma}
       Let $s \colon E \rightarrow M$ be a locally trivial Banach bundle over a locally compact space $M$ and  $\Phi \colon s^{-1}(C) \rightarrow C \times \C^d$ a \enquote{trivialization} as in Definition \ref{defloctrivvb}. Then there are constants $c_1, c_2 >0$ such that 
	            \begin{align*}
	                c_1 \|v\| \leq \|\mathrm{pr}_2(\Phi(v))\| \leq c_2 \|v\| \textrm{ for } v \in s^{-1}(C)
	            \end{align*}
	   , where $\mathrm{pr}_2 \colon C \times \C^d \rightarrow \C^d$  is the projection onto the second component and $\C^d$ carries the Euclidean norm.
\end{lemma}

The following Serre--Swan-type duality result (see \cite[Proposition 4.13]{EdKr2021}) shows that local triviality of a Banach bundle over a compact space is reflected nicely by its space of continuous sections (cf. \cite{Swan1962}). Recall here that a module $\Gamma$ over a ring $R$ is \emph{projective} if there is an $R$-module $\Lambda$ such that $\Gamma \oplus \Lambda$ is a free $R$-module.
    \begin{proposition}\label{SerreSwan}
    For a continuous Banach bundle $s \colon E \rightarrow M$ over a compact space $M$ the following assertions are equivalent.
        \begin{enumerate}[(a)]
            \item $s \colon E \rightarrow M$ is locally trivial.
            \item $\Gamma(E)$ is finitely generated and projective.
        \end{enumerate}
    \end{proposition}
    
    Obviously, every locally trivial Banach or Hilbert bundle defines, by forgetting the norms or the inner products, a locally trivial vector bundle in the usual sense. We have the following well-known converse result, see \cite[Lemma 2]{Swan1962}, showing that every locally trivial vector bundle over a compact base space can be turned into a locally trivial continuous Hilbert bundle (and in particular, a locally trivial Banach bundle).
        \begin{proposition}\label{existinnerprod}
            If $s \colon E \rightarrow M$ is a locally trivial vector bundle over a compact space, then there is a continuous map $(\cdot|\cdot) \colon E \times_M E \rightarrow \C$ such that the restriction $(\cdot|\cdot)_{E_m \times E_m}$ is an inner product on $E_m$ for every $m \in M$.
        \end{proposition}
    A map as in Proposition \ref{existinnerprod} is called a \emph{continuous inner product} of the locally trivial bundle.\medskip\\
    The following basic result on locally trivial Hilbert bundles will be of use later on (cf. \cite[Proof of Theorem 5.12]{ellis}).
    \begin{proposition}\label{denselydeforth}
        Let $s \colon E \rightarrow M$ be a locally trivial continuous Hilbert bundle over a locally compact space $M$ and $d \in \N_0$ with $\dim E_m = d$ for all $m \in M$. Then there are local continuous sections $\tau_1, \dots, \tau_d \colon O \rightarrow E$ defined on a dense open subset $O$ of $M$ such that $\{\tau_1(m), \dots, \tau_d(m)\}$ is an orthonormal basis of $E_m$ for every $m \in O$.
    \end{proposition}
    \begin{proof}
      For each $m \in M$, we find an open neighborhood and local continuous sections $\tau_1, \dots, \tau_d$ such that $\{\tau_1(m), \dots, \tau_d(m)\}$ is an orthonormal basis of $E_m$. The result now follows from Zorn's lemma.
    \end{proof}
    We call a collection of sections $\tau_1, \dots, \tau_d$ as in Proposition \ref{denselydeforth} a \emph{densely defined local orthonormal basis} of the bundle.\medskip\\
    As a consequence of \cite[Theorem 18.5]{Gierz1982} we have the following characterization of locally trivial continuous Banach bundles. 
    \begin{proposition}\label{locconst}
        For a continuous Banach bundle $s \colon E \rightarrow M$ over a locally compact space $M$ the following assertions are equivalent.
        \begin{enumerate}[(a)]
            \item $s \colon E \rightarrow M$ is locally trivial.
            \item All fibers $E_m$ for $m \in M$ are finite-dimensional and the dimension map
                \begin{align*}
                    \dim \colon M \rightarrow \mathbb{N}_0, \quad m \mapsto \dim(E_m)
                \end{align*}
                is locally constant.
        \end{enumerate}
    \end{proposition}

As a corollary, we obtain the following interesting and helpful characterization of locally trivial subbundles.
\begin{corollary}\label{charloctriv}
    Let $s \colon E \rightarrow M$ be a locally trivial vector bundle over a locally compact space $M$, and $F \subset E$ a subset such that $F \cap E_m$ is a linear subspace of $E_m$ for every $m \in M$. Then the restriction $s|_{F} \colon F \rightarrow M$ defines a locally trivial vector bundle if and only if the following two conditions are satisfied.
        \begin{enumerate}[(i)]
            \item $F$ is a closed subset of $E$.
            \item The restriction $s|_F \colon F \rightarrow M$ is open.
        \end{enumerate}
\end{corollary}
\begin{proof}
    Since all the properties can be checked locally on $M$, we may assume that $M$ is compact. Using Proposition \ref{existinnerprod} we choose a continuous inner product  $(\cdot| \cdot ) \colon E \times_M E \rightarrow \C$. If $s|_F$ is locally trivial, then condition (ii) is clearly fulfilled and (i) is a consequence of \cite[Lemma 4.16]{EdKr2021}. Now assume conversely that (i) and (ii) hold. Since the dimension of $E$ is locally constant, we may assume that $\sup_{m \in M} \dim(F_m) \leq \sup_{m \in M} \dim(E_m) < \infty$. In view of Proposition \ref{locconst} it now suffices to show that the dimension map $\dim \colon M \rightarrow \mathbb{N}_0, \, m \mapsto \dim(F_m)$ is locally constant. Let $d$ be the maximal dimension among fibers of $F$ and consider the set
        \begin{align*}
            M_{d} \coloneqq \{m \in M \mid \dim F_m = d\}
        \end{align*}
    which is open by \cite[Theorem 18.2]{Gierz1982}. By Proposition \ref{locconst}, $s^{-1}(M_{d}) \rightarrow M_{d}$ is a locally trivial bundle over $M_{d}$. We choose a densely defined local orthonormal basis $\sigma_1, \dots, \sigma_{d} \colon O \rightarrow s^{-1}(M_{d})$. Now pick a point $m \in \overline{M_{d}} \subset M$. We then find a net $(m_\alpha)_{\alpha \in A}$ in $O$ converging to $m$. Since $\|\sigma_k(m_\alpha)\| =1$ for every $\alpha \in A$ and bounded subsets of $E$ are relatively compact in $E$ by \cite[Lemma 4.15]{EdKr2021}, we may assume, by passing to a subnet, that the limit $v_i \coloneqq \lim_{\alpha} \sigma_i(m_\alpha)$ exists in $E$ and hence by (i) in $F$ for every $i \in \{1, \dots, d\}$. But then 
        \begin{align*}
            (v_i|v_j) = \lim_{\alpha}(\sigma_i(m_\alpha)|\sigma_j(m_\alpha)) = \delta_{ij}
        \end{align*}
    for all $i,j \in \{1, \dots, d\}$. Therefore the vectors $v_1, \dots, v_{d}$ form an orthonormal subset of $F_{m}$. We conclude that $\dim F_m = d$. This shows that $M_{d}$ is closed and consequently clopen in $M$. By considering $E$ and $F$ on the complement of $M_{d}$ and iterating the previous steps, we conclude that the dimension of $F$ is in fact locally constant.
\end{proof}

\section{Representations of compact group bundles}\label{seccompgrp}

\subsection{Compact group bundles}\label{compgroupbundles}
To state our version of the Mackey--Zimmer theorem, we need to introduce compact group bundles (cf.~\cite[Definition 2.19]{edeko}).
 \begin{definition}
			A continuous surjection $p \colon H \rightarrow M$ from a compact space $H$ to a compact space $M$ is a \emph{compact group bundle} if
			        \begin{enumerate}[(i)]
					\item every fiber $H_m \coloneqq p^{-1}(m)$ for $m \in M$ is a group, and
					\item the maps
					   \begin{align*}
					        	&H \times_M H \rightarrow H, \quad (x,y) \mapsto xy,\\
					        	&H \rightarrow H, \quad x \mapsto x^{-1}, \textrm{ and }\\
					        	&M \rightarrow H, \quad m \mapsto 1_{H_m}
					   \end{align*}
					   are continuous.
				\end{enumerate}
\end{definition}

Evidently, every compact group can be considered a compact group bundle over a singleton. The following are further basic examples.
\begin{examples}\label{examplescompgroup}
    Let $M$ be any compact space and $H$ a compact group.
    \begin{enumerate}[(i)]
         \item The product $M \times H$ defines an open compact group bundle over $M$ by projecting onto the first component. This is the \emph{trivial group bundle} with fiber $H$.
         \item If $m_0 \in M$, then the quotient of $M \times H$ given by identifying all points $(m_0,x)$ for $x \in H$ defines an open compact group bundle over $M$. The fiber over every point $m \in M \setminus \{m_0\}$ is isomorphic to $H$, whereas the fiber over $m_0$ is given by a trivial group.
         \item If $m_0 \in M$ is an accumulation point of $M$ and the group $H$ is non-trivial, then the subspace $M\setminus\{m_0\} \times \{1_H\} \cup \{m_0\} \times H \subset M \times H$ provides an example of a compact group bundle which is not open.
    \end{enumerate}
\end{examples}

Examples \ref{examplescompgroup} (ii) and (iii) are also examples for the following concepts. 

\begin{definition}
Let $p \colon H \rightarrow M$ be a compact group bundle over a compact space $M$.
    \begin{enumerate}[(i)]
        \item A restriction $p|_U \colon U \rightarrow M$ is a \emph{closed subgroup bundle} of $H$ if $U \subset H$ is a closed subset such that $U_m \coloneqq U \cap H_m$ is a subgroup of $H_m$ for every $m \in M$.
        \item If $p|_U \colon U \rightarrow M$ is a closed subgroup bundle of $p$, then the canonical map $p_{/U} \colon H/U \rightarrow M$ , where $H/U \coloneqq \bigcup_{m \in M} H_m/U_m$ is equipped with the quotient topology is the \emph{homogeneous space bundle} induced by $U$. For $x \in U$ the equivalence class in the quotient space $H/U$ is denoted by $xU$.
    \end{enumerate}
\end{definition}

The following lemma collects some basic facts about compact group bundles.
\begin{lemma}\label{remarkshomspace}
    Let $p \colon H \rightarrow M$ be a compact group bundle over a compact space $M$, and $p|_U \colon U \rightarrow M$ a closed subgroup bundle.
    \begin{enumerate}[(i)]
        \item The quotient space $H/U$ is a Hausdorff space.
        \item For each $m \in M$ the subspace $H_m/U_m \subseteq H/U$ carries the quotient topology with respect to the canonical surjection $H_m \rightarrow H_m/U_m$.
        \item The canonical action $H \times_{M} H/U \rightarrow H/U, \, (x,yU) \mapsto xyU$ is continuous.
        \item If $p\colon H \rightarrow M$ is open, then so is $p_{/U}\colon H/U \rightarrow M$.
        \end{enumerate}
\end{lemma}
\begin{proof}
    Since the equivalence relation defined by the canonical map $q \colon H \rightarrow H/U$ is closed, part (i) follows from \cite[Section I.10.4, Proposition 8]{Bour1995}. Parts (ii) and (iv) are easy to check.
    
    We prove (iii). By the closed graph theorem for compact spaces (see, e.g., \cite[Theorem XI.2.7]{Dug1966}), it suffices to show that the map has a closed graph. So let $((x_{\alpha},y_{\alpha}U))_{\alpha \in A}$ be a net $H \times_{M} H/U$ converging to $(x,yU)$ such that $(x_{\alpha}y_{\alpha}U)_{\alpha \in A}$ converges to $zU$ in $H/U$. We have to show that $zU = xyU$. Using compactness of $H$ and continuity of the quotient map $H \rightarrow H/U$, we may assume, by passing to a subnet, that $\lim_{\alpha} y_{\alpha} = y \in H$. But then $\lim_{\alpha} x_{\alpha}y_{\alpha} = xy$ in $H$ by continuity of multiplication in compact group bundles. We conclude that $\lim_{\alpha} x_{\alpha}y_{\alpha}U = xyU$ in $H/U$ and therefore $z = xyU$ as desired.
\end{proof}

The following special case of a result of Renault (see \cite[Lemma 1.3]{Rena1991}) characterizes when a compact group bundle is open, where we use the fact that a compact group bundle is a special case of a compact groupoid (cf. Section \ref{uniformenv}). Here and in the following the Haar measure of a compact group $H$ is denoted by $\lambda_H$.

\begin{proposition}\label{Renault}
    For a compact group bundle $p \colon H \rightarrow M$ over a compact space $M$ the following assertions are equivalent.
        \begin{enumerate}[(i)]
            \item The compact group bundle $p \colon H \rightarrow M$ is open.
            \item The map $M \rightarrow \mathrm{P}_p(H), \, m \mapsto \lambda_{H_m}$ is continuous, hence a relative measure.
        \end{enumerate}
\end{proposition}

\subsection{Group bundle C*-algebras}\label{staralg}
We now consider the C*-algebra associated with an open compact group bundle (see \cite[Section 3]{MRW1996} and \cite{Goeh2008} in the case of bundles of locally compact abelian groups), which is a special case of the general version for locally compact groupoids with a continuous Haar system, see Definition II.1.12 of \cite{Rena1980}.


\begin{definition}
	        Let $p \colon H \rightarrow M$ be an open compact group bundle over a compact space $M$. The space $\mathrm{C}(H)$ equipped with the 
	multiplication given by
	        	\begin{align*}
		        	(f * g)(x)\coloneqq \int_{H_{p(x)}} f(xy)g(y^{-1}) \, \mathrm{d}\lambda_{H_{p(x)}}(y) \quad \text{for every } x \in H
	        	\end{align*}
        	for $f,g \in \mathrm{C}(H)$ and the involution defined by
	        	\begin{align*}
		        	f^*(x) \coloneqq \overline{f(x^{-1})} \quad \text{for every } x \in H
	        	\end{align*}
        	for $f \in \mathrm{C}(H)$ is the \emph{convolution algebra} of $H$.
\end{definition}

One can check that this is indeed a *-algebra, see \cite[Proposition II.1.1]{Rena1980}. The associated C*-algebra is constructed via  the representations of the convolution algebra (see \cite[Definitions II.1.3 and II.1.5]{Rena1980}).
\begin{definition}\label{defboundedrep}
    Let $p \colon H \to M$ be an open compact group bundle over a compact space $M$. A \emph{(non-degenerate) bounded representation} of $\mathrm{C}(H)$ is a *-algebra homomorphism 
		$\pi \colon \mathrm{C}(H) \to \mathscr{L}(E)$ with $E$ a Hilbert space such that
			\begin{enumerate}[(i)]
				\item $\|\pi(f)\| \leq \sup_{m \in M} \int_{H_m} |f|\, \mathrm{d}\lambda_m$ for every $f \in \mathrm{C}(H)$, and
				\item $E = \overline{\mathrm{lin}} \{\pi(f)v\mid f \in \mathrm{C}(H), v \in E\}$.
			\end{enumerate}
\end{definition}
One can show that there is always a \enquote{faithful family of bounded representations} (see \cite[Proposition II.1.11]{Rena1980}) and this enables the following definition (see \cite[Definition II.1.12]{Rena1980}).
\begin{definition} 
	The completion of $\mathrm{C}(H)$ with respect to the norm $\|\cdot\|$ defined by
			\begin{align*}
			 	\|f\| \coloneqq \sup \{\|\pi(f)\|\mid \pi \textrm{ bounded representation of } \mathrm{C}(H)\} \textrm{ for } f \in \mathrm{C}(H)
			\end{align*}
			is called the \emph{group bundle C*-algebra} associated to an open compact group bundle $p \colon H \to M$ over a compact space $M$. We denote it by $\mathrm{C}^*(H)$.
\end{definition}
In this way, every open compact group bundle gives rise to a C*-algebra. In the case of the trivial bundle over a singleton defined by a compact group, this agrees with the usual group C*-algebra (as defined, e.g., in \cite[Section 13.9]{Dixm1977} or \cite[Section 7.1]{folland-harmonic-analysis}).

\subsection{Representations and the dual space}\label{dual}

The following result relates the irreducible representations of the group bundle C*-algebra to the irreducible unitary representation of its fiber groups (see \cite[Corollary 3.4]{MRW1996} for the case of abelian fiber groups). Recall here that a representation on a Hilbert space $E$ is \emph{irreducible} if $E$ is non-zero and there are no non-trivial closed subspaces of $E$ which are invariant with respect to the representation, see \cite[Sections 2.3 and 13.1]{Dixm1977}. Moreover, unitary representations are assumed to be continuous with respect to the strong operator topology, and vector-valued integrals are understood in the weak sense (see, e.g., \cite[Definition 3.26]{rudinfa}).

\begin{proposition}\label{spectrum}
	Let $p \colon H \rightarrow M$ be an open compact group bundle over a compact space $M$.
	\begin{enumerate}[(i)]
		\item If $m \in M$ and $\varrho \colon H_m \to \mathscr{L}(E)$ is a unitary 
		representation of $H_m$ on a Hilbert space $E$, then there is a unique non-degenerate representation $\pi_\varrho \colon \mathrm{C}^*(H) \rightarrow \mathscr{L}(E)$ with
		\begin{align*}
			\pi_\varrho(f)v = \int_{H_m} f(x) \varrho(x)v  \,\mathrm{d}\lambda_{H_m}(x) \textrm{ for every } v \in E \textrm{ and } f \in \mathrm{C}(H).
		\end{align*}
	    Moreover, $\varrho$ is irreducible if and only if $\pi_\varrho$ is irreducible.\label{spectrum1}
	\item Conversely, for each non-zero irreducible representation
	$\pi \colon \mathrm{C}^*(H) \to \mathscr{L}(E)$ there is a unique $q(\pi) \in M$ 
	and a unique irreducible unitary representation $\varrho_\pi \colon H_{p(\pi)} \to \mathscr{L}(E)$ with $\pi = \pi_{\varrho_\pi}$.\label{spectrum2}
	 \item Two non-zero irreducible representations $\pi_1$ and $\pi_2$ of $\mathrm{C}^*(H)$ are equivalent if and only if $q(\pi_1) = q(\pi_2)$ and the corresponding unitary group representations 
	 $\varrho_{\pi_1}$ and $\varrho_{\pi_2}$ are equivalent.\label{spectrum3}
	 \end{enumerate}
\end{proposition}

Note that this result is well-known for compact (even locally compact) groups, see \cite[Sections 13.3 and 13.9]{Dixm1977}.

\begin{remark}\label{irrednec}
    The assumption of irreducibility in (ii) cannot be dropped as can be seen easily by considering the trivial group bundle $M \times \{1\} \rightarrow M$ over a compact space $M$ with at least two elements.
\end{remark}

For the proof of Proposition \ref{spectrum} we recall the concept of \emph{multiplier algebra} which is a non-commutative analogue of the Stone--\v{C}ech compactification of locally compact spaces. The following is a possible construction (see \cite[Section II.7.3]{Blac2006}, see also \cite{busby}). Given a C*-algebra $\mathcal{A}$, consider the set $\mathcal{M}(\mathcal{A})$ of \emph{double centralizers}, i.e., the set of all pairs $(L,R)$ of linear maps on $\mathcal{A}$ with $fL(g) = R(f)g$ for all $f,g \in \mathcal{A}$. If $(L,R)$ is a double centralizer, then $L$ and $R$ are bounded operators on $\mathcal{A}$ with $\|L\| = \|R\|$. Moreover, $\mathcal{M}(\mathcal{A})$ becomes a unital C*-algebra when equipped with 
    \begin{enumerate}[(i)]
        \item the norm $\|(L,R)\| \coloneqq \|L\|$ for $(L,R) \in \mathcal{M}(\mathcal{A})$,
        \item addition and scalar multiplication defined componentwise,
        \item the multiplication defined by $(L_1,R_1) \cdot (L_2,R_2) \coloneqq (L_2L_1,R_1R_2)$ for $(L_1,R_1)$, $(L_2,R_2) \in \mathcal{M}(\mathcal{A})$, and
        \item the involution given by $(L,R)^* \coloneqq (R^*,L^*)$ for $(L,R) \in \mathcal{M}(\mathcal{A})$, where $L^*f \coloneqq (Lf^*)^*$ and $R^*f \coloneqq (Rf^*)^*$ for $f \in \mathcal{A}$.
    \end{enumerate}
    The canonical map 
	\begin{align*}
		\mathcal{A} \rightarrow \mathcal{M}(\mathcal{A}), \quad f \rightarrow (L_f,R_f)
	\end{align*}
	with $L_fg\coloneqq fg$ and $R_fg\coloneqq gf$ for $f,g \in \mathcal{A}$ is an embedding of the C*-algebra $\mathcal{A}$ onto a closed ideal of $\mathcal{M}(\mathcal{A})$. Moreover, every non-degenerate representation $\pi \colon \mathcal{A} \rightarrow \mathscr{L}(E)$ of $\mathcal{A}$ on a Hilbert space $E$ has a unique extension to a unital representation of $\mathcal{M}(\mathcal{A})$ and this satisfies $\pi((L,R))\pi(f) = \pi(Rf)$ for all $(L,R) \in \mathcal{M}(\mathcal{A})$ and $f \in \mathcal{A}$ (see \cite[Theorem II.7.3.9]{Blac2006}).
	
	Further note that a double centralizer is of the form $(S,S)$ for $S \in \mathscr{L}(A)$ if and only if it is in the center of $\mathcal{M}(\mathcal{A})$ (this can be checked easily by using the density of $\mathcal{A}$ in $\mathcal{M}(A)$ with respect to the strict topology, see, e.g., \cite[Proposition 3.5]{busby}, and \cite[Proposition 2.5]{busby}). Combining this observation with \cite[Proposition II.1.14]{Rena1980} yields the following.

\begin{lemma}\label{multiplier}
	 Let $p \colon H \rightarrow M$ be an open compact group bundle over a compact space $M$. For each $f \in \mathrm{C}(M)$ the map 
		\begin{align*}
			S_f \colon \mathrm{C}(H) \rightarrow \mathrm{C}(H), \quad g \mapsto (f \circ p) \cdot g
		\end{align*}
	has a unique extension to a bounded operator $S_f \in \mathscr{L}(\mathrm{C}^*(H))$. Moreover, the mapping 
		\begin{align*}
			\mathrm{C}(M) \rightarrow \mathcal{Z}(\mathcal{M}(\mathrm{C}^*(H))), \quad f \mapsto (S_f,S_f)
		\end{align*}
	is a unital *-embedding of the commutative unital C*-algebra $\mathrm{C}(M)$ into the center $\mathcal{Z}(\mathcal{M}(\mathrm{C}^*(H)))$ of the multiplier algebra $\mathcal{M}(\mathrm{C}^*(H))$.
\end{lemma}

We are now ready to prove Proposition \ref{spectrum} by using ideas of the proof of \cite[Theorem 2.4]{DuGi1983}.

\begin{proof}[Proof of Proposition \ref{spectrum}]
    For part (\ref{spectrum1}) observe that for $m \in M$ every unitary representation $\varrho \colon H_m \to \mathscr{L}(E)$ on a Hilbert space $E$ gives rise to a unique non-degenerate bounded representation $\pi_\varrho'$ of the convolution algebra $\mathrm{C}(H_m)$ of the fiber group $H_m$ via
        \begin{align*}
            \pi_\varrho'(f)v = \int_{H_m} f(x) \varrho(x)v  \,\mathrm{d}\lambda_{H_m}(x) \textrm{ for every } v \in E \textrm{ and } f \in \mathrm{C}(H_m)
        \end{align*}
    by \cite[Proposition 13.3.1]{Dixm1977}. By composing $\pi_\varrho'$ with the restriction map $\mathrm{C}(H) \rightarrow \mathrm{C}(H_m), \, f \mapsto f|_{H_m}$ we therefore obtain a non-degenerate bounded representation $\pi_{\varrho}$ of $\mathrm{C}(H)$ in the sense of Definition \ref{defboundedrep}. By definition of the group bundle C*-algebra, this representation extends to a representation of $\mathrm{C}^*(H)$. Uniqueness is clear since $\mathrm{C}(H)$ is a dense subspace of $\mathrm{C}^*(H)$. Moreover, by \cite[Remark 13.3.5]{Dixm1977} $\varrho$ is irreducible if and only if $\pi_\varrho'$ is irreducible which is the case if and only if $\pi_\varrho$ is irreducible. \medskip\\
	We now prove (\ref{spectrum2}) and let $\pi \colon \mathrm{C}^*(H) \to \mathscr{L}(E)$ be a non-zero irreducible 
	representation of the group bundle C*-algebra $\mathrm{C}^*(H)$. Then $\pi$ is non-degenerate and therefore has a unique extension to an irreducible unital representation of the mutilipier algebra $\mathcal{M}(\mathrm{C}^*(H))$, which we also denote by $\pi$, and this satisfies $\pi((L,R))\pi(f) = \pi(Rf)$ for every $f \in \mathrm{C}(H)$. We first show that $\pi$ induces a representation of the convolution algebra $\mathrm{C}(H_m)$ of a fiber group $H_m$ for some $m \in M$.
	
	For $f \in \mathrm{C}(M)$ we obtain with Lemma \ref{multiplier} that $\pi((S_f,S_f))$ is in the center of the image $\pi(\mathcal{M}(\mathrm{C}^*(H)))$. But the latter is given by $\{z \mathrm{Id}_{E}\mid z \in \C\} \subset \mathscr{L}(E)$ since the representation $\pi$ is irreducible, see \cite[Section 2.3]{Dixm1977}. We conclude that there is a multiplicative linear functional $\chi \colon \mathrm{C}(M) \rightarrow \C$ with $\pi((S_f,S_f)) = \chi(f) \mathrm{Id}_E$ for all $f \in \mathrm{C}(M)$. As $\pi$ is unital, we have
	    \begin{align*}
	        \pi((S_{\mathbbm{1}},S_{\mathbbm{1}})) = \pi((\mathrm{Id}_{\mathrm{C}^*(H)},\mathrm{Id}_{\mathrm{C}^*(H)})) = \mathrm{Id}_E.
	    \end{align*}
	Thus, $\chi(\mathbbm{1}) = 1$, and hence $\chi$ is a point evaluation (see, e.g., \cite[Lemma 4.10]{EFHN}). Therefore, there is a unique $m \in M$ such that $\pi(S_f,S_f) = f(m)\mathrm{Id}_E$ 
	for every $f \in \mathrm{C}(M)$.
	
	We now show that $\|\pi(g)\| \leq \int_{H_m} |g| \, \mathrm{d}\lambda_{H_m}$ for every $g \in \mathrm{C}(H)$. So take $g \in \mathrm{C}(H)$ and a neighborhood $O$ of $m$ in $M$. We then find
	$f \in \mathrm{C}(M)$ with $0 \leq f(m') \leq 1$ for all $m' \in M$ and support in $O$ such that $f(m)=1$. Since 
	$\|\pi(h)\| \leq \sup_{m' \in M} \int |h|\,\mathrm{d}\lambda_{H_{m'}}$ for every 
	$h \in \mathrm{C}(H)$, we obtain
	\begin{align*}
		\|\pi(g)\| = \|\pi((f \circ p)g)\| \leq \sup_{m' \in M}\int_{H_{m'}} f(m')|g|\,\mathrm{d}\lambda_{H_{m'}}\leq  \sup_{m' \in O}\int_{H_{m'}} |g|\,\mathrm{d}\lambda_{H_{m'}}. 
	\end{align*}
	Thus, $\|\pi(g)\| \leq \int_{H_{m}} |g|\,\mathrm{d}\lambda_{H_{m}}$ as desired by continuity of the Haar system (see Proposition \ref{Renault}). In particular, $\pi$ factors through the quotient $\mathrm{C}(H_m) \cong \mathrm{C}(H)/\{f \in \mathrm{C}(H)\mid f|_{H_m} = 0\}$.
	
	We conclude that $\pi$ induces a non-degenerate *-algebra representation $\tau$ of the convolution algebra $\mathrm{C}(H_m)$ on $E$ which admits no non-trivial closed invariant subspaces and is bounded with respect to the $\mathrm{L}^1$-norm on $\mathrm{C}(H_m)$. 	By \cite[Proposition 13.3.4 and Remark 13.3.5]{Dixm1977} there is a unique irreducible representation $\varrho$ of $H_m$ such that
	\begin{align*}
		\pi(g) = \varrho(g|_{H_m})
		= \int_{H_m} f(x) \varrho(x) \, \mathrm{d}\lambda_{H_m}(x) = \pi_{\varrho}(g) \textrm{ for every } g \in \mathrm{C}(H)
	\end{align*}	 
	and this shows (\ref{spectrum2}).\medskip\\
	Finally, part (\ref{spectrum3}) is an easy consequence of the uniqueness properties in (i) and (ii).
\end{proof}

Recall that for a C*-algebra $\mathcal{A}$, the spectrum $\hat{\mathcal{A}}$ are equivalence classes of non-zero irreducible representations of $\mathcal{A}$ (see \cite[Chapters 2 and 3]{Dixm1977}). Proposition \ref{spectrum} motivates the following definition (cf. \cite[Section 3]{MRW1996}).

\begin{definition}
	Let $p \colon H \to M$ be an open compact group bundle over a compact space $M$. The \emph{dual $\hat{H}$} of $H$ is the spectrum 
	of the C*-algebra $\mathrm{C}^*(H)$.
\end{definition}

In view of Proposition \ref{spectrum}, $\hat{H}$ is (as a set) in a one-to-one correspondence with the union $\bigcup_{m\in M} \hat{H}_m$, where $\hat{H}_m$ is the unitary dual of the compact group $H_m$ 
for each $m \in M$ (cf. \cite[Sections 7.1 and 7.2]{folland-harmonic-analysis} and \cite[Chapter 18]{Dixm1977}. In the following we will identify elements of $\hat{H}$ with their counterparts 
in $\bigcup_{m \in M} \hat{H}_m$. Moreover, we denote the canonical mapping 
$\hat{H} \to M$ by $\hat{p}$.

\subsection{Positive definite functions}\label{posdefsec}

For a better understanding of the dual  of an open compact group bundle, we recall the concept of positive definite function (see, e.g., \cite[Section 13.4]{Dixm1977}). 
Given a compact group $H$, a unitary representation $\pi \colon H \to \mathscr{L}(E)$ of $H$ on a non-zero Hilbert space $E$, and $v \in E$ with $\|v\| = 1$ we call the mapping 
    \begin{align*}
        f \colon H \to \C, \quad x \mapsto (\pi(x)v|v)
    \end{align*}
a \emph{normalized positive definite function associated with} $\pi$. Note that  if the vector $v \in E$ is \emph{cyclic}, i.e., $\lin \{\pi(x)v\mid x \in H\}$ is dense in $H$, then by \cite[Theorem 13.4.5]{Dixm1977} the unitary representation $\pi$ is determined by the induced map $f$ up to equivalence.
 A normalized positive definite function is \emph{pure} if it is associated with an irreducible unitary representation $\pi \colon H \rightarrow \mathscr{L}(E)$. We write $\mathrm{Pos}_1(H)$ for the set of all normalized positive definite functions associated to unitary representations of $H$ and $\mathrm{PPos}_1(H)$ for the subset of pure ones.

Now if $p \colon H \rightarrow M$ is an open compact group bundle over a compact space $M$, we set
    \begin{align*}
        \mathrm{Pos}_1(H) &\coloneqq \bigcup_{m \in M} \mathrm{Pos}_1(H_m), \textrm{ and }\\
        \mathrm{PPos}_1(H) &\coloneqq \bigcup_{m \in M} \mathrm{PPos}_1(H_m),
    \end{align*}
and equip these spaces with the subspace topology induced by $\mathrm{C}_{p}(H)$.

Similar to the case of locally compact groups, we can relate the space of normalized positive definite functions with the state space of the corresponding C*-algebra  (cf. \cite[Theorem 13.5.2 and Remark 18.1.3]{Dixm1977}). Recall that here, in analogy to the definition of positive definite functions above, for a C*-algebra $\mathcal{A}$, a representation $\pi \colon \mathcal{A} \rightarrow \mathscr{L}(E)$ of $\mathcal{A}$ on a Hilbert space $E$, and a vector $v \in E$ with $\|v\| = 1$ the map
    \begin{align*}
        \omega \colon \mathcal{A} \rightarrow \C,\quad f \mapsto (\pi(f)v|v)
    \end{align*}
is called a \emph{state associated to $\pi$} (see \cite[Sections 2.1 and 2.4]{Dixm1977}). Again, if $v \in E$ is \emph{cyclic}, i.e., $\{\pi(f)v\mid f \in \mathcal{A}\}$ is dense in $H$, then by \cite[Theorem 13.4.5]{Dixm1977}, $\pi$ is uniquely determined up to equivalence by $\omega$ (see \cite[Proposition 2.4.1]{Dixm1977}). The state $\omega$ is called \emph{pure} if it is associated with a non-zero irreducible representation (see \cite[Section 2.5]{Dixm1977}). The convex set of all states is denoted by $\mathrm{S}(\mathcal{A})$ and becomes a compact space when equipped with the weak* topology. We write $\mathrm{PS}(\mathcal{A})$ for the subspace of pure states.

\begin{proposition}\label{posdef}
    Let $p \colon H \rightarrow M$ be an open compact group bundle over a compact space $M$. 
        \begin{enumerate}[(i)]
            \item For every $f \in \mathrm{Pos}_1(H)$ associated to a unitary representation $\varrho$ of a fiber group there is a unique state $\omega_f \in \mathrm{S}(\mathrm{C}^*(H))$ associated to $\pi_\varrho$ with 
                \begin{align*}
                    \omega_f(g) = \int_{H_{s(f)}} f g \, \mathrm{d}\lambda_{H_{s(f)}}
                \end{align*}
                for all $g \in \mathrm{C}(H)$.\label{posdef1}
            \item The map
                \begin{align*} 
                    \mathrm{Pos}_1(H) \rightarrow \mathrm{S}(\mathrm{C}^*(H)), \quad f \mapsto \omega_f
                \end{align*}
            is a homeomorphism onto its range.\label{posdef2}
            \item The restricted map
                 \begin{align*} 
                    \mathrm{PPos}_1(H) \rightarrow \mathrm{PS}(\mathrm{C}^*(H)), \quad f \mapsto \omega_f
                \end{align*}
            is a homeomorphism.\label{posdef3}
        \end{enumerate}
\end{proposition}

The following lemma is the essential part of the proof of Proposition \ref{posdef}. It is is an analogue of \cite[Theorem 13.5.2]{Dixm1977} (see also \cite[Theorem 3.31]{folland-harmonic-analysis}) and uses similar methods.

\begin{lemma}\label{uniformstate}
  Let $p \colon H \rightarrow M$ be an open compact group bundle over a compact space $M$. For a net $(f_\alpha)_{\alpha \in A}$ in
  $\mathrm{Pos}_1(H)$ and $f \in \mathrm{Pos}_1(H)$ the following are equivalent.
    \begin{enumerate}[(a)]
      \item $\lim_\alpha f_\alpha = f$.
      \item $\lim_\alpha \lambda_{H_{s(f_\alpha)}} (f_\alpha g) = \lambda_{H_{s(f)}}(f g)$ for every $g \in \mathrm{C}(H)$.
    \end{enumerate}
\end{lemma}
	
\begin{proof}
	The implication \enquote{(a) $\Rightarrow$ (b)} is a direct consequence of Corollary \ref{convlemma}. 
	Now suppose that (b) holds and write $m_\alpha \coloneqq s(f_\alpha)$ for every 
	$\alpha \in A$ as well as $m \coloneqq s(f)$.
	
	We first show that $\lim_\alpha m_\alpha = m$. Assuming the contrary, we find a subnet $(m_\beta)_{\beta \in B}$ of
	$(m_\alpha)_{\alpha \in A}$ and a neighborhood $O$ of $m$ such that $m_{\beta} \notin O$ for every $\beta \in B$. Choose 
	a function $h \in \mathrm{C}(M)$ with support in $O$ such that $h(m) = 1$. By (b) we obtain
			\begin{align*}
				0 = \lim_\beta \int_{H_{m_{\beta}}} f_{\beta} \cdot (h \circ p) \cdot g  \, \mathrm{d}\lambda_{H_{m_{\beta}}} = \int_{H_{m}} f \cdot (h \circ p) \cdot  g\, \mathrm{d}\lambda_{H_{m}} = \int_{H_m} f \cdot g\, \mathrm{d}\lambda_{H_m}
			\end{align*}
	for every $g \in \mathrm{C}(H)$, hence $f = 0$. Since $f(1_{H_m}) =1$, this is a contradiction.
	
	We now check the second condition of Proposition \ref{descriptiontop} (b). So let $(f_\beta)_{\beta \in B}$ be a subnet of $(f_\alpha)_{\alpha \in A}$ and 
	$(x_\beta)_{\beta \in B}$ a net in $H$ converging to some $x \in H$ such that $p(x_\beta) = s(f_\beta)$ for every $\beta \in B$.
	
	Fix $\varepsilon > 0$ and choose, using the continuity of $f$, a function $g \in \mathrm{C}(H)$ with $0 \leq g(y) \leq 1$ for all $y \in H$ such that $g(1_{H_m}) =1$ and 
		\begin{align}\label{ineq1}
			|f(y) - 1| \leq \varepsilon
			\textrm{ for every } y \in g^{-1}((0,\infty)) \cap H_m.
		\end{align}
    Using the continuity of the Haar system (see Proposition \ref{Renault}) and the fact that $\lambda_{H_m}(g) \neq 0$, we choose $\beta_1 \in B$ such that $\lambda_{H_{m_\beta}}(g) \neq 0$, as well as
		\begin{align}\label{ineq2}
			\frac{ \lambda_{H_m}(g)}{ \lambda_{H_{m_\beta}}(g)}\leq 1+\varepsilon,
		\end{align}
	and 
		\begin{align}\label{ineq3}
			|\lambda_{H_{m_\beta}}(g) - \lambda_{H_{m}}(g)| \leq \lambda_{H_m}(g)\varepsilon
		\end{align}
	for every $\beta \geq \beta_1$. By (b) we find $\beta_2 \geq \beta_1$ such that
		\begin{align}\label{ineq4}
			|\lambda_{H_m}(f g) - \lambda_{H_{m_\beta}}(f_\beta g)| \leq \lambda_{H_m}(g)\varepsilon
		\end{align}
		for every $\beta \geq \beta_2$. 
	For $\beta \geq \beta_2$ we then obtain
		\begin{align*}
		    |\lambda_{H_{m_\beta}}((\mathbbm{1} - f_\beta)g)| \leq |\lambda_{H_{m_\beta}}(g) - \lambda_{H_{m}}(g)| + \lambda_{H_m}(|\mathbbm{1} - f|g) + |\lambda_{H_{m}}(fg) - \lambda_{H_{m_\beta}}(f_\beta g)|,
		\end{align*}
    hence
	    \begin{align}\label{ineq5}
             |\lambda_{H_{m_\beta}}((\mathbbm{1} - f_\beta)g)| \leq 3 \lambda_{H_m}(g)\varepsilon.
	    \end{align}
	by inequalities (\ref{ineq1}), (\ref{ineq3}) and (\ref{ineq4}). Setting $h_{\beta}(y) \coloneqq g(x_\beta y)$ for every $y \in H_{m_{\beta}}$ and $\beta \in B$ as well as $h(y)\coloneqq g(xy)$ for $y \in H_m$, we obtain $\lim_{\beta} h_{\beta} = h$ in $\mathrm{C}_p(H)$. Combining this with (b) yields
        \begin{align*}
            \lim_\beta (g|_{H_{m_\beta}} * f_\beta)(x_\beta)&= \lim_\beta \lambda_{H_{m_\beta}}(h_\beta f_\beta) = \lambda_{H_m}(hf) =  (g|_{H_{m}} * f)(x),
        \end{align*}
    see Proposition \ref{weakstaconvmeas}. We therefore find $\beta_3 \geq \beta_2$ such that
		\begin{align}\label{ineq6}
			\left|\frac{1}{\lambda_{H_{m_\beta}}(g)}\left(g|_{H_{m_\beta}} * f_\beta\right)(x_\beta) - \frac{1}{\lambda_{H_m}(g)} \left(g|_{H_m} * f\right)(x)\right| \leq \varepsilon
		\end{align}
	for every $\beta \geq \beta_3$.	By \cite[Lemma 3.30]{folland-harmonic-analysis} the inequality
		\begin{align*}
			|f_\beta(y^{-1}x_\beta) -f_\beta(x_\beta)|\leq \sqrt{2}(1- \mathrm{Re} \, f_\beta(y))^{\frac{1}{2}}
		\end{align*}
	holds for all $y \in H_{m_\beta}$ and every $\beta \in B$. Hence, 
  \begin{align*}
		\left| \frac{1}{\lambda_{H_{m_\beta}}(g)}\left(g|_{H_{m_\beta}} * f_\beta\right)(x_\beta) - f_\beta(x_\beta)\right| 
		&\leq \frac{1}{\lambda_{H_{m_\beta}}(g)} \int  g(y)|f_\beta(y^{-1}x_\beta) -f_\beta(x_\beta)| \, \mathrm{d}\lambda_{m_\beta}(y)\\
		&\leq \frac{\sqrt{2}}{\lambda_{H_{m_\beta}}(g)} \lambda_{H_{m_\beta}}( g^{\frac{1}{2}} \cdot (g \cdot(\mathbbm{1}-\mathrm{Re}\, f_\beta))^{\frac{1}{2}})
	\end{align*}
    for every $\beta \in B$. Since  $\lambda_{H_{m_\beta}}( g^{\frac{1}{2}} \cdot (g \cdot(\mathbbm{1}-\mathrm{Re}\, f_\beta))^{\frac{1}{2}}) \leq \lambda_{H_{m_\beta}}(g)^{\frac{1}{2}} \cdot  \lambda_{H_{m_\beta}}(g \cdot(\mathbbm{1}-\mathrm{Re}\, f_\beta))^{\frac{1}{2}}$ by Cauchy-Schwarz, we obtain 
        \begin{align*}
            	\left| \frac{1}{\lambda_{H_{m_\beta}}(g)}\left(g|_{H_{m_\beta}} * f_\beta\right)(x_\beta) - f_\beta(x_\beta)\right| 	\leq \left(\frac{2}{\lambda_{H_{m_\beta}}(g)}\right)^{\frac{1}{2}} \lambda_{H_{m_\beta}}(g \cdot(\mathbbm{1}-\mathrm{Re}\, f_\beta))^{\frac{1}{2}},
        \end{align*}
    and, in view of (\ref{ineq2}) and (\ref{ineq5}), 
		\begin{align}\label{ineq7}
		\left| \frac{1}{\lambda_{H_{m_\beta}}(g)}\left(g|_{H_{m_\beta}} * f_\beta\right)(x_\beta) - f_\beta(x_\beta)\right| \leq \sqrt{6\varepsilon}\left(\frac{\lambda_{H_m}(g)}{\lambda_{H_{m_\beta}}(g)}\right)^{\frac{1}{2}} \leq  \sqrt{6\varepsilon}(1 + \varepsilon),
		\end{align}
	for every $\beta \geq \beta_3$. Similar estimates yield
	\begin{align}\label{ineq8}
		\left| \frac{1}{\lambda_{H_m}(g)} (g|_{H_{m}}* f)(x) - f(x)\right|  \leq \sqrt{6\varepsilon}.
	\end{align}
	Combining (\ref{ineq6}), (\ref{ineq7}) and (\ref{ineq8}) we finally obtain
	\begin{align*}
		\left|f_\beta(x_\beta) - f(x)\right| \leq \sqrt{6\varepsilon} + \sqrt{6\varepsilon}(1+\varepsilon) +\varepsilon
	\end{align*}
	for every $\beta \geq \beta_3$.
\end{proof}

\begin{proof}[Proof of Proposition \ref{posdef}]
    Take $m \in M$ and a unitary representation $\varrho \colon H_m \to \mathscr{L}(E)$ of $H_m$ of the compact group $H_m$ on a non-zero Hilbert space $E$ and $v \in E$ with $\|v\|=1$. Let $f$ be the corresponding normalized positive definite function, i.e., $f(x) = (\varrho(x)v|v)$ for all $x \in H_m$. Moreover, consider the non-degenerate representation $\pi_\varrho$ of $\mathrm{C}^*(H)$ of Proposition \ref{spectrum} (\ref{spectrum1}) and the state
        \begin{align*}
            \omega_f \colon \mathrm{C}^*(H) \rightarrow \C, \quad g \mapsto (\pi_\varrho(g)v|v).
        \end{align*}
    associated with $\pi_\varrho$.
    Then,
        \begin{align*}
            \omega_f(g) = (\pi_\varrho(g)v|v) = \int_{H_m} g(x) (\varrho(x)v|v)\, \mathrm{d}\lambda_{H_{m}}(x) = \int_{H_m} fg\, \mathrm{d}\lambda_{H_{m}}
        \end{align*}
    for every $g \in \mathrm{C}(H)$. This shows existence of the desired state $\omega_f$. Uniqueness is clear as $\mathrm{C}(H)$ is dense  $\mathrm{C}^*(H)$. Hence (\ref{posdef1}) is proved.\medskip\\
    For (\ref{posdef2}) we observe that the defined map is clearly injective. It follows from Lemma \ref{uniformstate} that it is a homeomorphism onto its range.\medskip\\
    Finally, part (\ref{posdef3}) is an easy consequence of (\ref{posdef2}) and Proposition \ref{spectrum}.
\end{proof}

\subsection{Topology of the dual}\label{dualtop}
The spectrum of a C*-algebra can be equipped with a topology in a natural way (see, e.g., \cite[Chapter 3]{Dixm1977}). In particular, we obtain a topology on the dual $\hat{H}$ of an open compact group bundle $p \colon H \rightarrow M$ over a compact space $M$. Combining Proposition \ref{posdef} with \cite[Theorem 3.4.10]{Dixm1977} yields the following description of the closure with respect to this topology (cf. \cite[Chapter 18]{Dixm1977} for the case of locally compact groups). Here, we say that $f \in \mathrm{PPos}_1(H)$ is associated with an element $[\pi] \in \hat{H}$ (which is an equivalence class of representations) if it is associated with some (and then every) of its representatives. 
\begin{proposition}\label{uniformapprox}
	Let $p \colon H \to M$ be an open compact group bundle over a compact space $M$. For a subset $C \subset \hat{H}$ and $[\pi] \in \hat{H}$ the following assertions are equivalent.
	\begin{enumerate}[(a)]
		\item $[\pi] \in \overline{C}$.
		\item One $f \in \mathrm{PPos}_1(H)$ associated with $[\pi]$ is the limit of a net $(f_\alpha)_{\alpha \in A}$ in $\mathrm{PPos}_1(H)$ associated with elements of $C$.
		\item Every $f \in \mathrm{PPos}_1(H)$ associated with $[\pi]$ is the limit of a net $(f_\alpha)_{\alpha \in A}$ in $\mathrm{PPos}_1(H)$ associated with elements of $C$.
	\end{enumerate}
\end{proposition}

With the help of Proposition \ref{uniformapprox}, we obtain the following important topological properties of the dual of an open compact group bundle.

\begin{proposition}\label{proptop}
    For an open compact group bundle $p \colon H \rightarrow M$ over a compact space $M$ the following assertions hold. 
        \begin{enumerate}[(i)]
            \item The canonical map $\mathrm{PPos}_1(H) \rightarrow \hat{H}$ is continuous and open.\label{proptop1}
            \item The canonical map $\hat{p} \colon \hat{H} \to  M$ 
	is continuous and open.\label{proptop2}
	        \item $\hat{H}$ is a locally compact (Hausdorff) space.\label{proptop3}
	        \item The map
		        \begin{align*}
			        \dim\colon \hat{H} \to \N, \quad [\pi] \to \mathrm{dim}(\pi)
		        \end{align*}
		          is continuous.\label{proptop4}
        \end{enumerate}
\end{proposition}

The proof of openness in (\ref{proptop2}) is based on an idea communicated to the authors by Jean Renault and uses the following observation (see also \cite{Rena2021} for a different proof in the case of bundles of locally compact abelian groups).

\begin{lemma}\label{lowersemi}
	Let $p \colon H \rightarrow M$ be an open compact group bundle over a compact space $M$ and $f \in \mathrm{C}(H)$. Then the mapping 
            	\begin{align*}
	            	M \to [0, \infty), \quad m \mapsto \|f|_{H_m}\|_{\mathrm{C}^*(H_m)}
	            \end{align*}
        	is lower semicontinuous.
\end{lemma}
\begin{proof}
    Since compact groups are amenable, the full group C*-norm agrees with the reduced group C*-norm for $H_m$ (this is a special case of \cite[Theorem 7.13]{Will2007}), thus,
	\begin{align*}
    \|f|_{H_m}\|_{\mathrm{C}^*(H_m)} 
    = \sup_{\substack{g \in \mathrm{C}(H_m)\\ \|g\|_{\mathrm{L}^2(H_m)}\leq 1}} \sup_{\substack{h \in \mathrm{C}(H_m)\\ \|h\|_{\mathrm{L}^2(H_m)}\leq 1}} \left| \int_{H_m}\int_{H_m} g(x)f(x^{-1}y)h(y)\,\mathrm{d}\lambda_{H_m}(x)\,\mathrm{d}\lambda_{H_m}(y)\right|     
	\end{align*}
	for every $m \in M$.  Now if $g \in \mathrm{C}(H_m) \setminus \{0\}$ with $\|g\|_{\mathrm{L}^2(H_m)} \leq 1$ we find a function $h \in \mathrm{C}(H)$ with $h|_{H_m} = g$ 
	    \begin{align*}
	     \|g\|_2 \coloneqq \sup_{m' \in M}\|g|_{H_{m'}}\|_{\mathrm{L}^2(H_{m'})}\leq 1.
	    \end{align*}
	   This implies
	\begin{align*}
    \|f|_{H_m}\|_{\mathrm{C}^*(H_m)} 
    = \sup_{\substack{g \in \mathrm{C}(H)\\ \|g\|_2 \leq 1}} \sup_{\substack{h \in \mathrm{C}(H)\\ \|h\|_2 \leq 1}} \left| \int_{H_m}\int_{H_m} g(x)f(x^{-1}y)h(y)\,\mathrm{d}\lambda_m(x)\,\mathrm{d}\lambda_m(y)\right|     
	\end{align*}
	for every $m \in M$. Since the pointwise supremum of continuous functions is lower semicontinuous, it suffices to show that
	\begin{align*}
		M \to \C, \quad m \mapsto \int_{H_m}\int_{H_m} g(x)f(x^{-1}y)h(y)\,\mathrm{d}\lambda_m(x)\,\mathrm{d}\lambda_m(y)
	\end{align*}
	is continuous for all $g,h \in \mathrm{C}(H)$. But this readily follows from the continuity of the Haar system (see Proposition \ref{Renault}) and Corollary \ref{convlemma}.
\end{proof}

\begin{proof}[Proof of Proposition \ref{proptop}.]
    Part (\ref{proptop1}) is a consequence of \cite[Theorem 3.4.11]{Dixm1977} by using the identification of $\mathrm{PPos}_1(H)$ with $\mathrm{PS}(\mathrm{C}^*(H))$ established in Proposition \ref{posdef}.\medskip\\
    We proceed with (\ref{proptop2}). Continuity of $\hat{p}$ is an easy consequence of Proposition \ref{uniformapprox}. To show that $\hat{p}$ is open, it suffices to show that the set $\hat{p}(V_f)$ is open for every $f \in \mathrm{C}(H)$, where
		\begin{align*}
			V_f \coloneqq \{[\pi] \in \hat{H}\mid \|\pi(f)\| > 1\}
		\end{align*}
	as these sets form a base of the topology of $\hat{H}$ by  \cite[Lemma 3.3.3]{Dixm1977}.
	 Let $f \in \mathrm{C}(H)$ and pick $m \in p(V_f)$. Choose $[\pi]\in V_f$ with $\hat{p}([\pi]) = m$, i.e., $\|\pi(f)\| > 1$. Then also $\|f|_{H_m}\|_{\mathrm{C}^*(H_m)} > 1$ and by Lemma \ref{lowersemi} there is a neighborhood $O$ of $m$ such that $\|f|_{H_{m'}}\| > 1$ for every $m' \in O$. By \cite[Proposition 2.7.1]{Dixm1977} we therefore find for every $m' \in O$ some $[\pi] \in \hat{H}_{m'}$ with $\|\pi(f)\| > 1$. Hence $O \subset p(V_f)$ and this shows that $p(V_f)$ is open.\medskip\\
	 We next prove (\ref{proptop3}). To see that $\hat{H}$ is a Hausdorff space, pick $[\pi_1], [\pi_2] \in \hat{H}$ with $[\pi_1] \neq [\pi_2]$. If $\hat{p}([\pi_1]) \neq \hat{p}([\pi_2])$ we can use (\ref{proptop2}) to separate $[\pi_1]$ and $[\pi_2]$ by open neighborhoods. We may therefore assume that $m \coloneqq \hat{p}([\pi_1]) = \hat{p}([\pi_2])$. But the subspace topology of $\hat{H}_m \subset \hat{H}$ agrees with the topology of the dual $\hat{H}_m$ of the compact group $H_m$ (use, e.g., Proposition \ref{uniformapprox}) and the latter is discrete by \cite[Proposition 7.4]{folland-harmonic-analysis}. Thus, $[\pi_1]$ and $[\pi_2]$ can also be separated in this case. We now obtain from \cite[Corollary 3.3.8]{Dixm1977} that $\hat{H}$ is locally compact.\medskip\\
	 Finally, we show part (\ref{proptop4}). If $m \in M$, $[\pi] \in \hat{H}_m$ and $f$ is a normalized positive definite function associated with $[\pi]$, then
		\begin{align*}
			\int_{H_m} |f|^2\, \mathrm{d}\lambda_{H_m} = \frac{1}{\dim(\pi)}
		\end{align*}
	by Schur's orthogonality relations (see \cite[Result 5.8]{folland-harmonic-analysis}). Continuity of the dimension map is thus a consequence of Proposition \ref{uniformapprox}.
\end{proof}

\subsection{Continuous sections of the dual}\label{contsectdual}
After establishing the basic topological properties of the dual of an open compact group bundle, we are now ready to prove the following powerful selection theorem. Recall here the definition of local continuous sections from Subsection \ref{basicconc} above.
\begin{theorem}\label{section}
	Let $p \colon H \to M$ be an open compact group bundle over a compact space $M$. For every $[\pi] \in \hat{H}$ there 
	is a local continuous section $\hat{\sigma}$ of the bundle $\hat{p} \colon \hat{H} \rightarrow M$ with $\hat{\sigma}(\hat{p}([\pi])) = [\pi]$.
\end{theorem}

For the proof of Theorem \ref{section} we need the following lemma.

\begin{lemma}\label{neighborhood}
     Let $p \colon H \to M$ be an open compact group bundle over a compact space $M$. Then each $[\pi] \in \hat{H}$ has a neighborhood $O$ such 
	that $O \cap \hat{H}_m$ has at most one element for each $m \in M$.
\end{lemma}
\begin{proof}
	Let $[\pi] \in \hat{H}$ and $f \in \mathrm{PPos}_1(H)$ associated with $[\pi]$. By Schur's orthogonality relations (see \cite[Result 5.8]{folland-harmonic-analysis}) and Corollary \ref{convlemma} the set
		\begin{align*}
			W \coloneqq \{(f_1,f_2) \in \mathrm{PPos}_1(H) \times_M \mathrm{PPos}_1(H)\mid \lambda_{s(f_1)}(f_1 \overline{f_2}) \neq 0\} 
		\end{align*}
	is an open neighborhood of $(f,f)$ in $\mathrm{PPos}_1(H) \times_M \mathrm{PPos}_1(H)$. We therefore find a neighborhood $V$ of $f$ in $\mathrm{PPos}_1(H)$ such 
	that $V \times V \cap  (\mathrm{PPos}_1(H) \times_M \mathrm{PPos}_1(H)) \subset W$.
	
	We claim that the image $O$ of $V$ with respect to the canonical surjection $\mathrm{PPos}_1(H) \rightarrow \hat{H}$ has the desired property. Note first that by Proposition \ref{proptop} (\ref{proptop1}) the set $O$ is actually a neighborhood of $[\pi]$. Now take  $[\pi_1], [\pi_2] \in O$ with $m \coloneqq p([\pi_1]) = p([\pi_2])$. We then find normalized positive definite functions $f_i$ associated 
	with $[\pi_i]$ such that $f_i \in V$ for $i = 1,2$, hence $(f_1,f_2) \in W$. Consequently,
	\begin{align*}
		\int_{H_m} f_1 \overline{f_2}\, \mathrm{d}\lambda_{H_m} \neq 0,
	\end{align*}
	and again by Schur's orthogonality relations (see \cite[Result 5.8]{folland-harmonic-analysis}) we obtain $[\pi_1] = [\pi_2]$.
\end{proof}

\begin{proof}[Proof of Theorem \ref{section}]
	Let $[\pi] \in \hat{H}$ and apply Lemma \ref{neighborhood} to find an open neighborhood $O$ of $[\pi]$ such that
		\begin{align*}
			\hat{p}|_O \colon O \rightarrow \hat{p}(O), \quad [\varrho] \mapsto \hat{p}([\varrho])
		\end{align*}
	is bijective. By Proposition \ref{proptop} (ii) its inverse is a local continuous section $\hat{\sigma}$ of $\hat{p}$ with $\hat{\sigma}(\hat{p}([\pi])) = [\pi]$.
\end{proof}
The following example shows that one cannot expect to obtain global sections in Theorem \ref{section}.
\begin{example}
    Consider the compact space $M = [0,1]$,  the cyclic group of two elements $H = \{-1,1\}$, and the open compact group bundle from Example \ref{examplescompgroup} (ii) $p \colon M \times H/_{\sim} \rightarrow M$  obtained from the trivial bundle $M \times H$ by identifying the two points over $1 \in M$. Then for every $m \in [0,1)$ the element of the dual defined by the character 
        \begin{align*}
            \mathrm{id}_m \colon \{m\} \times \{-1,1\} \rightarrow \C, \quad (m,x) \mapsto x
        \end{align*}
    is not contained in the image of a global section  of the dual.
\end{example}

\subsection{Peter--Weyl theory for compact group bundles}\label{secpw}

Recall that for a compact group $H$ and an element of the dual we can consider the space of associated representative functions (cf. \cite[Section 5.2]{folland-harmonic-analysis}, \cite[Section 27]{hewittross2}): For an irreducible unitary representation $\pi \colon H \rightarrow \mathscr{L}(E)$ and an orthonormal basis $e_1, \dots, e_{\dim{\pi}}$ of $E$ we obtain a $\dim(\pi) \times \dim(\pi)$-matrix $(\pi^{ij})_{i,j}$ of continuous functions on $H$ defined by  
	\begin{align*}
		\pi^{ij} \colon H \rightarrow \C, \quad x \mapsto (\pi(x)e_j|e_i)
	\end{align*}
for $i,j \in \{1, \dots, \dim(\pi)\}$. We call such a matrix a \emph{matrix representation of $[\pi]$}. Note that this notion does not depend on the concrete representative of the equivalence class.
Moreover, 
    \begin{align*}
       \mathrm{R}_{[\pi]} \coloneqq \mathrm{lin} \{\pi^{ij} \mid i,j \in \{1, \dots, \dim(\pi)\}\} \subset \mathrm{C}(H)
    \end{align*}
is the \emph{space of representative functions associated with $[\pi]$}. This is independent of the concrete choice of a matrix representation of $[\pi]$. By the Peter--Weyl theorem \cite[Theorem 5.12]{folland-harmonic-analysis} we obtain the following.
    \begin{theorem}\label{pworg}
        Let $H$ be a compact group. Then the union $\bigcup_{[\pi] \in \hat{H}}  \mathrm{R}_{[\sigma]}$ of the finite-dimensional subspaces $\mathrm{R}_{[\pi]}$ for $[\pi] \in \hat{H}$ is total in $\mathrm{C}(H)$. 
    \end{theorem}

We prove a version of Theorem \ref{pworg} for open compact group bundles. If $p \colon H \rightarrow M$ is an open compact group bundle over a compact space $M$, and  $\hat{\sigma} \colon O \rightarrow \hat{H}$ is a local continuous section we write
        \begin{align*}
             \mathrm{R}_{\hat{\sigma}} \coloneqq \bigcup_{m \in O}  \mathrm{R}_{\hat{\sigma}(m)}. 
        \end{align*}

\begin{theorem}\label{loctrivrep}
    Let $p \colon H \rightarrow M$ be an open compact group bundle over a compact space $M$ and $\hat{\sigma} \colon O \rightarrow \hat{H}$ a local continuous section of the dual. Then $s|_{\mathrm{R}_{\hat{\sigma}}} \colon  \mathrm{R}_{\hat{\sigma}} \rightarrow O, \, f \mapsto s(f)$ is a locally trivial vector bundle. 
\end{theorem}

The proof of Theorem \ref{loctrivrep} requires some preparations. The first is the following additional description of the topology on the dual based on 
\cite[Proposition 18.1.9]{Dixm1977}.

\begin{proposition}\label{convergencematrix}
	Let $p \colon H \rightarrow M$ be an open compact group bundle over a compact space $M$ and consider the clopen subset $\hat{H}^d \coloneqq \{[\pi] \in \hat{H} \mid \dim(\pi) =d\} \subset \hat{H}$ for $d \in \N$.
	For $C \subset \hat{H}^d$ and $[\pi] \in \hat{H}^d$ the following are equivalent. 
	\begin{enumerate}[(a)] 
		\item $[\pi] \in \overline{C}$.
		\item For every matrix representation $(\pi^{ij})_{i,j}$  of $[\pi]$ there  is a 
		net of matrix representations $((\pi_\alpha^{ij})_{i,j})_{\alpha \in A}$ of 
		elements of $C$ such that $\lim_\alpha \pi_\alpha^{ij} = \pi^{ij}$ in $\mathrm{C}_p(H)$ for all $i,j \in \{1,...,d\}$.
	\end{enumerate}
\end{proposition}

\begin{proof}
    Observe that the implication \enquote{(b) $\Rightarrow$ (a)} is a consequence of Proposition \ref{uniformapprox} as the diagonal entries of a matrix representation are normalized positive definite functions.
    
    Now assume that (a) holds and let $P \coloneqq (\pi^{ij})_{i,j}$ be a matrix representation associated with $[\pi]$. 
	By \cite[Theorem 3.5.8]{Dixm1977} and Proposition \ref{spectrum} we find a net of matrix 
	representations $(P_\alpha)_{\alpha \in A} \coloneqq ((\pi_\alpha^{ij})_{i,j})_{\alpha \in A}$ associated with elements of $C$ such that
	\begin{align}\label{convequ}
		\lim_\alpha \int_{H_{\hat{p}([\pi_\alpha])}} f \pi_\alpha^{ij} \, \mathrm{d}\lambda_{H_{\hat{p}([\pi_\alpha])}}
		= \int_{H_{\hat{p}([\pi])}} f \pi^{ij} \, \mathrm{d}\lambda_{\hat{p}([\pi])}
	\end{align}
	for every $f \in \mathrm{C}(H)$ and all $i,j \in \{1, \dots d\}$. We set $m_\alpha \coloneqq \hat{p}([\pi_\alpha])$ for every $\alpha \in A$ and $m \coloneqq \hat{p}([\pi])$. 
	As explained above, we already know from Proposition \ref{uniformapprox} that $\lim_{\alpha} \pi_\alpha^{ii} = \pi^{ii}$ in $\mathrm{C}_p(H)$ for every $i \in \{1, \dots, d\}$ and thus, in particular, $\lim_{\alpha} m_\alpha = m$ in $M$.
	
	To check the second condition of Proposition \ref{descriptiontop} for the off-diagonal entries consider a subnet $(P_\beta)_{\beta \in B}$ of $(P_\alpha)_{\alpha \in A}$ 
	and a net $(x_\beta)_{\beta \in B}$ in $H$ with $\lim_\beta x_\beta = x \in H$. We show that 
	$\lim_{\beta} P_\beta(x_\beta)v = P(x)v$ for every $v \in \C^d$ which implies that $\lim_\beta \pi_\beta^{ij} (x_\beta) = \pi^{ij}(x)$ for all $i,j \in \{1, \dots, d\}$. It suffices to consider vectors $v = \int_{H_{m}} f(y) P(y) w\, \mathrm{d}\lambda_{H_m}(y)$ for some 
	$f\in \mathrm{C}(H)$ and $w \in \C^d$, since, by irreducibility of $\pi$, the whole space $\C^d$ is spanned by such vectors. We define $g_\beta \in \mathrm{C}(H_{m_\beta})$ by $g_\beta(y) \coloneqq f(x_\beta^{-1}(y))$ for every $y \in H_{m_\beta}$ and $g \in \mathrm{C}_p(H_m)$ by $g(y) \coloneqq f(x^{-1}y)$ for all $y \in H_m$. Then $\lim_\beta g_\beta = g$ in $\mathrm{C}_p(H)$. Noting that all matrix coefficients of unitary representations are uniformly bounded by $1$, we can apply Proposition \ref{weakstaconvmeas} componentwise in $\C^d$ to obtain
	    \begin{align*}
	         \lim_{\beta} \int_{H_{m_\beta}} g_\beta(y) P_\beta(y)w\,\mathrm{d}\lambda_{H_{m_\beta}}(y) = \int_{H_{m}} g(y) P(y)w\,\mathrm{d}\lambda_{H_{m}}(y).
	    \end{align*}
	But this means
	    \begin{align*}
	        \lim_{\beta} P_{\beta}(x_\beta)\left(\int_{H_{m_\beta}} f(y) P_{\beta}(y)w\,\mathrm{d}\lambda_{H_{m_\beta}}(y)\right) = P(x)v. 
	    \end{align*}
	On the other hand, (\ref{convequ}) also yields
	    \begin{align*}
	        \lim_{\beta} \int_{H_{m_\beta}} f(y) P_{\beta}(y)w\,\mathrm{d}\lambda_{H_{m_\beta}}(y) = \int_{H_{m}} f(y) P(y)w\,\mathrm{d}\lambda_{H_{m}}(y) = v,
	    \end{align*}
	and hence,
	   \begin{align*}
	        \lim_{\beta} \left\|P_{\beta}(x_\beta)v	- P_\beta(x_\beta)\left(\int_{H_{m_\beta}} f(y) P_{\beta}(y)w\,\mathrm{d}\lambda_{H_{m_\beta}}(y)\right)\right\| = 0.
	   \end{align*}
	We conclude that $\lim_{\beta} P(x_\beta)v = P(x)v$ as desired.
\end{proof}

We shall prove the following lemma with the help of Proposition \ref{convergencematrix}. Recall the notion of continuous Banach bundles and their local subbundles from Subsection \ref{bbundl} above.

\begin{lemma}\label{subbundleopen}
    Let $p \colon H \rightarrow M$ be an open compact group bundle over a compact space $M$ and $\hat{\sigma} \colon O \rightarrow \hat{H}$ a local continuous section of the dual. Then $s|_{ \mathrm{R}_{\hat{\sigma}}}\colon  \mathrm{R}_{\hat{\sigma}} \rightarrow O$ is a local continuous Banach subbundle of the continuous Banach bundle $s\colon \mathrm{C}_p(H) \rightarrow M$.
\end{lemma}
\begin{proof}
    We only have to show that $s|_{\mathrm{R}_{\hat{\sigma}}}\colon  \mathrm{R}_{\hat{\sigma}} \rightarrow O$ is open. By \cite[Theorem 17.7]{AliBor} this is the case if only if the multivalued map
        \begin{align*}
            O \rightarrow \mathcal{P}(\mathrm{C}_p(H)), \quad m \mapsto  \mathrm{R}_{\hat{\sigma}(m)}
        \end{align*}
    to the power set $\mathcal{P}(\mathrm{C}_p(H))$ of $\mathrm{C}_p(H)$ is lower hemicontinuous which is equivalent to
        \begin{align*}
            \bigcup_{m \in \overline{Z}}  \mathrm{R}_{\hat{\sigma}(m)} \subset \overline{\bigcup_{m \in Z}  \mathrm{R}_{\hat{\sigma}(m)} }
        \end{align*}
    for every subset $Z \subset O$. So take $Z \subset O$ and $m \in \overline{Z}$. We choose a matrix representation $(\pi^{ij})_{i,j}$ of $[\pi] \coloneqq \hat{\sigma}(m)$. Since $\hat{\sigma}(\overline{Z}) \subset  \overline{\hat{\sigma}(Z)}$, we can apply Proposition \ref{convergencematrix} to find a net $((\pi_\alpha^{ij})_{i,j})_{\alpha \in A}$ of matrix elements associated with elements of $\hat{\sigma}(Z)\cap \hat{H}^{\dim(\pi)}$ such that $\lim_\alpha \pi_{\alpha}^{ij} = \pi^{ij}$ in $\mathrm{C}_p(H)$ for all $i,j \in \{1,\dots, \mathrm{dim}(\pi)\}$. But this already implies that each element of $\mathrm{R}_{[\pi]}$ is a limit of elements in the union $\bigcup_{m' \in Z}  \mathrm{R}_{\hat{\sigma}(m')}$ as desired.
\end{proof}

We finally prove Theorem \ref{loctrivrep}.
\begin{proof}[Proof of Theorem \ref{loctrivrep}]
    The dimension of the fibers of the Banach bundle $s|_{ \mathrm{R}_{\hat{\sigma}}}$ is locally constant by Proposition \ref{proptop} (iv). Hence, the Banach bundle is locally trivial by Proposition \ref{locconst}.
\end{proof}

Using Theorem \ref{section} and the classical Peter--Weyl result (Theorem \ref{pworg}), we obtain the following Peter--Weyl-type theorem for open compact group bundles. Recall the notion of fiberwise dense and total subsets from Subsection \ref{bbundl}.
\begin{theorem}\label{pw1}
    Let $p \colon H \rightarrow M$ be an open compact group bundle over a compact space $M$. Then the union $\bigcup_{\hat{\sigma} \in \Gamma_{\mathrm{loc}}(\hat{H})}  \mathrm{R}_{\hat{\sigma}}$ defined by the locally trivial vector bundles  $s|_{ \mathrm{R}_{\hat{\sigma}}} \colon  \mathrm{R}_{\hat{\sigma}} \rightarrow O$ for local continuous sections $\hat{\sigma} \colon  O \rightarrow \hat{H}$ is fiberwise total in $\mathrm{C}_p(H)$.
\end{theorem}
We state another version of Theorem \ref{pw1}. For an open compact group bundle $p \colon H \rightarrow M$ and a local continuous section $\hat{\sigma} \colon O \rightarrow \hat{H}$ of the dual, we say that $f \in \mathrm{C}(H)$ is a \emph{relative representative function} with respect to $\hat{\sigma}$ if 
    \begin{enumerate}[(i)]
        \item $f|_{H_m} \in \mathrm{R}_{\hat{\sigma}(m)}$ for every $m \in O$, and
        \item $f|_{H_m} = 0$ for every $m \in M\setminus O$.
    \end{enumerate} 
The space of these relative representative functions is denoted by $\Gamma_{\hat{\sigma}} \subset \mathrm{C}(H)$.
    \begin{theorem}\label{pw}
        Let $p \colon H \rightarrow M$ be an open compact group bundle over a compact space $M$. Then the union $\bigcup_{\sigma \in \Gamma_{\mathrm{loc}}(\hat{H})} \Gamma_{\hat{\sigma}}$ is total in $\mathrm{C}(H)$.
    \end{theorem}
    \begin{proof}
        Let $\hat{\sigma} \colon O \rightarrow \hat{H}$ be a local continuous section of the dual. The bundle $s|_{\mathrm{R}_{\hat{\sigma}}} \colon \mathrm{R}_{\hat{\sigma}} \rightarrow O$ can be extended trivially to a continuous Banach subbundle $s|_{E_{\hat{\sigma}}} \colon E_{\hat{\sigma}} \rightarrow M$ of $s \colon \mathrm{C}_p(H) \rightarrow M$ (see Lemma \ref{remsub}) and the canonical isometric isomorphism $\mathrm{C}(H) \rightarrow \Gamma(\mathrm{C}_p(H))$ of $\mathrm{C}(M)$-modules from Example \ref{extensionbundle2} then maps $\Gamma_{\hat{\sigma}}$ to the $\mathrm{C}(M)$-submodule 
            \begin{align*}
                \Lambda_{\hat{\sigma}} \coloneqq \{\tau \in \Gamma(\mathrm{C}_p(H))\mid \tau(m) \in (E_{\hat{\sigma}})_m \textrm{ for every } m \in M\} \subset \Gamma(\mathrm{C}_q(H)).
            \end{align*}
        Using that every element of a Banach bundle is in the image of a continuous section, we conclude that $\mathrm{R}_{\hat{\sigma}} \subset  E_{\hat{\sigma}} = \{\tau(m) \mid m \in M, \tau \in \Lambda_{\hat{\sigma}}\}$.
        Hence, if $\Lambda$ is the sum of the modules $\Lambda_{\hat{\sigma}}$ for local continuous sections $\sigma \colon O \rightarrow \hat{H}$, then $\{\tau(m)\mid m \in M, \tau \in \Lambda\}$ contains the union $\bigcup_{\hat{\sigma} \in \Gamma_{\mathrm{loc}}(\hat{H})} \mathrm{R}_{\hat{\sigma}}$ and thus is is fiberwise dense in $\mathrm{C}_p(H)$ by Theorem \ref{pw1}. By the Stone-Weiterstrass theorem for bundles (Theorem \ref{swbundles}), we conclude that $\Lambda$ is dense in $\Gamma(\mathrm{C}_p(H))$. The sum of the modules $\Gamma_{\hat{\sigma}}$ for $\hat{\sigma} \in \Gamma_{\mathrm{loc}}(\hat{H})$ is therefore dense in $\mathrm{C}(H)$ as desired.
    \end{proof}
\subsection{Peter--Weyl theory for homogeneous space bundles}\label{secpw2}
\label{representativequotient}

We now turn to representative functions on homogeneous space bundles. For a compact group $H$, a closed subgroup $U \subset H$ and $[\pi] \in \hat{H}$ we write
    \begin{align*}
        \mathrm{R}_{[\pi],U}\coloneqq \{f \in \mathrm{R}_{[\pi]}\mid f(xy) = f(x) \textrm{ for all } x \in H, y \in U\} \subset \mathrm{C}(H),
    \end{align*}
for the corresponding \emph{space of representative functions}, cf. \cite[Subsection 28.72]{hewittross2} (see also \cite{Fara2017}).
Each $f \in \mathrm{R}_{[\pi],U}$ defines a continuous function $f' \in \mathrm{C}(H/U)$ and we set
    \begin{align*}
        \mathrm{R}_{[\pi],U}' \coloneqq \{f' \mid f \in \mathrm{R}_{[\pi],U}\} \subset \mathrm{C}(H/U).
    \end{align*}
We then obtain the following Peter--Weyl type theorem for homogeneous spaces
\begin{theorem}
    Let $H$ be a compact group and $U \subset H$ a closed subgroup. Then the union $\bigcup_{[\pi] \in \hat{H}} \mathrm{R}_{[\pi],U}'$ of the finite-dimensional subspaces $\mathrm{R}_{[\pi],U}'$ for $[\pi] \in \hat{H}$ is total in $\mathrm{C}(H/U)$.
\end{theorem}
\begin{proof}
    Recall that the spaces $\mathrm{R}_{[\pi]}$ for $[\pi] \in \hat{H}$ are pairwise orthogonal subspaces of $\mathrm{L}^2(H)$ and invariant with respect to multiplication from the right. This implies that a function $f \in \lin \bigcup \{\mathrm{R}_{[\pi]}\mid [\pi] \in \hat{H}\}$ satisfies $f(xu) = f(x)$ for all $x \in H$ and $u \in U$ if and only if $f \in \lin \bigcup\{\mathrm{R}_{[\pi],U}\mid [\pi] \in \hat{H}\}$. The result therefore follows from \cite[Subsection 30.60 (a)]{hewittross2}.
\end{proof}
In analogy to the previous section we write, for an open compact group bundle $p \colon H \rightarrow M$, a closed subgroup bundle $p|_U \colon U \rightarrow M$ and a local continuous section $\hat{\sigma} \colon O \rightarrow \hat{H}$ of the dual,
    \begin{align*}
        \mathrm{R}_{\hat{\sigma},U} &\coloneqq \bigcup_{m \in O} \mathrm{R}_{\hat{\sigma}(m),U_m} \subset \mathrm{C}_p(H), \textrm{ and }\\
        \mathrm{R}_{\hat{\sigma},U}' &\coloneqq \bigcup_{m \in O} \mathrm{R}_{\hat{\sigma}(m),U_m}' \subset \mathrm{C}_{p_{/U}}(H/U).
    \end{align*}

The following result shows that these constructions yield locally trivial bundles if and only if the closed subgroup bundle is open. (In the following we will often speak of \enquote{open closed subgroup bundles} $p|_U \colon U \rightarrow M$ of open compact group bundles $p \colon H \rightarrow M$. Note that the term \enquote{open} refers to the map $p|_U$ being open, whereas \enquote{closed} refers to $U \subset H$ being closed in $H$.)

\begin{theorem}\label{loctrivrep2}
    Let $p \colon H \rightarrow M$ be an open compact group bundle over a compact space $M$. For a closed subgroup bundle $p|_{U}\colon U \rightarrow M$ the following assertions are equivalent.
        \begin{enumerate}[(a)]
            \item The map $p|_U \colon U \rightarrow M$ is open.\label{loctrivrep2a}
            \item $s|_{\mathrm{R}_{\hat{\sigma},U}} \colon \mathrm{R}_{\hat{\sigma},U} \rightarrow O$ is a locally trivial vector bundle for every local continuous section $\hat{\sigma} \colon O \rightarrow \hat{H}$.\label{loctrivrep2b}
            \item $s|_{\mathrm{R}_{\hat{\sigma},U}'} \colon \mathrm{R}_{\hat{\sigma},U}' \rightarrow O$ is a locally trivial vector bundle for every local continuous section $\hat{\sigma} \colon O \rightarrow \hat{H}$.\label{loctrivrep2c}
        \end{enumerate}
\end{theorem}

For the proof of Theorem \ref{loctrivrep2} we again need some preliminary results. The first is the following generalization of Lemma \ref{subbundleopen}. 

\begin{lemma}\label{subbundleopen2}
    Consider an open compact group bundle $p \colon H \rightarrow M$ over a compact space $M$, a closed subgroup bundle $p|_U \colon U \rightarrow M$ and a local continuous section of the dual $\hat{\sigma} \colon O\rightarrow \hat{H}$. Then $s|_{\mathrm{R}_{\hat{\sigma},U}} \colon \mathrm{R}_{\hat{\sigma},U} \rightarrow O$ is a local continuous Banach subbundle of the continuous Banach bundle $s \colon \mathrm{C}_p(H) \rightarrow M$.
\end{lemma}
\begin{proof}
    As in the proof of Lemma \ref{subbundleopen} it suffices to show that for every subset $Z \subset O$ the inclusion
        \begin{align*} 
            \bigcup_{m \in \overline{Z}} \mathrm{R}_{\hat{\sigma}(m),U_{m}} \subset \overline{\bigcup_{m \in Z} \mathrm{R}_{\hat{\sigma}(m),U_m}}
        \end{align*}
    holds in $\mathrm{C}_p(H)$. So pick $Z\subset O$ and $m \in \overline{Z}$. We choose a matrix representation $(\pi^{ij})_{i,j}$ of $[\pi] \coloneqq \hat{\sigma}(m)$ such that
        \begin{align*}  
            \mathrm{R}_{[\pi],U} = \mathrm{lin} \{\pi^{ij}\mid 1 \leq i \leq \dim(\pi), 1 \leq j \leq d\}
        \end{align*}
    for some $d \in \{1, \dots, \dim(\pi)\}$ (see  \cite[Subsection 28.72]{hewittross2} (b)).  Since $\hat{\sigma}( \overline{Z}) \subset  \overline{\hat{\sigma}(Z)}\cap \hat{H}^{\dim(\pi)}$, we can apply Proposition \ref{convergencematrix} to find a net $((\pi_\alpha^{ij})_{i,j})_{\alpha \in A}$ of matrix representations associated with elements of $\hat{\sigma}(Z)$ such that $\lim_\alpha \pi_{\alpha}^{ij} = \pi^{ij}$ in $\mathrm{C}_p(H)$ for all $i, j \in \{1,\dots, \mathrm{dim}(\pi)\}$. We write $(m_\alpha)_{\alpha \in A}$ for the net of corresponding base points in $Z$ (which then converges to $m$).
    
    We do not know whether the elements $\pi_\alpha^{ij}$ for $\alpha \in A$, $i \in \{1,\dots, \dim(\pi)\}$ and $j \in \{1, \dots, d\}$ are contained in $\mathrm{R}_{\hat{\sigma},U}$. However, as in \cite{Fara2017} we can project them to $\mathrm{R}_{\hat{\sigma},U}$ and define the functions $\tau_\alpha^{ij} \in \mathrm{R}_{\hat{\sigma},U}$ by
        \begin{align*}
            \tau_\alpha^{ij}(x) \coloneqq  \int_{U_{m_\alpha}} \pi_\alpha^{ij}(xy)\, \mathrm{d}\lambda_{U_{m_\alpha}}(y) = \sum_{k=1}^{\dim(\pi)} \int_{U_{m_\alpha}} \pi_{\alpha}^{kj}(y)\, \mathrm{d}\lambda_{U_{m_\alpha}}(y) \cdot \pi_\alpha^{ik}(x) 
        \end{align*}
    for $x \in H_{m_\alpha}$,  $\alpha \in A$ and $i,j \in \{1, \dots, \dim(\pi)\}$. We claim that
        \begin{align*}
            \lim_{\alpha} \tau_\alpha^{ij} = \pi^{ij} \textrm{ in } \mathrm{C}_p(H) \textrm{ for all } i \in \{1, \dots, \dim(\pi)\} \textrm{ and } j \in \{1, \dots, d\}.
        \end{align*}
    Take a subnet $((\tau_\beta^{ij})_{i,j})_{\beta \in B}$ of $((\tau_\alpha^{ij})_{i,j})_{\alpha \in A}$ and a net $(x_\beta)_{\beta \in B}$ in $H$ with $p(x_\beta) = m_\beta$ for every $\beta \in B$ converging to some point $x \in H_m$. Since $\pi^{kj}(y) = \delta_{kj}$ for all $y \in U_m$ and $\lim_{\beta} \pi_{\beta}^{kj} = \pi^{kj}$, we readily obtain $\lim_{\beta} \sup_{y \in U_{m_\beta}} |\pi_\beta^{kj}(y) - \delta_{kj}| = 0$ for all $k\in \{1, \dots, \dim(\pi)\}$ and $j \in \{1, \dots, d\}$. But 
        \begin{align*}
            \left|\int_{U_{m_\beta}} \pi^{kj}_\beta(y) \, \mathrm{d}\lambda_{U_{m_\beta}}(y) - \delta_{kj}\right| = \left|\int_{U_{m_\beta}} (\pi^{kj}_\beta(y) - \delta_{kj}) \, \mathrm{d}\lambda_{U_{m_\beta}}(y) \right| \leq \sup_{y \in U_{m_\beta}} |\pi_\beta^{kj}(y) - \delta_{kj}|
        \end{align*}
    for every $\beta \in B$ and therefore also $\lim_\beta \int_{U_{m_\beta}} \pi^{kj}_\beta(y) \, \mathrm{d}\lambda_{U_{m_\beta}}(y) = \delta_{kj}$ for every $k \in \{1, \dots, \dim(\pi)\}$ and $j \in \{1, \dots, d\}$. This yields
        \begin{align*}
            \lim_{\beta} \tau^{ij}_{\beta}(x_\beta) = \lim_{\beta} \sum_{k=1}^{\dim(\pi)} \int_{U_{m_\beta}} \pi_{\beta}^{kj}(y)\, \mathrm{d}\lambda_{U_{m_\beta}}(y) \cdot \pi_\alpha^{ik}(x_\beta) = \sum_{k=1}^{\dim(\pi)} \delta_{kj} \pi^{ik}(x)    = \pi^{ij}(x)
        \end{align*}
    for all $i \in \{1, \dots, \dim(\pi)\}$ and $j \in \{1, \dots, d\}$. We conclude that $\lim_\beta \tau^{ij}_\beta = \tau^{ij}$ for all $i \in \{1, \dots, \dim(\pi)\}$ and $j \in \{1, \dots, d\}$. By linearity this already implies that each element of $\mathrm{R}_{[\pi],U_m}$ is a limit of elements in 
        \begin{align*}
            \bigcup_{m' \in Z} \mathrm{R}_{\hat{\sigma}(m'),U_{m'},}
        \end{align*}
    which finishes the proof.
\end{proof}

We now prove Theorem \ref{loctrivrep2}.
    \begin{proof}[Proof of Theorem \ref{loctrivrep2}.]
        We first assume that (\ref{loctrivrep2a}) holds and show (\ref{loctrivrep2b}).
        In view of Lemma \ref{subbundleopen2} and Corollary \ref{charloctriv} as well as Theorem \ref{loctrivrep} it suffices to show that $\mathrm{R}_{\hat{\sigma},U}$ is closed in $\mathrm{R}_{\hat{\sigma}}$. Let $(f_\alpha)_{\alpha \in A}$ be a net in $\mathrm{R}_{\hat{\sigma},U}$ converging to some $f \in \mathrm{R}_{\hat{\sigma}}$. Write $m_\alpha \coloneqq s(f_\alpha)$ for $\alpha \in A$ and $m \coloneqq s(f)$. Now take $u \in U_m$ and $x \in H_m$. Since $p$ and $p|_U$ are both open, we may assume, by passing to a subnet, that there are $u_\alpha \in U_{m_\alpha}$ and $x_\alpha \in H_{m_\alpha}$ for $\alpha \in A$ such that $\lim_\alpha u_\alpha = u$ and  $\lim_\alpha x_\alpha = x$. But then
            \begin{align*}
                f(xu) = \lim_{\alpha} f_\alpha(x_\alpha u_\alpha) = \lim_\alpha f_\alpha(x_\alpha) = f(x),
            \end{align*}
        hence $f \in \mathrm{R}_{\hat{\sigma},U}$ as desired.\medskip\\
    Assume now conversely that (\ref{loctrivrep2b}) holds. To show (\ref{loctrivrep2a}) it suffices by Proposition \ref{Renault} to check that the Haar measures $\lambda_{U_m}$ of $U_m$ for $m \in M$ define a relative measure for $p|_U \colon U \rightarrow M$, i.e., that the map
        \begin{align*}
            M \rightarrow \C, \quad m \mapsto \lambda_{U_m}(f)
        \end{align*}
    is continuous for all $f \in \mathrm{C}(H)$. By Theorem \ref{pw} it is enough to show this for functions $f \in \Gamma_{\hat{\sigma}}$, where $\hat{\sigma} \colon O \rightarrow \hat{H}$ is a local continuous section of the dual. We note that for every $m \in M$ the map $P_m \colon \mathrm{R}_{\hat{\sigma}(m)} \rightarrow  \mathrm{R}_{\hat{\sigma}(m)}$ given by
        \begin{align*}
            P_mg(x) \coloneqq \int_{U_m} g(xu)\, \mathrm{d}\lambda_{U_m}(u)
        \end{align*}
    for $g \in \mathrm{R}_{\sigma(m)}$ and $x \in H_m$ is precisely the orthogonal projection onto $\mathrm{R}_{\hat{\sigma}(m),U_m}$ with respect to the inner product of  $\mathrm{L}^2(H_{m})$ (see also \cite{Fara2017}). Now let $(m_\alpha)_{\alpha \in A}$ be a net in $M$ converging to some $m \in M$. If $m \in M \setminus O$, then $\lim_\alpha \sup_{x \in H_{m_\alpha}} |f(x)| = 0$ since $f|_{H_m} = 0$, and thus $\lim_\alpha \lambda_{H_{m_\alpha}}(f) = 0 = \lambda_{H_m}(f)$.
    We may therefore assume that $m \in O$. Since the Haar system of $p \colon H \rightarrow M$ is continuous (see again Proposition \ref{Renault}) and $s|_{\mathrm{R}_{\hat{\sigma},U}}$ is locally trivial, we find local continuous sections sections $\tau_1, \dots, \tau_d \colon V \rightarrow \mathrm{R}_{\hat{\sigma},U}$ on an open neighborhood $V \subset O$ of $m$ such that $\{\tau_k(m') \mid k \in \{1, \dots, d\}\}$ is an orthonormal basis of $\mathrm{R}_{\hat{\sigma}(m'),U_{m'}}$ with respect to the inner product of $\mathrm{L}^2(H_{m'})$ for every $m' \in V$. Thus,
        \begin{align*}
            P_{m'}g = \sum_{k=1}^d (g|\tau_k(m'))_{\mathrm{L}^2(H_{m'})} \tau_k(m')
        \end{align*}
    for every $g \in \mathrm{R}_{\hat{\sigma}(m'),U_{m'}}$ and $m' \in V$. This shows
        \begin{align*}
            \lambda_{H_m}(f) &=  P_mf|_{H_m}(1_{H_m}) = \sum_{k=1}^d (f|_{H_m}|\tau_k(m))_{\mathrm{L}^2(H_m)} \tau_k(m)(1_{H_m}) \\
            &= \lim_\alpha \sum_{k=1}^d (f|_{H_{m_\alpha}}|\tau_k(m_\alpha))_{\mathrm{L}^2(H_{m_\alpha})} \tau_k(m_\alpha)(1_{H_{m_\alpha}}) = \lim_{\alpha} P_{m_\alpha}f|_{H_{m_\alpha}}(1_{H_{m_\alpha}}) \\
            &= \lim_{\alpha} \lambda_{H_{m_\alpha}}(f),
        \end{align*}
    as desired.
    
    Finally, for the equivalence of (\ref{loctrivrep2b}) and (\ref{loctrivrep2c}) observe that for a local continuous section $\hat{\sigma}\colon O \rightarrow \hat{H}$ of the dual and the quotient map $q \colon H \rightarrow H/U$ the map
        \begin{align*}
            \Phi \colon \mathrm{R}_{\hat{\sigma},U}' \rightarrow \mathrm{R}_{\hat{\sigma},U}, \quad f \mapsto f \circ q
        \end{align*}
    is a homeomorphism with $\Phi \circ s|_{\mathrm{R}_{\hat{\sigma},U}'} = s|_{\mathrm{R}_{\hat{\sigma},U}}$ and a vector space isomorphism on each fiber. Hence $s|_{\mathrm{R}_{\hat{\sigma},U}'}\colon \mathrm{R}_{\hat{\sigma},U}' \rightarrow O$ is a locally trivial vector bundle if and only if $s|_{\mathrm{R}_{\hat{\sigma},U}}\colon \mathrm{R}_{\hat{\sigma},U} \rightarrow O$ is a locally trivial vector bundle.
\end{proof}
We have the following generalization of Theorem \ref{pw1}.
\begin{theorem}\label{pw2}
    Let $p \colon H \rightarrow M$ be an open compact group bundle over a compact space $M$ and $p|_U \colon U \rightarrow M$ an open closed subgroup bundle. Then the union $\bigcup_{\hat{\sigma} \in \Gamma_{\mathrm{loc}}(\hat{H})} \mathrm{R}_{\hat{\sigma},U}'$ defined by the locally trivial vector bundles  $s|_{\mathrm{R}_{\hat{\sigma},U}'} \colon \mathrm{R}_{\hat{\sigma},U}' \rightarrow O$ for local continuous sections $\hat{\sigma} \colon  O \rightarrow \hat{H}$ is fiberwise total in $\mathrm{C}_{p_{/U}}(H/U)$.
\end{theorem}
Similar to the previous subsection, we reformulate Theorem \ref{pw2}. For an open compact group bundle $p \colon H \rightarrow M$ over a compact space $M$, a closed subgroup bundle $p|_U \colon U \rightarrow M$ and a local continuous section $\hat{\sigma} \colon O \rightarrow \hat{H}$ of the dual, we call $f \in \mathrm{C}(H/U)$ a \emph{relative representative function} with respect to $\hat{\sigma}$ if
\begin{enumerate}[(i)]
    \item $f|_{H_m/U_m} \in \mathrm{R}_{\hat{\sigma}(m),U_m}'$ for every $m \in O$, and
    \item $f|_{H_m/U_m} = 0$ for every $m \in M \setminus O$.
\end{enumerate}
We write $\Gamma_{\hat{\sigma},U}'$ for the space of all such relative representative functions with respect to $\hat{\sigma}$. The same arguments as in the proof of Theorem \ref{pw} then yield the following.

\begin{theorem}\label{pw3}
    Let $p \colon H \rightarrow M$ be an open compact group bundle over a compact space $M$ and $p|_U \colon U \rightarrow M$ an open closed subgroup bundle. Then the union $\bigcup_{\hat{\sigma} \in \Gamma_{\mathrm{loc}}(\hat{H})} \Gamma_{\hat{\sigma},U}'$ is total in $\mathrm{C}(H/U)$.
\end{theorem}

\section{Structured extensions over Stonean spaces}\label{sectop}
In this section, we prove a representation theorem for certain structured extensions of topological dynamical systems. Throughout the whole section, $G$ is a fixed group equipped with the discrete topology.
\subsection{Topological dynamical systems and their extensions}\label{sectopdyn}
We start by recalling some key concepts from topological dynamics while at the same time fixing our notation. We refer to \cite{Ausl1988} for a general introduction to topological dynamical systems (see also \cite{Bron1979} and \cite{deVr1993}).\medskip\\
    \textbf{Topological dynamical systems.} A \emph{topological dynamical system} $(K;\varphi)$ is a pair of a compact space $K$ and a group homomorphism
		\begin{align*}
			\varphi \colon G \rightarrow \mathrm{Homeo}(K), \quad t \mapsto \varphi_t.
		\end{align*}
	from $G$ to the group of homeomorphisms $\mathrm{Homeo}(K)$ on $K$. A \emph{morphism} $q \colon (K;\varphi) \rightarrow (L;\psi)$ between topological dynamical systems is a continuous map $q \colon K \rightarrow L$ such that the diagram 
	\[
		\xymatrix{
			K \ar[r]^{\varphi_t}  \ar[d]^q &K  \ar[d]^q\\
			L \ar[r]^{\psi_t} & L
		}
	\]	
	commutes for every $t \in G$. If $q$ is surjective, then the morphism is called an \emph{extension}, and in this case, $(L;\psi)$ is a \emph{factor} of $(K;\varphi)$. We say that $q$ is an \emph{isomorphism} if $q$ is bijective (hence a homeomorphism).
	
	We  also want to compare extensions of dynamical systems over a given factor: If $q_1 \colon (K_1;\varphi_1) \rightarrow (L;\psi)$ and $q_2 \colon (K_2;\varphi_2) \rightarrow (L;\psi)$ are extensions between topological dynamical systems, then a \emph{morphism} $\Phi \colon q_1 \rightarrow q_2$  between these extensions is a morphism $\Phi \colon (K_1;\varphi_1) \rightarrow (K_2;\varphi_2)$ of topological dynamical systems such that the diagram
	\[
		\xymatrix{
			(K_1;\varphi_1) \ar[rr]^{\Phi}  \ar[rd]_{q_1} & &(K_2;\varphi_2) \ar[ld]^{q_2}\\
		    & (L;\psi) &
		}
	\]	
	commutes. It is an \emph{isomorphism of extensions} if the underlying morphism $\Phi \colon (K_1;\varphi_1) \rightarrow (K_2;\varphi_2)$ is an isomorphism of topological dynamical systems.\medskip\\
	\textbf{Koopman representations.} Every topological dynamical system $(K;\varphi)$ defines a \emph{Koopman representation}
		\begin{align*}
			T_\varphi \colon G \rightarrow \mathscr{L}(\mathrm{C}(K)), \quad t \mapsto (T_{\varphi})_t
		\end{align*}
	on the Banach space $\mathrm{C}(K)$ by setting $(T_{\varphi})_tf \coloneqq f \circ \varphi_{t^{-1}}$ for $f \in \mathrm{C}(K)$ and $t \in G$ (cf. \cite[Subsection 4.3]{EFHN} and \cite[Section 2]{jt-foundational}). Likewise, every extension $q \colon (K;\varphi) \rightarrow (L;\psi)$ between topological dynamical systems induces an embedding
		\begin{align*}
			J_q \colon \mathrm{C}(L) \rightarrow \mathrm{C}(K),\quad f \mapsto f \circ q
		\end{align*}
	which intertwines the corresponding Koopman representations. Recall that $J_q$ induces the structure of a $\mathrm{C}(L)$-module on $\mathrm{C}(K)$: For $f\in \mathrm{C}(L)$ and $g \in \mathrm{C}(K)$ we define the product $f \cdot g \coloneqq (J_qf) g$.
	
	The pair $(\mathrm{C}(K);T_\varphi)$ induced by a topological dynamical system $(K;\varphi)$ is the prototypical example of a unital commutative $\mathrm{C}^*$-algebra and a representation of $G$ as automorphisms of this C*-algebra. By Gelfand--Naimark representation theory (see, e.g., \cite[Theorem 4.13 and Theorem 4.23]{EFHN}), for every pair $(\mathcal{A};T)$ of a unital commutative $\mathrm{C}^*$-algebra $\mathcal{A}$ and a representation $T \colon G \rightarrow \mathrm{Aut}(\mathcal{A})$ of $G$ there is a topological dynamical system $(K;\varphi)$ such that $(\mathcal{A};T)$ is isomorphic to $(\mathrm{C}(K);T_\varphi)$ (that is, there is an isomorphism of $\mathrm{C}^*$-algebras $I \colon \mathcal{A} \rightarrow \mathrm{C}(K)$ intertwining the representations). Moreover, such a system $(K;\varphi)$ is uniquely determined up to an isomorphism and can be built canonically: Take $K$ to be the Gelfand space consisting of all unital *-homomorphisms $\chi \colon \mathcal{A} \rightarrow \C$ equipped with the weak* topology and let $\varphi_t$ be the restriction of the adjoint operator $T_{t^{-1}}' \in \mathscr{L}(\mathcal{A}')$ to $K$ for $t \in G$. We call $(K;\varphi)$ the \emph{Gelfand system} associated with $(\mathcal{A};T)$ and the isomorphism $I \colon (\mathcal{A};T) \rightarrow (\mathrm{C}(K);T_\varphi)$ given by $I f(\chi) = \chi(f)$ for $f \in \mathcal{A}$ and $\chi \in K$ is called the \emph{Gelfand isomorphism}.\medskip\\
	\textbf{Topological ergodicity.} The usual concept of irreducibility in topological dynamics is minimality (see, e.g., \cite[Chapter 1]{Ausl1988}). However, here we will be concerned with the more general concept of \emph{topologically ergodic} system and extension. Call a factor $(L;\mathrm{id})$ of $(K;\varphi)$ \emph{trivial}, if $\mathrm{id}$ is the trivial action of $G$ on $L$. There is always a largest trivial factor $(M;\mathrm{id})$ of $(K;\varphi)$ with respect to taking extensions, i.e., if $q \colon (K;\varphi) \rightarrow (M;\mathrm{id})$ is the corresponding extension and $p \colon (K;\varphi) \rightarrow (L;\mathrm{id})$ is an extension to another trivial factor $(L;\mathrm{id})$, then there is a surjective morphism of extensions $\Phi\colon q \rightarrow p$. Such a largest trivial factor is unique up to an isomorphism and can be constructed concretely by taking the trivial action on the Gelfand space $K_{\fix}$ of the \emph{fixed space}
	    \begin{align*}
	        \fix(T_\varphi) \coloneqq \{f \in \mathrm{C}(K)\mid (T_\varphi)_tf = f \textrm{ for every } t \in G\}
	   \end{align*}
	  which is a invariant unital C*-subalgebra of $\mathrm{C}(K)$, and the factor map $q_{\fix} \colon K \rightarrow K_{\fix}$ induced by the inclusion $\fix(T_\varphi)\hookrightarrow \mathrm{C}(K)$ (cf. \cite[Subsection 3.3]{EdKr2021}). We call $K_{\fix}$ the \emph{fixed factor} of $(K;\varphi)$.
	  
	If $K_{\fix}$ is a singleton (meaning that $\fix(T_\varphi)$ only contains the constant functions), then we say that the system $(K;\varphi)$ is \emph{topologically ergodic}. An extension $q\colon (K;\varphi) \rightarrow (L;\psi)$ is \emph{topologically ergodic} if the induced map between the fixed factors is a homeomorphism or, equivalently, $J_q(\fix(T_\psi)) = \fix(T_\varphi)$. 

    The notion of topological ergodicity used here is inspired by the notion of ergodicity in ergodic theory, see also \cite[Remark 3.22]{EdKr2021}. We warn the reader that some authors in topological dynamics use a different concept of topological ergodicity which is closely related to topological transitivity of a system.

\subsection{Relative skew-products}\label{relskew}
Important examples of extensions are given by so-called skew-products.
\begin{example}\label{skewprod}
 For $G = \Z$ consider the $\Z$-action $(L;\psi)$ on the torus $L = \T \coloneqq \{x \in \C\mid |x| = 1\}$ given by $\psi_1\colon \T \rightarrow \T, \, x \mapsto ax$ for some fixed $a \in \T$. By equipping $K \coloneqq L \times \T = \T^2$ with with the action $\varphi$ induced by  $\varphi_1\colon \T^2 \rightarrow \T^2, \, (x,y) \mapsto (ax,xy)$ we obtain the so-called \emph{skew-torus} $(\T^2;\varphi)$, explicitly given by 
    \begin{align*}
         \varphi\colon \Z \times \T^2 \rightarrow \T^2, \quad (n,x,y) \mapsto (a^nx,a^{\frac{n(n-1)}{2}}x^n y).
    \end{align*}
 Projecting onto the first component yields an extension $q \colon (\T^2;\varphi) \rightarrow (\T;\psi)$ of $G$-dynamical systems.
 
    More generally and for arbitrary (discrete) $G$, let $(L;\psi)$ be any topological dynamical system, $H$ a compact group, $U$ a closed subgroup of $H$ and and $c \colon G \times L \rightarrow H$ a \emph{cocycle}, i.e., $c$ is continuous and satisfies the condition
        \begin{align*}
            c(t_1t_2,l) = c(t_1,t_2l) c(t_2,l) \textrm{ for all } t_1,t_2 \in G \textrm{ and } l \in L.
        \end{align*}
    We then obtain a topological dynamical system $(L \times H/U; \varphi_c)$,  called a \emph{skew-product system}, by setting $(\varphi_c)_t(l,xU) \coloneqq (\psi_t(l),c(t,l)xU)$ for $(l,xU) \in L \times H/U$ and $t \in G$. By projecting onto the first component we obtain an extension $q_c \colon (L \times H/U;\varphi_c) \rightarrow (L;\psi)$ which we call a \emph{skew-product extension} (cf. \cite[Subsection 2.2.2]{EFHN}). 
 \end{example}

With the concept of compact group bundle from the previous section, we generalize the construction of skew-products (cf. \cite[Definitions 3.3 and 3.4]{austin-fund}).

    \begin{definition}\label{relativecoc}
		Let $(L;\psi)$ be a topological dynamical system and $p\colon H \rightarrow L_{\mathrm{fix}}$ a compact group bundle over its fixed factor. A continuous map $c \colon G \times L \rightarrow H$ is a \emph{relative cocycle over $(L;\psi)$} if 
		    \begin{enumerate}[(i)]
		        \item $p(c(t,l)) = q_{\mathrm{fix}}(l)$ for every $t \in G$ and $l \in L$, and
		        \item $c(t_1t_2,l) = c(t_1,t_2l) c(t_2,l)$ for all $t_1, t_2 \in G$ and $l \in L$.
            \end{enumerate}
        Given such a relative cocycle $c$ and a closed subgroup bundle $p|_U \colon U \rightarrow M$ of $p$, we  call the dynamical system $(L \times_{\mathrm{\fix}} H/U;\varphi_c)$ with
			\begin{align*}
				\varphi_c\colon G \times (L \times_{L_{\fix}} H/U) \rightarrow  L \times_{L_{\fix}} H/U, \quad (t,(l,xU)) \mapsto (\psi_t(l), c(t,l)xU)
			\end{align*}
		the induced \emph{relative skew-product system}, and
			\begin{align*}
				q_c\colon (L \times_{L_{\fix}} H/U;\varphi_c) \rightarrow (L;\psi), \quad (l,xU) \mapsto l
			\end{align*}
    	the corresponding \emph{relative skew-product extension} of $(L;\psi)$.
	\end{definition}
    Notice that we indeed obtain topological dynamical systems in this way (use Lemma \ref{remarkshomspace}).
    
	If the system $(L;\psi)$ in Definition \ref{relativecoc} is topologically ergodic, these concepts reduce to the ones from Example \ref{skewprod}. We give another example.
	\begin{example}
	    Let $p \colon H \rightarrow M$ be a compact group bundle over a compact space $M$. Considering the trivial dynamical system $(M;\mathrm{id})$ on $M$, we can identify $M_{\fix}$ with $M$, and a relative cocycle is then a continuous map $c \colon G \times M \rightarrow H$ such that
	        \begin{enumerate}[(i)]
	            \item $p(c(t,m)) = m$ for every $t \in G$ and $m \in M$, and
	            \item the restrictions  $c|_{G \times \{m\}} \colon G \times \{m\} \cong G\rightarrow H_m$ are group homomorphisms for every $m \in M$.
	        \end{enumerate}
	    Thus, loosely speaking, $c$ consists of a family of group homomorphisms $c_m \colon G \rightarrow H_m$ to the fiber groups $H_m$ which depend continuously on the base point $m \in M$. For a closed subgroup bundle $p|_{U} \colon U \rightarrow M$, the corresponding relative skew-product system is given by fiberwise rotations via 
	        \begin{align*}
	            \varphi_c \colon G \times H/U \rightarrow H/U, \quad (t,xU) \mapsto c(t,l)xU.
	        \end{align*}
	    For $G= \Z$ this special case has already been considered by the first author (cf. the notion of \enquote{group rotation bundles} from \cite[Definition 2.18]{edeko}).
	\end{example}

\subsection{Relative discrete spectrum over Stonean spaces}\label{secStone}
To prove our version of the Mackey--Zimmer representation theorem, we first study \enquote{structured} extensions $q \colon (K;\varphi) \rightarrow (L;\psi)$ of topological dynamical systems. We focus on the case that $L$ is a Stonean space, i.e., an extremally disconnected compact space. From a functional analytical perspective this means, that the corresponding C*-algebra $\mathrm{C}(L)$ is Dedekind complete, i.e., every subset of real-valued functions bounded from above has a supremum in $\mathrm{C}(L)$. Thus, $\mathrm{C}(L)$ is a so-called \emph{Stone algebra} (see \cite[Subsection 1.3]{EHK2021}). The following basic observation will be important later.

\begin{lemma}\label{lemstonean}
    Let $(L;\psi)$ be a topological dynamical system on a Stonean space $L$. Then $L_{\fix}$ is also Stonean and the map $q_{\fix} \colon L \rightarrow L_{\fix}$ is open.
\end{lemma}
\begin{proof}
    If $F \subset \fix(T_\psi)$ is a subset of real-valued functions which is bounded from above and $g \coloneqq \sup_{f \in F} f$ its supremum in $\mathrm{C}(L)$, then also $g \in \fix(T_\psi)$. This shows that $\fix(T_\varphi)$ is also Dedekind complete, hence $L_{\fix}$ is Stonean, and that the embedding $\fix(T_\psi) \hookrightarrow \mathrm{C}(L)$ is order continuous, hence $q_{\fix}$ is open (see \cite[Proposition III.9.3]{Schaefer1970} or \cite[Satz 2.4]{Nage1972}).
\end{proof}

The following definition of \enquote{structured extensions} is motivated by the article \cite{EdKr2021}, in which an operator-theoretic characterization of open so-called pseudoisometric extensions between topologically ergodic systems is shown (see \cite[Theorem 7.2]{EdKr2021}). 
	\begin{definition}\label{reldiscspdef}
		We say that an open extension $q \colon (K;\varphi) \rightarrow (L;\psi)$ of topological dynamical systems over a Stonean space $L$ has \emph{relative discrete spectrum} if $\mathrm{C}(K)$ is generated as a C*-algebra by its finitely generated, projective, closed, $T_\varphi$-invariant $\mathrm{C}(L)$-submodules.
	\end{definition}
    \begin{remarks}\label{remdiscspec}
        \begin{enumerate}[(i)]
            \item By the Stone--Weierstrass theorem, the condition in Definition \ref{reldiscspdef} is equivalent to saying that the finitely-generated, closed, projective, invariant $\mathrm{C}(L)$-submodules separate the points of the space $K$.
            \item If $(L;\psi)$ is the trivial action of $G$ on a point, then Definition \ref{reldiscspdef} reduces to the usual notion of discrete spectrum for the group representation $T_\varphi$ as discussed in \cite{kreidler-hermle} (see also \cite{NaWo1972}).
	       \item While Definition \ref{reldiscspdef} can be written down for arbitrary compact $L$, it is probably not \enquote{the correct one} in the situation of a general extension $q \colon (K;\varphi) \rightarrow (L;\psi)$, see also Remarks \ref{remarkgrpext} and \ref{gap} below. We forgo a systematic study of structured extensions between arbitrary topological dynamical systems here as the case of Stonean base spaces is sufficient to cover the measure-theoretic applications of Section \ref{measure}.
        \end{enumerate}
	\end{remarks}
	We use the representation theory of compact group bundles developed in Section \ref{seccompgrp} to show the following result on relative skew-product extensions.
	\begin{theorem}\label{reldiscsp}
	    Consider
	        \begin{enumerate}[(i)]
	            \item a topological dynamical system $(L;\psi)$ on a Stonean space $L$,
	            \item an open compact group bundle $p \colon H \rightarrow L_{\fix}$,
	            \item an open closed subgroup bundle $p|_U \colon U \rightarrow L_{\fix}$, and 
	            \item a relative cocycle $c \colon G \times L \rightarrow H$.
	        \end{enumerate}
	   Then $q_c \colon (L \times_{L_{\fix}} H/U;\varphi_c) \rightarrow (L;\psi)$ is an open extension such that the union of all finitely generated, projective, closed, $T_{\varphi_c}$-invariant $\mathrm{C}(L)$-submodules is total in $\mathrm{C}(L \times_{L_{\fix}} H/U)$. In particular, $q_c$ has relative discrete spectrum.
	\end{theorem}
	We first prove the following auxiliary result based on the famous Serre-Swan duality between locally trivial vector bundles and finitely generated, projective modules (see \cite{Swan1962}, and also Proposition \ref{SerreSwan} above).
	\begin{lemma}\label{lemmareldiscsp}
	    Let $p_1 \colon L \rightarrow M$ and $p_2 \colon \Omega \rightarrow M$ be continuous surjections between compact spaces and assume that $p_2$ is open. If $E\subset \mathrm{C}_{p_2}(\Omega)$ is such that the restriction $s|_E \colon E \rightarrow M$ is a locally trivial vector bundle, then 
	        \begin{align*}
	            \Gamma_E \coloneqq \{f \in \mathrm{C}(L \times_M \Omega)\mid f_{l} \in E_{p_1(l)} \textrm{ for every } l \in L\},
	       \end{align*}
	  where $f_{l}(x) \coloneqq f(l,x)$ for $x \in \Omega_{p_1(l)}$, $l \in L$ and $f \in \mathrm{C}(L \times_M \Omega)$, is a finitely generated, projective, and closed $\mathrm{C}(L)$-submodule of $\mathrm{C}(L \times_M \Omega)$ with $E_{p_1(l)} = \{f_l \mid f \in \Gamma_E\}$ for every $l \in L$.
	\end{lemma}
	\begin{proof}
	    For each $l \in L$ the map $\mathrm{C}(L \times_M \Omega) \rightarrow \mathrm{C}(\Omega_{p_1(l)}), \, f \mapsto f_l$ is continuous and $E_{p_1(l)} \subseteq \mathrm{C}(\Omega_{p_1(l)})$ is closed (since it is finite-dimensional). We conclude that $\Gamma_E$ is closed as an intersection of closed sets. We now apply the Serre-Swan duality theorem to see that it is also finitely generated and projective. We observe that the pullback bundle given by $L \times_{p_1,s} E \rightarrow L, \, (l,v) \mapsto l$ is a locally trivial vector bundle over $L$ (see, e.g., \cite[Chapter 3, Proposition 3.1]{Huse1994}). With the projection map $\mathrm{pr}_1 \colon L \times_M \Omega \rightarrow L$ onto the first component, which is is continuous and open, consider
            \begin{align*}
                \Phi \colon L \times_{p_1,s} E  \rightarrow \mathrm{C}_{\mathrm{pr}_1}(L \times_{M} \Omega), \quad (l,f) \mapsto f^l,
            \end{align*}
       where $f^l(l,x) = f(x)$ for $f \in E$ and $(l,x) \in L \times_{M} \Omega$ with $s(f) = p_1(l) = p_2(x)$. Then $\Phi$ is a homeomorphism onto its range.
        Since $\Phi$ is also fiberwise linear, we conclude that $F \coloneqq \Phi(L \times_{p_1,s} E)$
        is also a locally trivial vector bundle over $L$. By the Serre-Swan duality theorem (see \cite{Swan1962}), the space of its continuous sections $\Gamma(F)$ is a projective and finitely generated $\mathrm{C}(L)$-module.
        
        We show that the canonical isomorphism $\mathrm{C}(L \times_{M} \Omega) \rightarrow \Gamma(\mathrm{C}_{\mathrm{pr}_1}(L \times_M \Omega))$ from Example \ref{extensionbundle2} maps $\Gamma_E$ to $\Gamma(F)$ showing that also $\Gamma_E$ is a finitely generated and projective $\mathrm{C}(L)$-module. In fact, if $f \in \Gamma_E$ and $l \in L$, then the corresponding section $\tau \colon L \rightarrow L \times_M \Omega$ given by $\tau(l)(l,x) = f(l,x)$ for $(l,x) \in L\times_M \Omega$ satisfies $\tau(l) = (f_l)^l \in F$ for every $l \in L$. Conversely, assume that for $f \in \mathrm{C}(\Omega \times_L M)$ the corresponding section $\tau \colon L \rightarrow L \times_M \Omega$ belongs to $\Gamma(F)$ and take $l \in L$. Then there is $g \in E_{p_1(l)}$ with $\tau(l) = g^l$ and therefore $f_l (x) = f(l,x) = \tau(l)(l,x) = g(x)$ for all $x \in \Omega_{p_1(l)}$, hence $f_l = g \in E_{p_1(l)}$. This shows $f \in \Gamma_E$ as desired.
	       
	    Finally, we take $l \in L$ and show that $E_{p_1(l)} = \{f_l \mid f \in \Gamma_E\}$ for every $l \in L$. If $g \in E_{p_1(l)}$, then $g^l \in F$ and we therefore find $\tau \in \Gamma(F)$ with $\tau(l) = g^l$. But then the corresponding function $f \in \Gamma_E$ satisfies $f_l(x) = \tau(l)(l,x) = g(x)$ for all $x \in \Omega_l$, i.e., $g = f_l$. Thus, $E_{p_1(l)} = \{f_l \mid f \in \Gamma_E\}$ for every $l \in L$ as claimed.
	\end{proof}
    We now prove Theorem \ref{reldiscsp}.
    \begin{proof}[Proof of Theorem \ref{reldiscsp}.]
        We first check that $q_c$ is open. By assumption $p$ is open, and hence also $p_{/U} \colon H/U \rightarrow L_{\fix}$ is open (see Lemma \ref{remarkshomspace} (iv)). Thus, if $(l,xU) \in L \times_{L_{\fix}} H/U$ and $(l_{\alpha})_{\alpha \in A}$ is a net in $L$ converging to $l$, we may assume, by passing to a subnet, that there is a net $(x_{\alpha}U)_{\alpha \in A}$ in $H/U$ converging to $xU$ with $p_{/U}(x_{\alpha}U) = q_{\fix}(l_{\alpha})$ for each $\alpha \in A$. But then the net $((l_{\alpha},x_{\alpha}U))_{\alpha \in A}$ in $L \times_{L_{\fix}} H/U$ converges to $(l,xU)$ and satisfies $q_c((l_{\alpha},x_{\alpha}U)) = l_{\alpha}$ for each $\alpha \in A$. This shows that $q_c$ is open.
        
        Now let $\Gamma$ be the linear span of all finitely generated, projective, closed, $T_{\varphi_c}$-invariant $\mathrm{C}(L)$-submodules of $\mathrm{C}(L \times_{L_{\fix}} H/U)$. In view of the canonical isomorphism $\mathrm{C}(L \times_{L_{\fix}} H/U) \rightarrow \Gamma(\mathrm{C}_q(L \times_{L_{\fix}} H/U))$ and the Stone-Weierstrass theorem for Banach bundles (Proposition \ref{swbundles}) it suffices to show that $\{f|_{\{l\} \times H_{q_{\fix}(l)}/U_{q_{\fix(l)}}}\mid f \in \Gamma\}$ is dense in $\mathrm{C}(\{l\} \times H_{q_{\fix}(l)}/U_{q_{\fix(l)}})$ for every $l \in L$. Equivalently and with the notation of Lemma \ref{lemmareldiscsp}, we have to show that $\Gamma_l \coloneqq \{f_l\mid f \in \Gamma\}$ is dense in $\mathrm{C}(H_{q_{\fix}(l)}/U_{q_{\fix}(l)})$ for each $l \in L$.
        
        By Theorem \ref{pw2} it is enough to prove the inclusion $\mathrm{R}'_{\hat{\sigma}(m),U_{m}} \subset \Gamma_l$ for every $l \in L$ and each local continuous section $\hat{\sigma} \colon O \rightarrow \hat{H}$ of the dual defined on an open neighborhood $O$ of $m \coloneqq q_{\fix}(l)$. Since $L_{\fix}$ is Stonean by Lemma \ref{lemstonean}, we may assume (by restricting the section) that $O \subset L_{\fix}$ is clopen. Then the locally trivial vector bundle $s|_{\mathrm{R}'_{\hat{\sigma},U}} \colon \mathrm{R}'_{\hat{\sigma},U} \rightarrow O$ can be extended trivially to a continuous Banach subbundle $s|_E \colon E \rightarrow M$ of $s \colon \mathrm{C}_{p_{/U}}(H/U) \rightarrow L_{\mathrm{fix}}$ (see Lemma \ref{remsub}). But since $O$ is clopen, the extended bundle is still locally trivial and hence $\Gamma_{E} \subset \mathrm{C}(L \times_{L_{\fix}} H/U)$ is a finitely generated, projective, and closed $\mathrm{C}(L)$-submodule of $\mathrm{C}(L \times_{L_{\fix}} H/U)$ with $\mathrm{R}'_{\hat{\sigma}(m),U_{m}} = E_m = \{f_l \mid f \in \Gamma_E\}$ by Lemma \ref{lemmareldiscsp}. It is easy to see that $\Gamma_E$ is also $T_{\varphi_{c}}$-invariant since $\mathrm{R}'_{\hat{\sigma}(m'),U_{m'}}$ is invariant with respect to the left regular representation of $H_{m'}$ on $\mathrm{C}(H_{m'}/U_{m'})$ for every $m' \in O$. Thus, $\mathrm{R}'_{\hat{\sigma}(m),U_{m}} \subset \Gamma_l$ as desired.
    \end{proof}
   
    \begin{remark}\label{remarkgrpext}
        We note that for the proof of Theorem \ref{reldiscsp} it is actually enough to assume that $L_{\fix}$ is totally disconnected.
    \end{remark}
    
    Our goal is now to prove a converse result of Theorem \ref{reldiscsp} which represents a given extension with relative discrete spectrum over a Stonean space as a relative skew-product system. To do so, we first discuss some further concepts.

\subsection{Uniform enveloping semigroupoids}\label{uniformenv}
    
    As a key tool to study structured extensions in topological dynamics we recall the concept of uniform enveloping semigroupoid introduced in \cite{EdKr2021}. Semigroupoids generalize semigroups in that the multiplication is only partially defined (see \cite[Definition 1.1]{EdKr2021}). 
    \begin{definition}
        A \emph{semigroupoid} consists of a set $\euS$, a subset $\euS^{(2)} \subset \euS \times \euS$ and a map $\cdot \colon \euS^{(2)} \rightarrow \euS$ which is associative in the sense that if $(\vartheta_1,\vartheta_2), (\vartheta_2, \vartheta_3) \in \euS^{(2)}$, then also $(\vartheta_1 \cdot \vartheta_2,\vartheta_3)$, $(\vartheta_1, \vartheta_2 \cdot \vartheta_3) \in \euS^{(2)}$ and $(\vartheta_1 \cdot \vartheta_2) \cdot \vartheta_3 = \vartheta_1 \cdot (\vartheta_2 \cdot \vartheta_3)$.
        
        If $\euS$ is equipped with a topology such that the map $\cdot \colon \euS^{(2)} \rightarrow \euS$ is continuous, then the semigroupoid is called a \emph{topological semigroupoid}.
    \end{definition}
    We refer to \cite[Subsection 2.1.1]{2013MitreaMonniaux} for an introduction as well as basic properties and simple examples of semigroupoids. Here, we only need the following observation (cf. \cite[Remark 2.2 and Section 3]{EdKr2021}): If $q\colon K \rightarrow L$ is an open continuous surjection between compact spaces, then the space of continuous fiber maps $\mathrm{C}_q^q(K,K)$ equipped with the compact-open topology, the set of composable pairs
        \begin{align*}
            \mathrm{C}_q^q(K,K)^{(2)} \coloneqq \{(\vartheta_1,\vartheta_2) \in  \mathrm{C}_q^q(K,K)^2 \mid r(\vartheta_2) = s(\vartheta_1)\}
        \end{align*}
    as well as the multiplication 
        \begin{align*}
            \cdot \,\colon \mathrm{C}_q^q(K,K)^{(2)} \rightarrow \mathrm{C}_q^q(K,K), \quad (\vartheta_1, \vartheta_2) \mapsto \vartheta_1 \circ \vartheta_2
        \end{align*}
    is a Hausdorff topological semigroupoid.
    
    Based on the classical enveloping semigroup introduced by Ellis (cf. \cite{Ellis1960}), we define the uniform enveloping semigroupoid of an extension in the following way (see \cite[Definition 3.3]{EdKr2021}).
	\begin{definition}\label{defuniform}
		For an open extension $q \colon (K;\varphi) \rightarrow (L;\psi)$ consider the set
			\begin{align*}
				\euS(q) \coloneqq \{\varphi_t|_{K_l} \colon K_l \rightarrow K_{\psi_t(l)}\mid t \in G, l \in L\} \subset \mathrm{C}_q^q(K,K).
			\end{align*}
		The smallest closed subsemigroupoid of $\mathrm{C}_q^q(K,K)$ containing $\euS(q)$ is called the \emph{uniform enveloping semigroupoid} and denoted by $\euE_{\mathrm{u}}(q)$.
	\end{definition}
	\begin{remark}
	    It should be pointed out that in Definition \ref{defuniform} $\euE_{\mathrm{u}}(q)$ is not always given by the closure of $\euS(q)$ as this might not be a subsemigroupoid (see \cite[Remark 3.9 and Example 3.10]{EdKr2021}).
	\end{remark}
	In the situation of extensions with relative discrete spectrum, the next result shows that the uniform enveloping semigroupoid is a actually a compact \emph{groupoid} (ee \cite[Definition 1.1]{EdKr2021} or \cite[Subsection 2.1.2]{2013MitreaMonniaux} for the definition of groupoids). It is a consequence of an Arzelà-Ascoli type theorem (see the implication \enquote{(d) $\Rightarrow$ (a)} in \cite[Theorem 4.14]{EdKr2021} which actually does not rely on topological ergodicity of $(L;\psi)$).
		\begin{proposition}
			Assume that $q \colon (K;\varphi) \rightarrow (L;\psi)$ is an open extension such that $\mathrm{C}(K)$ is generated as a C*-algebra by its finitely generated, projective, closed, $T_{\varphi}$-invariant $\mathrm{C}(L)$-submodules. Then $\euE_{\mathrm{u}}(q)$ is a compact groupoid, i.e., 
			    \begin{enumerate}[(i)]
			        \item $\euE_{\mathrm{u}}(q)$ is a compact space,
			        \item every $\vartheta \in \euE_{\mathrm{u}}(q)$ is bijective with inverse $\vartheta^{-1} \in \euE_{\mathrm{u}}(q)$, and
			        \item the map
			            \begin{align*}
			                \euE_{\mathrm{u}}(q) \rightarrow \euE_{\mathrm{u}}(q), \quad \vartheta \mapsto \vartheta^{-1}
			            \end{align*}
			           is continuous.
			    \end{enumerate}
		\end{proposition}
	
	We compute the uniform enveloping semigroupoid explicitly in an example (see \cite[Example 3.7]{EdKr2021}).
	    \begin{example}\label{skewtorus2}
	        Let $G = \Z$ and $q \colon (\T^2;\varphi) \rightarrow (\T;\psi)$ be the skew-torus extension from Example \ref{skewprod} defined by an element $a \in \T$ which is not a root of unity. Then the uniform enveloping semigroupoid $\euE_{\mathrm{u}}(q)$ consists of all \enquote{rotations between fibers}
	        \begin{align*}
	            \vartheta_{x_1,x_2,b} \colon \{x_1\} \times \T \rightarrow \{x_2\} \times \T,\quad (x_1,y) \mapsto (x_2,by),
	        \end{align*}
	       where $x_1,x_2,b \in \T$. In fact, it is easy to check that 
            \begin{align*}
                \euE \coloneqq \{ \vartheta_{x_1,x_2,b}\mid x_1,x_2,b \in \T\} \subseteq \mathrm{C}_q^q(\T^2,\T^2)
            \end{align*}
        is a closed subsemigroupoid of $\mathrm{C}_q^q(\T^2,\T^2)$ containing $\euS(q)$, and thus $\euE_{\mathrm{u}}(q) \subseteq \euE$. On the other hand, for $x_1,x_2, b \in \T$ we find by minimality of the skew-torus (see, e.g., \cite[Corollaries 10.9 and 10.18]{EFHN2015}) a sequence $(n_{k})_{k \in \N}$ in $\Z$ with 
            \begin{align*}
                \lim_{k \rightarrow \infty} (a^{n_k}x_1,a^{\frac{n_k(n_k-1)}{2}}x_1^{n_k}) = \lim_{k \rightarrow \infty} \varphi_{n_k}(x_1,1) = (x_2,b).
            \end{align*}
        But this implies $\vartheta_{x_1,x_2,b} = \lim_{k \rightarrow \infty} \varphi_{n_k}|_{\{x_1\} \times \T}$ in $\mathrm{C}_q^q(\T^2,\T^2)$, and hence $\vartheta_{x_1,x_2,b} \in \euE_{\mathrm{u}}(q)$.
	   \end{example}
	Thus, in Example \ref{skewtorus2} the uniform enveloping semigroupoid is in one-to-one correspondence with the set $\T \times \T \times \T$. This is reflected by an irreducibility condition of the corresponding cocycle
	     \begin{align*}
	        c \colon \Z \times \T \rightarrow H, \quad (k,x) \mapsto a^{\frac{k(k-1)}{2}}x^k.
	    \end{align*}
	We make this precise in the following definition. Note here that if $p\colon L \rightarrow M$ is a continuous surjection between compact spaces and $H$ is a compact group bundle over $M$, the compact space $(L \times_M L) \times_M H$ canonically becomes a compact topological semigroupoid by setting 
	    \begin{align*}
            ((L \times_M L) \times_M H)^{(2)} \coloneqq \{((l,l_2,x),(l_1,l,y))\mid l,l_1,l_2 \in L, x,y \in H\}
        \end{align*}
    and $(l,l_2,x) \cdot (l_1,l,y) \coloneqq (l_1,l_2,xy)$ for $((l,l_2,x),(l_1,l,y)) \in  ((L \times_{M} L) \times_{M} H)^{(2)}$. It is even a compact groupoid with inverses $(l_1,l_2,x)^{-1} \coloneqq (l_2,l_1,x^{-1})$ for $(l_1,l_2,x) \in (L \times_{M} L) \times_{M} H$.

\begin{definition}
    Let $(L;\psi)$ be a topological dynamical system and $p \colon H \rightarrow L_{\fix}$ a compact group bundle over its fixed factor $L_{\mathrm{fix}}$. A relative cocycle $c \colon G \times T \rightarrow H$ is \emph{ergodic} if the smallest closed subsemigroupoid of $(L \times_{L_{\fix}} L) \times_{L_{\fix}} H$ containing
        \begin{align*}
            \{(l,\psi_t(l),c(t,l))\mid l \in L, t \in G\}
        \end{align*}
    is already the whole space $(L \times_{L_{\fix}} L) \times_{L_{\fix}} H$.
\end{definition}

For certain relative skew-product extensions induced by ergodic relative cocycles we now obtain a description of the uniform enveloping semigroupoid which is analoguous to the one of Example \ref{skewtorus2}. Call a closed subgroup bundle $p|_U \colon U \rightarrow M$ of a compact group bundle $p\colon H \rightarrow M$ \emph{core-free} if $U_m$ is a core-free subgroup of $H_m$ for $m \in M$, i.e., the only normal subgroup of $H_m$ contained in $U_m$ is the trivial one. In the case of an abelian group $H_m$ this means that $U_m$ is itself trivial. Moreover, define homomorphisms and isomorphisms of topological semigroupoids and groupoids  in the obvious way, see \cite[Definition 1.7]{EdKr2021}.

For relative skew-product extensions we obtain the following consequence.
\begin{proposition}\label{uniformcomp}
    Let 
        \begin{enumerate}[(i)]
            \item $(L;\psi)$ be a topological dynamical system,
            \item $p \colon H \rightarrow L_{\fix}$ an open compact group bundle,
            \item $p|_U\colon U \rightarrow L_{\fix}$ a closed subgroup bundle, and
            \item $c \colon G \times L \rightarrow H$ a relative cocycle.
        \end{enumerate}
    Then
        \begin{align*}
            \Phi \colon  (L \times_{L_{\fix}} L) \times_{L_{\fix}} H \rightarrow \mathrm{C}_{q_c}^{q_c}(L \times_{L_{\fix}} H/U,L \times_{L_{\fix}} H/U), \quad (l_1,l_2,x) \mapsto \vartheta_{l_1,l_2,x}
        \end{align*}
    with $\vartheta_{l_1,l_2,x}(l_1,yU) \coloneqq (l_2,xyU)$ for $l_1,l_2 \in L$ and $x,y \in H$ is a homomorphism of topological semigroupoids with  $\euE_{\mathrm{u}}(q_c) \subseteq \mathrm{ran}(\Phi)$. Moreover, the following assertions hold.
        \begin{enumerate}[(1)]
            \item If $p|_U\colon U \rightarrow L_{\fix}$ is core-free, then $\Phi$ is injective.
            \item If $c$ is ergodic, then $\euE_{\mathrm{u}}(q_c) = \mathrm{ran}(\Phi)$.
        \end{enumerate}
    In particular, if $p|_U$ is core-free and $c$ is ergodic, then $\Phi$ defines an isomorphism between the compact groupoids $(L \times_{L_{\fix}} L) \times_{L_{\fix}} H$ and $\euE_{\mathrm{u}}(q_c)$.
\end{proposition}
\begin{proof}
    The map
        \begin{align*}
            \Phi \colon  (L \times_{L_{\fix}} L) \times_{L_{\fix}} H \rightarrow \mathrm{C}_{q_c}^{q_c}(L \times_{L_{\fix}} H/U,L \times_{L_{\fix}} H/U), \quad (l_1,l_2,x) \mapsto \vartheta_{l_1,l_2,x}
        \end{align*}
    is a homomorphism of semigroupoids. Continuity readily follows from Lemma \ref{remarkshomspace} and the characterization of convergence in $\mathrm{C}_{q_c}^{q_c}(L \times_{L_{\fix}} H/U,L \times_{L_{\fix}} H/U)$ of Proposition \ref{descriptiontop}. The image of $\Phi$ is then a compact, hence closed subsemigroupoid of $\mathrm{C}_{q_c}^{q_c}(L \times_{L_{\fix}} H/U,L \times_{L_{\fix}} H/U)$ containing $\euS(q_c)$ and therefore $\euE_{\mathrm{u}}(q_c) \subseteq \mathrm{ran}(\Phi)$ by the definition of the uniform enveloping semigroupoid.
    
    For (1) assume that $p|_U$ is core-free.  Pick $(l_1,l_2,x), (l_1,l_2,y) \in (L \times_{L_{\fix}} L) \times_{L_{\fix}} H$ with $\vartheta_{l_1,l_2,x} = \vartheta_{l_1,l_2,y}$. Then $(y^{-1}x)zU_{q_{\fix}(l_1)} = zU_{q_{\fix}(l_1)}$ for every $z \in H_{q_{\fix}(l_1)}$. But since $U_{q_{\fix}(l_1)}$ is a core-free subgroup of $H_{q_{\fix}(l_1)}$, the group $H_{q_{\fix}(l_1)}$ acts effectively on the quotient $H_{q_{\fix}(l_1)}/U_{q_{\fix}(l_1)}$. This shows $x = y$ and hence $\Phi$ is injective.  
    
    To show (2) we assume that $c$ is ergodic. Then the set $\euS \coloneqq \{(l_1,l_2,x)\mid \vartheta_{l_1,l_2,x} \in \euE_{\mathrm{u}}(q_c)\}$ is a compact, hence closed subsemigroupoid of $(L \times_{L_{\fix}} L) \times_{L_{\fix}} H$
    containing the set $\{(l,\psi_t(l),c(t,l))\mid l \in L, t \in G\}$. Since $c$ is ergodic, we obtain $\euS = (L \times_{L_{\fix}} L) \times_{L_{\fix}} H$. 
  
    Finally, if $p|_U$ is core-free and $c$ is ergodic, then $\Phi  \colon  (L \times_{L_{\fix}} L) \times_{L_{\fix}} H \rightarrow \euE_{\mathrm{u}}(q_c)$ is a bijective continuous map between compact spaces, and thus a homeomorphism. This implies that is an isomorphism of compact groupoids in that case.
\end{proof}
	
	In Example \ref{skewtorus2} the uniform enveloping (semi)groupoid acts transitively on $K = \T^2$, i.e., every point of $K$ can be reached from any other point in $K$ via a map in $\euE_{\mathrm{u}}(q)$. The following result gives a general description of the orbits of the action of the uniform enveloping semigroupoid, provided that it is a compact groupoid. A proof can be found in \cite[Lemma 5.10 (ii)]{EdKr2021} (the cited result covers the case of pseudoisometric extensions, but its proof only uses the fact that the uniform enveloping semigroupoid is a compact groupoid).
	\begin{proposition}\label{charergodicity}
	    	Take an open extension $q \colon (K;\varphi) \rightarrow (L;\psi)$ of topological dynamical systems and write $q_{\fix} \colon K \rightarrow K_{\fix}$ for the factor map to the fixed factor $K_{\fix}$ of $(K;\varphi)$. If $\euE_{\mathrm{u}}(q)$ is a compact groupoid, then
	    	    \begin{align*}
	    	        q_{\fix}^{-1}(q_{\fix}(x)) = \{\vartheta(x)\mid \vartheta \in \euE_{\mathrm{u}}(q) \textrm{ with } s(\vartheta) = q(x)\}
	    	    \end{align*}
	        for every $x \in K$.
	\end{proposition}
	We only need the following consequences for topologically ergodic extensions.
	\begin{corollary}\label{charergodicity2}
	    Let $q \colon (K;\varphi) \rightarrow (L;\psi)$ be an open topologically ergodic extension of topological dynamical systems such that $\euE_{\mathrm{u}}(q)$ is a compact groupoid.
	        \begin{enumerate}[(i)]
	            \item For $l_1, l_2 \in L$ the following assertions are equivalent.\label{charergodicity21}
	        \begin{enumerate}[(a)]
	            \item $q_{\fix}(l_1) = q_{\fix}(l_2)$.\label{charergodicity21a}
	            \item There is $\vartheta \in \euE_{\mathrm{u}}(q)$ with $s(\vartheta) = l_1$ and $r(\vartheta) = l_2$.\label{charergodicity21b}
	       \end{enumerate}
	        \item For $x,y \in K$ with $q(x) = q(y)$ there is $\vartheta \in \euE_{\mathrm{u}}(q)$ with $s(\vartheta) = r(\vartheta) = q(x) = q(y)$ and $\vartheta(x) = y$.\label{charergodicity22}
	    \end{enumerate}
	\end{corollary}
	\begin{proof}  
	    For part (\ref{charergodicity21}) take $l_1,l_2 \in L$. If (a)  holds, take $x_i \in K$ with $q(x_i) = l_i$ for $i =1,2$. Then also $q_{\fix}(x_1) = q_{\fix}(x_2)$ by topological ergodicty of $q$ and by Proposition \ref{charergodicity} we find $\vartheta \in \euE_{\mathrm{u}}(q)$ with $s(\vartheta) = q(x_1) = l_1$ and $\vartheta(x_1) = x_2$, hence $r(\vartheta) = q(x_2) = l_2$. Conversely, if (b) holds, we pick any $x \in K_{l_1}$ and an element $\vartheta \in \euE_{\mathrm{u}}(q)$ as in (b). Then $q_{\fix}(x) = q_{\fix}(\vartheta(x))$ by Proposition \ref{charergodicity} and thus also 
	        \begin{align*}
	            q_{\fix}(l_1) = q_{\fix}(q(x)) = q_{\fix}(q(\vartheta(x))) = q_{\fix}(r(\vartheta)) = q_{\fix}(l_2).
	       \end{align*}
	   
	   For (ii) observe that if $x,y \in K$ with $q(x) = q(y)$, then also $q_{\fix}(x) = q_{\fix}(y)$ by topological ergodicity of $q$. Hence, Proposition \ref{charergodicity} yields the claim.
	\end{proof}

\subsection{Relatively invariant measures}\label{RelInvMeas}
We have already discussed relative measures in Subsection \ref{secrelmeas}. We now recall the dynamical version of this concept (cf. \cite{Glas1975}, \cite[Subsection V.3.19]{deVr1993} and \cite[Section 5]{EdKr2021}).

	\begin{definition}
	    Let $q \colon (K;\varphi) \rightarrow (L;\psi)$ be an extension of topological dynamical systems. A relative measure $\nu \in \mathrm{RM}(q)$ is a \emph{relatively invariant measure}\footnote{In our terminology, \enquote{equivariant relative measure} would be more fitting, but the term \enquote{relatively invariant measure} is standard in topological dynamics.} for $q$ if it is equivariant, i.e., if $(\varphi_t)_*\nu(l) = \nu(\psi_t(l))$ for every $l \in L$ and every $t \in G$.
	    We denote the \emph{space of all relatively invariant measures} of $q$ by $\mathrm{RIM}(q)$.
	\end{definition}
	\begin{remark}
        For open pseudoisometric extensions between topologically ergodic systems there is always a unique relatively invariant measure and this is fully supported (see \cite[Theorem 7.3]{EdKr2021}). We will obtain the same statement for open topologically ergodic extensions with relative discrete spectrum over Stonean spaces in Corollary \ref{col-uniquerel} below.
    \end{remark}

    The following is a dynamical version of Proposition \ref{disintegration}. Here, for a topological dynamical system $(K;\varphi)$, we denote by $\mathrm{P}^\varphi(K) \subset \mathrm{P}(K)$ the closed subspace of measures $\mu \in \mathrm{P}(K)$ which are invariant with respect to $\varphi$, i.e., which satisfy $(\varphi_t)_*\mu = \mu$ for every $t \in G$.
    	\begin{proposition}\label{dynamicaldisint}
	        Let $q \colon (K;\varphi) \rightarrow (L;\psi)$ be an extension of topological dynamical systems and $\mu_L \in \mathrm{P}^\psi(L)$ a fully supported invariant measure. The canonical map
				\begin{align*}
					\int_L \mathrm{d}\mu_L \colon \mathrm{RIM}(q) \rightarrow \{\mu_K \in \mathrm{P}^{\varphi}(K)\mid q_*\mu_K = \mu_L\}, \quad \nu \mapsto \int_L \nu\,\mathrm{d}\mu_L
				\end{align*}
			with $(\int_L \nu\,\mathrm{d}\mu_L)(f) \coloneqq \int_L  \nu_l(f) \, \mathrm{d}\mu_L(l)$ for $f \in \mathrm{C}(K)$ and $\nu \in \mathrm{RIM}(q)$ is
			\begin{enumerate}[(i)]
				\item injective, and
				\item bijective if $\mu_L$ is normal.
			\end{enumerate}
	\end{proposition}
	\begin{proof}
	    It is clear that the map is well-defined, i.e., $\int_L \nu\,\mathrm{d}\mu_L \in \mathrm{P}^{\varphi}(K)$ for every $\nu \in \mathrm{RIM}(q)$. Moreover, it is injective by Proposition \ref{disintegration} (\ref{disintegration1}). Now, if $\mu_L$ is normal and $\mu_K \in \mathrm{P}^{\varphi}(K)$,  we find $\nu \in \mathrm{RM}(q)$ with $\int_L \nu\,\mathrm{d}\mu_L = \mu_K$ by Proposition \ref{disintegration} (\ref{disintegration2}). For $t \in G$ we observe that
	        \begin{align*}
	            \varrho \colon L \rightarrow \mathrm{P}_q(K), \quad l \mapsto (\varphi_{-t})_*\nu_{\psi_t(l)}
	        \end{align*}
	   is also a relative measure with
	   \begin{align*}
	        \left(\int_L \varrho\,\mathrm{d}\mu_L\right)(f) &= \int_L\nu_{\psi_t(l)}(f \circ \varphi_{-t}) \,\mathrm{d}\mu_L (l) = \int_L\nu_{l}(f \circ \varphi_{-t}) \,\mathrm{d}\mu_L (l) \\
	        &= \mu_K((T_\varphi)_t f) = \mu_K(f)
	   \end{align*}
	   for every $f \in \mathrm{C}(K)$. Again by Proposition \ref{disintegration} (\ref{disintegration1}) we obtain $\mu = \tau$ which shows $\mu \in \mathrm{RIM}(q)$.
	\end{proof}
 We now discuss the existence and uniqueness of relatively invariant measures for relative-skew products. Here and in the following we write $\lambda_{H/U}$ for the normalized Haar measure of the homogoneous space $H/U$ of a compact group $H$ with respect to a closed subgroup $U$.

\begin{theorem}\label{uniquerel}
    Let 
        \begin{enumerate}[(i)]
            \item $(L;\psi)$ be a topological dynamical system,
            \item $p \colon H \rightarrow L_{\fix}$ an open compact group bundle,
            \item $p|_U \colon U \rightarrow L_{\fix}$ a closed subgroup bundle of $p$, and
            \item $c\colon G \times L \rightarrow H$ a relative cocycle.
        \end{enumerate}
        Then
            \begin{align*}
                \delta_L \otimes \lambda_{H/U}\colon L \rightarrow \mathrm{P}_{q_c}(L \times_{L_{\fix}} H/U), \quad l \mapsto \delta_l \otimes \lambda_{H_{q_{\fix}(l)}/U_{q_{\fix}(l)}}
            \end{align*}
        with 
            \begin{align*}
                (\delta_l \otimes \lambda_{H_{q_{\fix}(l)}/U_{q_{\fix}(l)}})(f) = \int_{H_{q_{\fix}(l)}/U_{q_{\fix}(l)}} f(l,xU_{q_{\fix}(l)}) \, \mathrm{d}\lambda_{H_{q_{\fix}(l)}/U_{q_{\fix}(l)}}(xU_{q_{\fix}(l)})
            \end{align*}
        for each $f \in \mathrm{C}(L \times_{L_{\fix}} H/U)$ and $l \in L$ is a relatively invariant measure of $q_c$ with full support. If $p|_U$ is core-free and $c$ is ergodic, then $\delta_L \otimes \lambda_{H/U}$ is the only relatively invariant measure of $q_c$.
\end{theorem}



\begin{proof}
     To see that $\delta_L \otimes \lambda_{H/U}$ is continuous, it suffices, by the Stone-Weierstrass theorem, to prove continuity of the functions
        \begin{align*}
            L \rightarrow \C, \quad l \mapsto (\delta_l \otimes \lambda_{H_{q_{\fix}(l)}/U_{q_{\fix}(l)}})(g \otimes h)
        \end{align*}
    with $g \in \mathrm{C}(L)$ and $h \in \mathrm{C}(H/U)$, where $(g \otimes h)(l,xU) \coloneqq g(l) \cdot h(xU)$ for all $(l,xU) \in L \times_{L_{\fix}} H/U$. But this is a direct consequence of the continuity of the Haar system (see Proposition \ref{Renault}). It is clear that $\delta_L \otimes \lambda_{H/U}$ is equivariant, hence a relatively invariant measure, and that it has full support. 

    Now assume that $p|_U$ is core-free and $c$ is ergodic, and take any relatively invariant measure $\nu \in \mathrm{RIM}(q_c)$. The set
        \begin{align*}
            \euS \coloneqq \{\vartheta \in \euE_{\mathrm{u}}(q)\mid \vartheta_*\nu_{s(\vartheta)} = \nu_{r(\vartheta)}\}
        \end{align*}
    is a closed subsemigroupoid of $\euE_{\mathrm{u}}(q)$ containing $\euS(q)$ by Corollary \ref{convlemma}, hence $\euS = \euE_{\mathrm{u}}(q)$ (see also \cite[Lemma 5.13]{EdKr2021}). With the notation of Proposition \ref{uniformcomp}, we obtain in particular for each $l \in L$ that $(\vartheta_{l,l,x})_{*}\nu_{l} = \nu_l$ in $\mathrm{P}(\{l\} \times H_{q_{\fix}(l)}/U_{q_{\fix}(l)})$ for every $x \in H_{q_{\fix}(l)}$ which implies $\nu_l = \delta_l \otimes \lambda_{H_m/U_m}$ by uniqueness of the Haar measure.
\end{proof}

\subsection{A topological Mackey--Zimmer type representation theorem}\label{topmz}
We are now in the position to prove a Mackey--Zimmer type representation result for topologically ergodic, open extensions with relative discrete spectrum over a Stonean space.

\begin{theorem}\label{reptheorem1}
    Let $q \colon (K;\varphi) \rightarrow (L;\psi)$ be an open, topologically ergodic extension with relative discrete spectrum over a Stonean space $L$. Then there is 
        \begin{enumerate}[(i)]
            \item an open compact group bundle $p \colon H \rightarrow L_{\mathrm{fix}}$,
            \item a core-free closed subgroup bundle $p|_U \colon U \rightarrow L_{\fix}$, and
            \item an ergodic relative cocycle $c \colon G \times L \rightarrow H$,
        \end{enumerate}
        such that $q$ is isomorphic to the induced relative skew-product extension $q_c \colon (L \times_{\fix} H/U;\varphi_c) \rightarrow (L;\psi)$.
\end{theorem}
In view of Theorem \ref{uniquerel}, we obtain the following consequence of Theorem \ref{reptheorem1}.

\begin{corollary}\label{col-uniquerel}
    Every open topologically ergodic extension $q \colon (K;\varphi) \rightarrow (L;\psi)$ with relative discrete spectrum over a Stonean space $L$ has a unique relatively invariant measure and this is fully supported.
\end{corollary}

If we assume that the systems are topologically ergodic (instead of the extension), then the representation result simplifies considerably.
\begin{corollary}\label{corold}
    Let $q \colon (K;\varphi) \rightarrow (L;\psi)$ be an open extension with relative discrete spectrum between topologically ergodic systems with $L$ Stonean. Then there is 
        \begin{enumerate}[(i)]
            \item a compact group $H$,
            \item a core-free closed subgroup $U \subset H$, and
            \item an ergodic cocycle $c \colon G \times L \rightarrow H$,
        \end{enumerate}
        such that $q$ is isomorphic to the induced skew-product extension $q_c \colon (L \times H/U;\varphi_c) \rightarrow (L;\psi)$.
\end{corollary}

An important ingredient of the proof of Theorem \ref{reptheorem1} is the following classical result of Gleason (see \cite{gleason}).
\begin{theorem}\label{gleason}
    Every continuous surjection from a compact space to a Stonean space has a continuous section.
\end{theorem}

As a first step towards the proof of Theorem \ref{reptheorem1}, we prove that under certain assumptions the existence of an ergodic relative cocycle already forces the corresponding compact group bundle to be open.
	
\begin{lemma}\label{continuousHaar}
    Consider a topological dynamical system $(L;\psi)$ and a compact group bundle $p \colon H \rightarrow L_{\fix}$. Assume that 
        \begin{enumerate}[(i)]
            \item there is an ergodic relative cocycle $c \colon G \times L \rightarrow H$,\label{continuousHaar1}
            \item $q_{\fix} \colon L \rightarrow L_{\fix}$ is open, and\label{continuousHaar2}
            \item $L_{\fix}$ is Stonean.\label{continuousHaar3}
        \end{enumerate}
    Then $p$ is open.
\end{lemma}

\begin{proof}
    We prove that the Haar measures $\lambda_{H_m}$ for $m \in L_{\fix}$ define a relative measure for $p$. By Proposition \ref{Renault} this is equivalent to $p$ being open.
    
    For a fiber measure $\mu \in \mathrm{P}_p(H)$ and $x \in H_{p_*(\mu)}$ we write $x_*\mu$ for the push-forward of $\mu$ with respect to the left rotation $H_{p_*(\mu)} \rightarrow H_{p_*(\mu)}, \, y \mapsto xy$. We first show that 
        \begin{align*}
            C \coloneqq \{\mu \in \mathrm{P}_p(H)\mid c(t,l)_*\mu = \mu \textrm{ for all } (t,l) \in G \times L \textrm{ with } q_{\fix}(l) = p_*(\mu)\}.
        \end{align*}
   is a closed subset of $\mathrm{P}_p(H)$. So pick a net $(\mu_{\alpha})_{\alpha \in A}$ in $C$ converging to $\mu \in \mathrm{P}_p(H)$. We abbreviate $m_\alpha \coloneqq p_*(\mu_\alpha) \in L_{\fix}$ for $\alpha \in A$ and $m \coloneqq p_*(\mu) \in L_{\fix}$. Let $t \in G$ and $l \in L$ with $q_{\fix}(l) = m$. We have to show that $(c(t,l)_*\mu)(f) = \mu(f)$ for every $f \in \mathrm{C}(H)$.
   
   Since $q_{\fix}$ is open,  we may assume, by passing to a subnet, that there is a net $(l_\alpha)_{\alpha \in A}$ in $L$ with $q_{\fix}(l_\alpha) = m_\alpha$ for every $\alpha \in A$ and $\lim_{\alpha} l_\alpha = l$. Then $\lim_\alpha c(t,l_\alpha) = c(t,l)$ since $c$ is continuous. Now if $f \in \mathrm{C}(H)$,  consider the net $(g_\alpha)_{\alpha \in A}$ in $\mathrm{C}_p(H)$ defined by $g_\alpha(x) \coloneqq f(c(t,l_\alpha)x)$ for $x \in H_{m_\alpha}$ and $\alpha \in A$. Then $\lim_\alpha g_\alpha = g$ in $\mathrm{C}_p(H)$, where $g \in \mathrm{C}(H_{m})$ is given by $g(x) \coloneqq f(c(t,l)x)$ for $x\in H_m$. By Corollary \ref{convlemma}, 
        \begin{align*}
             (c(t,l)_*\mu)(f) &= \mu(g) = \lim_\alpha \mu_\alpha(g_\alpha) = \lim_{\alpha}(c(t_\alpha,l_\alpha)_*\mu_\alpha)(f) = \lim_{\alpha} \mu_\alpha(f)
             = \mu(f).
        \end{align*}
    
    We conclude that $C$ is indeed closed in $\mathrm{P}_p(H)$ and hence compact. The Haar measure $\lambda_{H_m}$ of $H_m$ is contained in $C$ for every $m \in L_{\fix}$. Hence the restriction $p_*|_{C} \colon C \rightarrow L_{\fix}$ is a continuous surjection. By Theorem \ref{gleason} we find a continuous section $\nu \colon L_{\fix} \rightarrow C, \, m \mapsto \nu_m$ of $p_*|_{C}$. In particular, $\nu$ is a relative measure. Applying Corollary \ref{convlemma} once again, the set
        \begin{align*}
            \euS \coloneqq \{(l_1,l_2,x) \in (L \times_{L_{\fix}} L) \times_{L_{\fix}} H\mid x_*\nu_{q_{\fix}(l_1)} = \nu_{q_{\fix}(l_1)}\}
        \end{align*}
    is a closed subsemigroupoid of $(L \times_{L_{\fix}} L) \times_{L_{\fix}} H$. Since $c$ is ergodic, we obtain $\euS = (L \times_{L_{\fix}} L) \times_{L_{\fix}} H$. But this shows $x_*\nu_m =\nu_m$ for every $x \in H_m$ and $m \in L_{\fix}$. Thus, $\nu_m = \lambda_{H_m}$ for every $m \in L_{\fix}$.
\end{proof}
Note that by Lemma \ref{lemstonean} assumptions (\ref{continuousHaar2}) and (\ref{continuousHaar3}) of Lemma \ref{continuousHaar} are automatically satisfied if $L$ itself is a Stonean space.\medskip\\
Before showing the general version of Theorem \ref{reptheorem1}, let us briefly sketch the proof in the case of topologically ergodic systems considered in Corollary \ref{corold}. So assume that $q \colon (K;\varphi) \rightarrow (L;\psi)$ is an open extension with relative discrete spectrum between topologically ergodic dynamical systems and assume that $L$ is a Stonean space. Then, if we choose any point $x_0 \in K$ and write $l_0 \coloneqq q(x)$, the compact group
    \begin{align*}
        H \coloneqq \{\vartheta\in \euE_{\mathrm{u}}(q)\mid s(\vartheta) = r(\vartheta) = l_0\}
    \end{align*}
acts transitively on the fiber $K_l$ by Corollary \ref{charergodicity2} (\ref{charergodicity22}), and hence gives rise to a homeomorphism $H/U \rightarrow K_{l_0}$, where $U$ is the stabilizer group at $x_0$ (which is core-free since the action is effective).

The uniform enveloping semigroupoid $\euE_{\mathrm{u}}(q)$ also allows us to \enquote{switch between fibers}: By  Corollary \ref{charergodicity2} (\ref{charergodicity21}) we find for every point $l \in L$ some $\vartheta \in \euE_{\mathrm{u}}(q)$ with $s(\vartheta) = l_0$ and $r(\vartheta) = l$, i.e., the restriction of the range map $r \colon \euE_{\mathrm{u}}(q) \rightarrow M$ to the closed subspace
    \begin{align*}
        C \coloneqq \{\vartheta \in \euE_{\mathrm{u}}(q)\mid s(\vartheta) =l_0\}
    \end{align*}
is surjective and therefore admits a continuous section $\varrho \colon L \rightarrow C$ by Theorem \ref{gleason}. We then obtain a homeomorphism
    \begin{align*}
        \Phi \colon L \times H/U \rightarrow K, \quad (l,\vartheta U) \mapsto \varrho(l)(\vartheta(x_0)).
    \end{align*}
By setting $c(t,l) \coloneqq \varrho(\psi_t(l))^{-1} \circ \varphi_t \circ \varrho(l) \in H$ for $(t,l) \in G \times L$ we obtain a cocycle $c \colon G \times L \rightarrow H$ and one can check that $\Phi$ then defines an isomorphism between the induced skew-product extension $q_c$ and $q$ and that $c$ is ergodic.\medskip\\
In the general situation of a topologically ergodic extension, we have to perform the constructions relative to the fixed factor. This will now be done in detail.   

\begin{proof}[Proof of Theorem \ref{reptheorem1}.]
    Using Theorem \ref{gleason} and Lemma \ref{lemstonean} we find continuous sections $\sigma \colon L \rightarrow K$ and $\tau \colon L_{\mathrm{fix}} \rightarrow L$ for $q \colon K \rightarrow L$ and $q_{\fix} \colon L \rightarrow L_{\fix}$, respectively. The set
		\begin{align*}
			C \coloneqq \{\vartheta \in \euE_{\mathrm{u}}(q)\mid s(\vartheta) = \tau(q_{\mathrm{fix}}(r(\vartheta)))\} \subset  \euE_{\mathrm{u}}(q)
		\end{align*}
	is closed in $\euE_{\mathrm{u}}(q)$, hence defines a compact subspace. Consider the restricted range map $r|_C \colon C \rightarrow L, \,  \vartheta \mapsto r(\vartheta)$
	and claim that this is surjective. By Corollary \ref{charergodicity2} (\ref{charergodicity21}), for every point $l \in L$, we find some $\vartheta \in \euE_{\mathrm{u}}(q)$ with $s(\vartheta) = \tau(q_{\fix}(l))$ and $r(\vartheta) = l$ such that $q_{\fix}(l) = q_{\fix}(\tau(q_{\fix}(l)))$. Hence, $\vartheta \in C$ with $r(\vartheta) = l$.
    Applying Theorem \ref{gleason} once more, we choose a continuous section $\varrho \colon L \rightarrow M$ for $r|_M \colon M \rightarrow L$.
	
	With the continuous sections $\sigma$, $\tau$ and $\varrho$ at hand, we now construct the desired objects.
	The closed subspace
		\begin{align*}
    		H \coloneqq \{\vartheta \in \euE_{\mathrm{u}}(q)\mid s(\vartheta) = r(\vartheta) \in \tau(L_{\fix})\}
		\end{align*}
	of $\euE_{\mathrm{u}}(q)$ defines a compact group bundle
	    \begin{align*}
	        p \coloneqq q_{\fix} \circ s|_H = q_{\fix} \circ r|_H \colon H \rightarrow L_{\mathrm{fix}}
	    \end{align*}
	over $L_{\fix}$. Moreover,
	    \begin{align*}
			c \colon G \times L \rightarrow H, \quad (t,l) \mapsto \varrho(\psi_t(l))^{-1} \circ \varphi_t \circ \varrho(l)
		\end{align*}
	is a relative cocycle over $(L;\psi)$. Now consider the continuous map
		\begin{align*}
			\Phi \colon L \times_{L_{\mathrm{fix}}} H \rightarrow K, \quad (l,\vartheta) \mapsto (\varrho(l) \circ \vartheta)(\sigma(\tau(q_{\mathrm{fix}}(l))).
		\end{align*}
		
	 Then $\Phi$ defines an morphism of topological dynamical systems as
	    \begin{align*}
	        \Phi(\psi_t(l),c(t,l)\vartheta) &= (\varrho(\psi_t(l)) \circ \varrho(\psi_t(l))^{-1} \circ \varphi_t \circ \varrho(l)\circ  \vartheta)(\sigma(\tau(q_{\mathrm{fix}}(\psi_t(l)))))\\
	        &= (\varphi_t \circ \varrho(l)\circ  \vartheta)(\sigma(\tau(q_{\mathrm{fix}}(l)))) = \varphi_t(\Phi(l,\vartheta))
	    \end{align*}
	for all $(l,\vartheta) \in L \times_{L_{\fix}} H$ and $t \in G$. 
	    
	To show that $\Phi$ is surjective, observe first that $H_m$ acts transitively on the fiber $K_{\tau(m)}$ by Corollary \ref{charergodicity2} (\ref{charergodicity22}) for every $m \in M$. Now take $x \in K$ and set $l \coloneqq q(x)$ and $m \coloneqq q_{\fix}(l)$. Then $\varrho(l)^{-1}(x), \sigma(\tau(m)) \in K_{\tau(m)}$ and we therefore find $\vartheta \in H_m$ with $\vartheta(\sigma(\tau(m))) = \varrho(l)^{-1}(x)$. This means
	    \begin{align*}
	        \Phi(l,\vartheta) = (\varrho(l) \circ \vartheta)(\sigma(\tau(m))) = x.
	    \end{align*}
	Hence, $\Phi$ is a surjection.
	    
    Let $U_m \coloneqq \{\vartheta \in H_m \mid \vartheta(\sigma(\tau(m))) = \sigma(\tau(m))\}$  be the stabilizer group of the group action $H_m \times K_{\tau(m)} \rightarrow K_{\tau(m)}, \, (\vartheta, x) \mapsto \vartheta(x)$ at $\sigma(\tau(m))$ for $m \in L_{\mathrm{fix}}$. Since the action is effective and transitive, the subgroup $U_m$ is core-free in $H_m$ for every $m \in L_{\fix}$. We conclude that the union $U \coloneqq \bigcup_{m \in L_{\fix}} U_m$ defines a core-free closed subgroup bundle $p|_U\colon U \rightarrow M$ of $p$.
        Moreover, for $(l_1,\vartheta_1)$, $(l_2,\vartheta_2) \in L \times_{L_{\fix}} H$ we have $\Phi(l_1,\vartheta_1) = \Phi(l_2,\vartheta_2)$ if and only if $l_1 = l_2$ and $\vartheta_2^{-1} \circ \vartheta_1 \in U_{q_{\fix}(l_1)} = U_{q_{\fix}(l_2)}$. Thus, $\Phi$ induces an isomorphism $(L \times_{L_{\fix}} H/U;\varphi_c) \rightarrow (K;\varphi)$ which then also defines an isomorphism between the extensions $q_c$ and $q$.
    	
    	To finish the proof we show that the relative cocycle $c$ is ergodic which, by Lemma \ref{lemstonean} and Proposition \ref{continuousHaar}, also implies that $p$ is open. Let $\euS$ be the smallest closed subsemigroupoid of $(L \times_{L_{\fix}} L) \times_{L_{\fix}} H$ containing the set $\{(l,\psi_t(l),c(t,l))\mid l \in L, t \in G\}$.
        Then
            \begin{align*}
                \euT \coloneqq \{\vartheta \in \euE_{\mathrm{u}}(q)\mid (s(\vartheta),r(\vartheta),\varrho(r(\vartheta))^{-1} \circ \vartheta \circ \varrho(s(\vartheta)) \in \euS\}
            \end{align*}
        contains $\euS(q)$ by definition of the relative cocycle $c$. Moreover, $\euT$ is a closed subsemigroupoid of $\euE_{\mathrm{u}}(q)$ by definition of $\euS$. We conclude that $\euT = \euE_{\mathrm{u}}(q)$. Now pick $(l_1,l_2,\gamma) \in (L \times_{L_{\fix}} L) \times_{L_{\fix}} H$. Then $\vartheta \coloneqq \varrho(l_1) \circ \gamma \circ \varrho(l_2)^{-1} \in \euE_{\mathrm{u}}(q) = \EuScript{T}$ and hence 
            \begin{align*}
                (l_1,l_2,\gamma) = (s(\vartheta),r(\vartheta),\varrho(r(\vartheta))^{-1} \circ \vartheta \circ \varrho(s(\vartheta)) \in \euS.
            \end{align*}
        Thus $\euS = (L \times_{L_{\fix}} L) \times_{L_{\fix}} H$ showing that $c$ is ergodic.
\end{proof}
\begin{remark}\label{gap}
    The main results of this section are not entirely satisfying: By Theorem \ref{reptheorem1}, every open topologically ergodic extension $q \colon (K;\varphi) \rightarrow (L;\psi)$ with relative discrete spectrum over a Stonean space as a skew-product extension induced by an open compact group bundle $p \colon H \rightarrow L_{\fix}$ and a closed subgroup bundle $p|_{U} \colon U \rightarrow L_{\fix}$. On the other hand, in Theorem \ref{reldiscsp}  we have only shown that relative skew-product extensions associated to \emph{open} closed subgroup bundles $p|_{U} \colon U \rightarrow L_{\fix}$ have relative discrete spectrum. So it is a natural question whether this openness assumption in Theorem \ref{reldiscsp} can be dropped. In fact, there are simple examples showing that openness of $p|_{U} \colon U \rightarrow L_{\fix}$ is not necessary for the skew-product extension to have relative discrete spectrum. However, in the measure-theoretic applications of the next section, this issue does not occur as the closed subgroup bundles in the representation theorem will automatically be open (see Theorem \ref{mz} below).
\end{remark}

\section{The Mackey--Zimmer representation theorem for ergodic extensions}\label{measure}
In this final section, we use the topological results of Section \ref{sectop} to show our version of the Mackey--Zimmer theorem. Again, $G$ will be a fixed group equipped with the discrete topology.

\subsection{Measure-preserving systems and their extensions.}\label{measpres}
We start by recalling some beasic concepts from ergodic theory. Classically, an \emph{invertible measure-preserving point-transformation} is a pair $(\uX;\varphi)$ consisting of a measurable and measure-preserving mapping $\varphi \colon X \rightarrow X$ on a probability space $\uX = (X,\Sigma_{\uX},\mu_{\uX})$ which is \emph{essentially invertible}, i.e., there is a measurable map $\psi \colon X \rightarrow X$ such that the identity $\psi \circ \varphi = \mathrm{id}_X = \varphi \circ \psi$ holds almost everywhere. We refer to \cite{glasner2015ergodic}, \cite{Einsiedler2011}, and \cite{EFHN} for a general introduction to these systems.

Every such transformation gives rise to a \emph{Koopman operator} $T_\varphi \in \mathscr{L}(\mathrm{L}^2(\uX))$ via $T_\varphi f \coloneqq f \circ \varphi$ for $f \in \mathrm{L}^2(\uX)$. The operator $T = T_\varphi$ is a Markov lattice isomorphism, i.e., $T$ is a unitary operator with $|Tf| = T|f|$ for every $f \in \mathrm{L}^2(\uX)$ and $T\mathbbm{1} = \mathbbm{1}$. Conversely, if $\uX$ is a standard probability space, then every Markov lattice isomorphism $T \in \mathscr{L}(\mathrm{L}^2(\uX))$ is in fact a Koopman operator of an invertible measure-preserving point-transformation, by a result going back to von Neumann (see \cite[Proposition 7.19 and Theorem 7.20]{EFHN}, see also \cite[Proposition 3.2]{jt19}). For a general probability space $\uX$, a Markov lattice isomorphism $T \in \mathscr{L}(\mathrm{L}^2(\uX))$ is still induced by a unique transformation of the probability algebra associated to $\uX$ (see \cite[Theorem 12.10]{EFHN}). Studying transformations of probability algebras (see, e.g., \cite{jamneshan2019fz,jt19,jt20}) or Markov lattice isomorphisms (see \cite[Section 12.3]{EFHN}, \cite{EHK2021}) yield two (equivalent) ways to approach ergodic theory on general probability spaces.\medskip\\
\textbf{Operator theoretic measure-preserving systems.} Here we take the operator theoretic path of \cite[Section 12.3]{EFHN} and call, for a fixed group $G$, a pair $(\uX;T)$ a \emph{measure-preserving system} (with respect to $G$) if $\uX$ is a probability space and $T \colon G \rightarrow \mathscr{L}(\mathrm{L}^2(\uX)), \, t \mapsto T_t$ is a representation of $G$ as Markov lattice isomorphisms on $\mathrm{L}^2(\uX)$. It is \emph{ergodic} if the fixed space
		\begin{align*}
			\fix(T) \coloneqq \{f \in \mathrm{L}^2(\uX)\mid T_t f = f \textrm{ for every } t \in G\}
		\end{align*}
	is one-dimensional.\medskip\\
An \emph{extension} $J \colon (\uY;S) \rightarrow (\uX;T)$ between measure-preserving systems is a \emph{Markov embedding} $J \in \mathscr{L}(\mathrm{L}^2(\uY),\mathrm{L}^2(\uX))$, i.e., $J$ is a linear isometry with $|Jf| = J|f|$ for every $f \in \mathrm{L}^2(\uY)$, and $J\mathbbm{1} = \mathbbm{1}$, such that the diagram 
    \[
		\xymatrix{
			\mathrm{L}^2(\uX) \ar[r]^{T_t}   &\mathrm{L}^2(\uX)\\
			\mathrm{L}^2(\uY) \ar[u]^{J}  \ar[r]^{S_t} & \mathrm{L}^2(\uY) \ar[u]^{J}
		}
	\]	
commutes for every $t \in G$. It is \emph{relatively ergodic} if $J(\fix(S)) = \fix(T)$. We say that $J$ is an \emph{isomorphism}, if $J$ is bijective (in which case also $J^{-1}$ defines an extension).\medskip\\
    We will need the concept of morphism between extensions in a more flexible version than discussed in the topological case (cf. Subsection \ref{sectopdyn}). Given extensions $J_i \colon (\uY_i;S_i) \rightarrow (\uX_i;T_i)$ between measure-preserving systems for $i=1,2$, a \emph{morphism} $(I^{\flat};I^{\sharp}) \colon J_1 \rightarrow J_2$ consists of two extensions $I^{\flat} \colon (\uY_1;S_1) \rightarrow (\uY_2;S_2)$ and $I^{\sharp} \colon (\uX_1;T_1) \rightarrow (\uX_2;T_2)$ such that the diagram 
    \[
		\xymatrix{
			(\uX_1;T_1) \ar[r]^{I^\sharp}   & (\uX_2;T_2)\\
			(\uY_1;S_1) \ar[u]^{J_1}  \ar[r]^{I^{\flat}} & (\uY_2;S_2) \ar[u]^{J_2}
		}
	\]	
    commutes. It is an \emph{isomorphism} if both $I^{\flat}$ and $I^{\sharp}$ are isomorphisms of measure-preserving systems.\medskip\\
    \textbf{Conditional $\mathrm{L}^2$-spaces.} For every extension $J \colon (\uY;S) \rightarrow (\uX;T)$ the adjoint $J^* \in \mathscr{L}(\mathrm{L}^2(\uX),\mathrm{L}^2(\uY))$ extends uniquely to a bounded linear operator $\mathbb{E}_{\uY}\colon \mathrm{L}^1(\uX) \rightarrow \mathrm{L}^1(\uY)$ which is called the corresponding \emph{conditional expectation operator}, see \cite[Chapter 13]{EFHN}. This allows to define the \emph{conditional $\mathrm{L}^2$-space}
	\begin{align*}
			\mathrm{L}^2(\uX|\uY) \coloneqq \{f \in \mathrm{L}^2(\uX) \mid \mathbb{E}_{\uY}|f|^2 \in \mathrm{L}^\infty(\uY)\} 
	\end{align*}
    with respect to $J$ (see \cite[Section 2.13]{Tao2009}, \cite[Section 7.2]{EHK2021} and \cite[Section 3]{jamneshan2019fz}). The space $\mathrm{L}^2(\uX|\uY)$ equipped with the multiplication
    \begin{align*}
        \mathrm{L}^\infty(\uY) \times 	\mathrm{L}^2(\uX|\uY)  \rightarrow 	\mathrm{L}^2(\uX|\uY), \quad (f,g) \mapsto Jf \cdot g 
    \end{align*}
    and the \enquote{vector-valued inner product}
    	\begin{align*}
	 		(\cdot \mid \cdot)_{\uY} \colon \mathrm{L}^2(\uX|\uY) \times \mathrm{L}^2(\uX|\uY) \rightarrow \mathrm{L}^\infty(\uY), \quad (f,g) \mapsto \mathbb{E}_{\uY}(f \overline{g}).
	 	\end{align*}
    is a so called \emph{Kaplansky--Hilbert module} over the Stone algebra $\mathrm{L}^\infty(\uY)$. This allows to effectively apply operator theory on Kaplansky--Hilbert modules (which generalizes results of operators on ordinary Hilbert spaces) to ergodic theory. For a systematic introduction to Kaplansky--Hilbert modules we refer to \cite[Section 2]{EHK2021}. \medskip\\
    \textbf{Topological dynamics versus ergodic theory.} We finally review the close connection between topological dynamics and ergodic theory via topological models (see \cite[Section 12.3]{EFHN} and \cite[Section 7]{jt-foundational}). Recall that for a given topological dynamical system $(K;\varphi)$, the space of all $\varphi$-invariant regular Borel probability measures on $K$ is denoted by $\mathrm{P}^\varphi(K)$. Every $\mu_K \in \mathrm{P}^\varphi(K)$ gives rise to a measure-preserving system through $(T_\varphi)_t f \coloneqq f \circ \varphi_{t^{-1}}$ for $f \in \mathrm{L}^2(K,\mu_K)$ and $t \in G$ which we write as $(K,\mu_K;T_\varphi)$. If $q \colon (K;\varphi) \rightarrow (L;\psi)$ is an extension of topological dynamical systems and $\mu_K \in \mathrm{P}^\varphi(K)$, we denote the induced extension $(L,q_*\mu_K;T_\psi) \rightarrow (K,\mu_K;T_\varphi)$ of measure-preserving systems by $J_q$ (i.e., $J_qf \coloneqq f \circ q$ for every $f \in \mathrm{L}^2(L,q_*\mu_K)$).
    
    A \emph{topological model} $(K;\varphi;\mu;J)$ for a measure-preserving system $(\uX;T)$ then consists of a topological dynamical system $(K;\varphi)$, a fully supported measure $\mu_K \in \mathrm{P}^{\varphi}(K)$ and an isomorphism $I \colon (K,\mu_K;T_\varphi) \rightarrow (\uX;T)$.
	
	Given such a model $(K;\varphi;\mu_K;V)$, the image $\mathcal{A} = I(\mathrm{C}(K)) \subset \mathrm{L}^\infty(\uX)$ is
	    \begin{enumerate}[(i)]
	        \item a unital C*-subalgebra of $\mathrm{L}^\infty(\uX)$,
	        \item invariant with respect to $T$, and
	        \item dense in $\mathrm{L}^2(\uX)$.
	    \end{enumerate}
	   Conversely, given any $\mathcal{A} \subset \mathrm{L}^\infty(\uX)$ satisfying conditions (i) -- (iii) one can employ the Gelfand--Naimark theory (cf. Subsection \ref{sectopdyn}) to construct a topological model $(K;\varphi;\mu;I)$ for $(\uX;T)$ with $I(\mathrm{C}(K)) = \mathcal{A}$ in a canonical way: There is a unique invariant fully supported measure $\mu_K \in \mathrm{P}^\varphi(K)$ of the Gelfand system $(K;\varphi)$ of $(\mathcal{A};T)$ such that the Gelfand isomorphism $I \colon (\mathcal{A};T) \rightarrow (\mathrm{C}(K);T_\varphi)$ extends to an isomorphism $I \colon (\uX;T) \rightarrow (K,\mu_K;T_\varphi)$ of measure-preserving systems (see \cite[Section 12.3]{EFHN}, \cite[Section 7]{jt-foundational}, or \cite[Section 2.2]{jt20} for the details).
	   
	   Thus, constructing topological models for a given system $(\uX;T)$ comes down to finding invariant, $\mathrm{L}^2$-dense unital $\mathrm{C}^*$-subalgebras $\mathcal{A} \subset \mathrm{L}^\infty(\uX)$. In particular, we can always construct the \emph{Stone model}, also called the \emph{canonical model}, $(K_{\mathrm{st}};\varphi_{\mathrm{st}};\mu_{K_{\mathrm{st}}};I_{\mathrm{st}})$ which corresponds to the algebra $\mathcal{A} = \mathrm{L}^\infty(\uX)$ and yields a Stonean space $K_{\mathrm{st}}$ and a normal measure $\mu_{K_{\mathrm{st}}}$ (see \cite[Subsection 12.4]{EFHN}, \cite[Section 7]{jt-foundational}, and \cite[Section 2.2]{jt20}).

    One can apply this procedure also to an extension $J \colon (\uY;S) \rightarrow (\uX;T)$ of measure-preserving dynamical systems. A tuple $(q;\mu;I^{\flat};I^{\sharp})$ is a \emph{topological model for $J$} (cf. \cite[Definition 3.3]{EdKr2022}, \cite[Section 7]{jt-foundational}, \cite[Section 2.2]{jt20}) if 
        \begin{enumerate}[(i)]
            \item $q \colon (K;\varphi) \rightarrow (L;\psi)$ is an extension of topological dynamical systems,
            \item $\mu \in \mathrm{P}^\varphi(K)$ is a fully supported invariant measure, and
            \item $(I^{\flat};I^{\sharp}) \colon J_q \rightarrow J$ is an isomorphism from $J_q \colon (L,q_*\mu_K;T_\psi) \rightarrow (K,\mu_K;T_\varphi)$ to $J$.
        \end{enumerate}
    In this case, the corresponding unital C*-subalgebras $\mathcal{B} \coloneqq I^{\flat}(\mathrm{C}(L)) \subset \mathrm{L}^\infty(\uY)$ and $\mathcal{A} \coloneqq I^{\sharp}(\mathrm{C}(K)) \subset \mathrm{L}^\infty(\uX)$ satisfy $J(\mathcal{B}) \subset \mathcal{A}$.
    Conversely, if $\mathcal{B}$ is an $\mathrm{L}^2$-dense invariant unital C*-subalgebra of  $\mathrm{L}^\infty(\uY)$ and $\mathcal{A}$ is an $\mathrm{L}^2$-dense invariant unital C*-subalgebra of $\mathrm{L}^\infty(\uX)$ satisfying $J(\mathcal{B}) \subset \mathcal{A}$, then we can find a topological model $(q;\mu;I^{\flat};I^{\sharp})$ for $J$ with $I^{\flat}(\mathrm{C}(L)) = \mathcal{B}$ and $I^{\sharp}(\mathrm{C}(K)) = \mathcal{A}$ (see \cite[Section 3]{EdKr2022} or \cite[Sections 2 and 7]{jt-foundational}). 

 \subsection{Extensions with relative discrete spectrum}\label{measdiscr}
     There are several ways to define compact extensions in ergodic theory (see, e.g., \cite[Theorem 4.1]{jamneshan2019fz}). 
     The following, based on the classical notion of  systems with discrete spectrum, is easy to state and implicitly used in \cite{ellis}. Recall here that for an extension $J \colon (\uY;S) \rightarrow (\uX;T)$ the space $\mathrm{L}^\infty(\uX)$ is canonically a module over $\mathrm{L}^\infty(\uY)$ via the multiplication $f \cdot g \coloneqq (Jf)g$ for $f \in \mathrm{L}^\infty(\uY)$ and $g \in \mathrm{L}^2(\uX)$.
    \begin{definition}\label{defreldiscmeas}
	    An extension $J \colon (\uY;S) \rightarrow (\uX;T)$ \emph{has relative discrete spectrum} if
		\begin{align*}
			\bigcup \{\Gamma \subset \mathrm{L}^\infty(\uX)\mid \Gamma \textrm{ finitely generated,  $T$-invariant } \mathrm{L}^\infty(\uY) \textrm{-submodule of } \mathrm{L}^\infty(\uX)\}
		\end{align*}
	    is dense in $\mathrm{L}^2(\uX)$.
    \end{definition}
    Given a relative skew-product extension (see Section \ref{relskew}) over a topological dynamical system $(L;\psi)$ and a fully supported invariant measure $\mu_L \in \mathrm{P}^\psi(L)$, we obtain a fully supported invariant measure on the relative skew-product by integrating the canonical relatively invariant measure from Theorem \ref{uniquerel} with respect to $\mu_L$. We obtain the following result for the induced extension of measure-preserving systems.
    \begin{theorem}\label{skewdiscr}
	    Consider
	        \begin{enumerate}[(i)]
	            \item a topological dynamical system $(L;\psi)$,
	            \item an open compact group bundle $p \colon H \rightarrow L_{\fix}$,
	            \item an open closed subgroup bundle $p|_U \colon U \rightarrow L_{\fix}$, 
	            \item a relative cocycle $c \colon G \times L \rightarrow H$, and
	            \item a fully supported invariant measure $\mu_L \in \mathrm{P}^\psi(L)$.
	        \end{enumerate}
	    Then $J_{q_c} \colon (L,\mu_L;T_\psi) \rightarrow \left(L \times_{L_{\fix}} H/U, \int_L \delta_L \otimes \lambda_{H/U} \, \mathrm{d}\mu_L;T_{\varphi_c}\right)$ is an extension with relative discrete spectrum.
	\end{theorem}
	In the case of a Stonean space $L$, Theorem \ref{skewdiscr} is a direct consequence of \ref{reldiscsp}. The general result is implied by the following  observation allowing to lift relative skew-products to the canonical model (cf. Subsection \ref{measpres}).
	\begin{proposition}
	    Consider
	        \begin{enumerate}[(i)]
	            \item a topological dynamical system $(L;\psi)$,
	            \item an open compact group bundle $p \colon H \rightarrow L_{\fix}$,
	            \item an open closed subgroup bundle $p|_U \colon U \rightarrow L_{\fix}$, 
	            \item a relative cocycle $c \colon G \times L \rightarrow H$, and
	            \item a fully supported invariant measure $\mu_L \in \mathrm{P}^\psi(L)$.
	        \end{enumerate}
	   Then there is 
	        \begin{enumerate}[(1)]
	            \item an open compact group bundle $p_{\mathrm{st}} \colon H_{\mathrm{st}} \rightarrow (L_{\mathrm{st}})_{\fix}$,
	            \item an open closed subgroup bundle $(p_{\mathrm{st}})|_{U_{\mathrm{st}}} \colon U_{\mathrm{st}} \rightarrow L_{\fix}$, and
	            \item a relative cocycle $c_{\mathrm{st}} \colon G \times L_{\mathrm{st}} \rightarrow H_{\mathrm{st}}$,
	        \end{enumerate}
	   such that the extensions $J_{q_c}$ and $J_{q_{c_{\mathrm{st}}}}$ are isomorphic.  
	\end{proposition}
	\begin{proof}
	    We abbreviate $M \coloneqq L_{\fix}$ and $M_{\mathrm{st}} \coloneqq (L_{\mathrm{st}})_{\fix}$. Moreover, we write $q_{\mathrm{st}}^L\colon (L_{\mathrm{st}};\psi_{\mathrm{st}}) \rightarrow (L;\psi)$ for the canonical extension of topological dynamical systems induced by the embedding $\mathrm{C}(L) \hookrightarrow \mathrm{L}^\infty(L,\mu_L)$ (see Sections  \ref{sectopdyn} and \ref{measpres}). Similarly, we denote the canonical quotient map $M_{\mathrm{st}} \rightarrow M$ by $q_{\mathrm{st}}^{M}$.
	    
	    The pullback bundle $p_{\mathrm{st}} \colon H_{\mathrm{st}} \coloneqq M_{\mathrm{st}} \times_{M} H \rightarrow M_{\mathrm{st}}$ is an open compact group bundle and $U_{\mathrm{st}} \coloneqq M_{\mathrm{st}} \times_{M} U\subset H$ gives rise to an open closed subgroup bundle $p|_U\colon U \rightarrow M_{\mathrm{st}}$. Moreover, the induced homogeneous space bundle $p_{/U_{\mathrm{st}}} \colon H_{\mathrm{st}}/U_{\mathrm{st}} \rightarrow M_{\mathrm{st}}$ can and will be identified with the pullback bundle $M_{\mathrm{st}} \times_M H/U \rightarrow M_{\mathrm{st}}$. We obtain a relative cocycle $c_{\mathrm{st}} \colon G \times L_{\mathrm{st}} \rightarrow H_{\mathrm{st}}$ by setting $c_{\mathrm{st}}(t,l') \coloneqq (q_{\fix}(l'), c(t,q_{\mathrm{st}}(l')))$ for $(t,l') \in G \times L_{\mathrm{st}}$. Now consider the continuous map
	        \begin{align*}
	            \Phi \colon L_{\mathrm{st}} \times_{M_{\mathrm{st}}} H_{\mathrm{st}}/U_{\mathrm{st}} \rightarrow L \times_{M} H/U, \quad (l', (m',xU)) \mapsto (q_{\mathrm{st}}(l'), xU).
	        \end{align*}
	   Then $\Phi$ defines an extension between the corresponding topological relative skew-product systems.
	   One can readily check that 
	    \begin{align*}
	        \Phi_*(\delta_{l'} \times \lambda_{H_{q_{\fix}(l')}/U_{q_{\fix}(l')}}) = \delta_{q_{\mathrm{st}}(l')} \times \lambda_{H_{q_{\fix}(q_{\mathrm{st}}(l'))}/U_{q_{\fix}(q_{\mathrm{st}}(l'))}}
	   \end{align*}
	   for every $l' \in L_{\mathrm{st}}$. This implies
	    \begin{align*}
	        \Phi_*\left(\int_{L_{\mathrm{st}}} \delta_{L_{\mathrm{st}}} \otimes \lambda_{H_{\mathrm{st}}/U_{\mathrm{st}}}\, \mathrm{d}\mu_{L_{\mathrm{st}}}\right) &= \int_{L_{\mathrm{st}}} \Phi_*(\delta_{l'} \times \lambda_{H_{q_{\fix}(l')}/U_{q_{\fix}(l')}})\, \mathrm{d}\mu_{L_{\mathrm{st}}}(l')\\
	        &= \int_{L_{\mathrm{st}}} \delta_{q_{\mathrm{st}}(l')} \times \lambda_{H_{q_{\fix}(q_{\mathrm{st}}(l'))}/U_{q_{\fix}(q_{\mathrm{st}}(l'))}}\, \mathrm{d}\mu_{L_{\mathrm{st}}}(l')\\
	        &=  \int_{L} \delta_{l} \times \lambda_{H_{q_{\fix}(l)}/U_{q_{\fix}(l)}}\, \mathrm{d}\mu_{L}(l) = \int_L \delta_L \otimes \lambda_{H/U}\, \mathrm{d}\mu_L.
	    \end{align*}
	  Consequently, $\Phi$ is measure-preserving and induces an extension $J_\Phi \colon (L \times_{M} H/U, \int_L \delta_L \otimes \lambda_{H/U}\, \mathrm{d}\mu_L; T_{\varphi_c}) \rightarrow (L_{\mathrm{st}} \times_{M_{\mathrm{st}}} H_{\mathrm{st}}/U_{\mathrm{st}}, \int_{L_{\mathrm{st}}} \delta_{L_{\mathrm{st}}} \otimes \lambda_{H_{\mathrm{st}}/U_{\mathrm{st}}}\, \mathrm{d}\mu_{L_{\mathrm{st}}}; T_{\varphi_{c_{\mathrm{st}}}})$. Moreover, $J_{\Phi}J_{q_c} = J_{q_{c_{\mathrm{st}}}} J_{q_{\mathrm{st}}}$ and thus the pair $(J_{q_{\mathrm{st}}}; J_{\Phi})$ gives rise to a morphism from the extension $J_{q_c}$ to the extension $J_{q_{c_{\mathrm{st}}}}$. We show that $J_{\Phi} \colon \mathrm{L}^2(L \times_M H/U) \rightarrow \mathrm{L}^2(L_{\mathrm{st}} \times_{M_{\mathrm{st}}} H_{\mathrm{st}}/U_{\mathrm{st}})$
	  has dense range, hence is surjective, and thus $(J_{q_{\mathrm{st}}}; J_{\Phi})$ is an isomorphism of extensions. 
	  By the Stone-Weiterstrass theorem, the functions 
	    \begin{align*}
	        f' \otimes (\mathbbm{1} \otimes \mathbbm{1}) &\colon L_{\mathrm{st}} \times_{M_{\mathrm{st}}} H_{\mathrm{st}}/U_{\mathrm{st}} \rightarrow \C, \quad (l',(m',xU)) \mapsto f'(l'), \textrm{ and }\\
	        \mathbbm{1} \otimes (\mathbbm{1} \otimes g') &\colon L_{\mathrm{st}} \times_{M_{\mathrm{st}}} H_{\mathrm{st}}/U_{\mathrm{st}} \rightarrow \C, \quad (l',(m',xU)) \mapsto g'(xU)
	    \end{align*}
	   for $f' \in \mathrm{C}(L_{\mathrm{st}})$ and $g' \in \mathrm{C}(H/U)$ generate $\mathrm{C}(L_{\mathrm{st}} \times_{M_{\mathrm{st}}} H_{\mathrm{st}}/U_{\mathrm{st}})$ as a C*-algebra. It therefore suffices to show that these are contained in the image of the restriction $       (J_{\Phi})|_{\mathrm{L}^\infty} \colon  \mathrm{L}^\infty(L \times_M H/U) \rightarrow \mathrm{L}^\infty(L_{\mathrm{st}} \times_{M_{\mathrm{st}}} H_{\mathrm{st}}/U_{\mathrm{st}})$ which is a unital *-homomorphism. So take $f' \in \mathrm{C}(L_{\mathrm{st}})$. Then, by definition of the Stone model, we find $f \in \mathrm{L}^\infty(L,\mu_L)$ with $J_{q_{\mathrm{st}}}f = f'$ and thus 
	    \begin{align*}
	        f' \otimes (\mathbbm{1} \otimes \mathbbm{1}) = J_{q_{c_{\mathrm{st}}}}f' = J_{q_{c_{\mathrm{st}}}}J_{q_{\mathrm{st}}}f = J_{\Phi}J_{q_c}f \in J_{\Phi}(\mathrm{L}^\infty(L \times_M H/U)).
	    \end{align*}
	  On the other hand, if $g' \in \mathrm{C}(H/U)$, then the continuous function
	    \begin{align*}
	       h \colon L \times_M H/U \rightarrow \C, \quad (l, xU) \mapsto g'(xU)
	    \end{align*}
	   satisfies $J_{\Phi}h = \mathbbm{1} \otimes (\mathbbm{1} \otimes g')$. This yields the claim.
	\end{proof}
    Applying the theory of Kaplansky--Hilbert modules to $\mathrm{L}^2(\uX|\uY)$ (cf. Subsection \ref{measpres} above), we can improve the properties of the submodules considered in Definition \ref{defreldiscmeas}. 
    \begin{definition}
            Let $J \colon (\uY;S) \rightarrow (\uX;T)$ be an extension of measure-preserving dynamical systems. An $\mathrm{L}^\infty(\uY)$-submodule $\Gamma$ of $\mathrm{L}^2(\uX|\uY)$ is a \emph{finitely generated KH-submodule} if it has a finite \emph{suborthonormal basis}, i.e., there are $e_1, \dots, e_m \in \Gamma$ generating $\Gamma$ such that 
        \begin{enumerate}[(i)]
            \item $\mathbb{E}_{\uY}(e_i\overline{e_j}) = 0$ for $i,j \in \{1, \dots, m\}$ with $i \neq j$, and
            \item $\mathbb{E}_{\uY}|e_i|^2$ is given by a characteristic function for $i \in \{1, \dots, m\}$.
        \end{enumerate}
    \end{definition}
    We then obtain the following characterization of extensions with relative discrete spectrum  from \cite[Subsection 2.5 and Proposition 8.5]{EHK2021} (cf. \cite[Theorem 4.1]{jamneshan2019fz}).
    \begin{proposition}\label{chardiscrspectrum3}
         For an extension $J \colon (\uY;S) \rightarrow (\uX;T)$ of measure-preserving dynamical systems the following assertions are equivalent.
            \begin{enumerate}[(a)]
                \item $J$ has relative discrete spectrum.
                \item The set
                    \begin{align*}
			            \bigcup \{\Gamma \subset \mathrm{L}^2(\uX|\uY)\mid \Gamma \textrm{ $T$-invariant, finitely generated } \textrm{KH-submodule of } \mathrm{L}^2(\uX|\uY)\}
			        \end{align*}
			        is dense in $\mathrm{L}^2(\uX)$.
            \end{enumerate}
    \end{proposition}

\subsection{Natural models for structured extensions}\label{canmod}
 We now want to prove a partial converse of Theorem \ref{skewdiscr}, i.e., represent extensions of measure-preserving systems with relative discrete spectrum as skew-product extensions. The strategy is to build a suitable topological model and then apply Theorem \ref{reptheorem1}. For the first part we need the following lemma, the first part of which is \cite[Remark 5.16]{ellis} (see also \cite[Proposition 2.8]{EdKr2022}).

\begin{lemma}\label{submodules}
	Let $J \colon (\uY;S) \rightarrow (\uX;T)$ be an extension of measure-preserving systems.
		\begin{enumerate}[(i)]
			\item If $(\uX;T)$ is ergodic, then every $T$-invariant, finitely generated KH-submodule $\Gamma$ of $\mathrm{L}^2(\uX|\uY)$ is contained in $\mathrm{L}^\infty(\uX)$.
			\item If $J$ is relatively ergodic and $\Gamma \subset \mathrm{L}^2(\uX|\uY)$ is a $T$-invariant, finitely generated KH-submodule of $\mathrm{L}^2(\uX|\uY)$, then there is a sequence $(\Gamma_n)_{n \in \N}$ of $T$-invariant, finitely generated KH-submodules with $\Gamma_n \subset \Gamma \cap \mathrm{L}^\infty(\uX)$ for every $n \in \N$ such that
				\begin{align*}
					\Gamma = \overline{\bigcup_{n \in \N} \Gamma_n} \textrm{ in } \mathrm{L}^2(\uX).
				\end{align*}
		\end{enumerate}
\end{lemma}
\begin{proof}
	Let $\Gamma$ be a $T$-invariant, finitely generated KH-submodule of $\mathrm{L}^2(\uX|\uY)$ and choose a suborthonormal basis $e_1, \dots, e_m$ for $\Gamma$. Then,
		\begin{align*}
			e \coloneqq \left(\sum_{i=1}^m |e_i|^2\right)^\frac{1}{2} \in \fix(T)
		\end{align*}
	by \cite[Lemma 8.3]{EHK2021}. Thus, if $(\uX;T)$ is ergodic, then there is a constant $c \geq 0$ with $\sum_{i=1}^m |e_i|^2 = c\mathbbm{1}$ and hence $|e_i| \leq c \mathbbm{1}$ for all $i \in \{1, \dots, m\}$. This shows (i).
	If only $J$ is ergodic, then we find $f \in \fix(S)$ with $Jf = e$. For $n \in \N$, consider $f_n \coloneqq \mathbbm{1}_{|f| \leq n} \in \fix(S) \cap \mathrm{L}^\infty(\uY)$. Then $\Gamma_n \coloneqq f_n \Gamma$ is a $T$-invariant, finitely generated KH-submodule of $\mathrm{L}^2(\uX|\uY)$ with suborthonormal basis $f_ne_1, \dots, f_ne_m$. Moreover, 
		\begin{align*}
			\left(\sum_{i=1}^m|f_ne_i|^2\right)^{\frac{1}{2}} = J|f_n| \cdot e = J(|f_n|f) \leq n \mathbbm{1}
		\end{align*}
	and hence $f_ne_i \in \mathrm{L}^\infty(\uX)$ for every $i \in \{1, \dots, m\}$. This shows $\Gamma_n \subset \mathrm{L}^\infty(\uX)$. Finally, 
	\begin{align*}
		\|f_ne_i - e_i\|_{\mathrm{L}^2(\uX)}^2 = \int_{\uY} \mathbb{E}_{\uY}|f_ne_i - e_i|^2\,\mathrm{d}\mu_{\uY} = \int_{|f| > n}\mathbb{E}_{\uY}|e_i|^2\,\mathrm{d}\mu_{\uY} \rightarrow 0
	\end{align*}	
	for $n \rightarrow \infty$. We therefore obtain that the union of the modules $\Gamma_n$ for $n \in \N$ is dense in $\Gamma$.	
\end{proof}

We also use the following auxiliary result.

\begin{lemma}\label{proj}
    Let $J \colon (\uY;S) \rightarrow (\uX;T)$ be an extension of measure-preserving systems. Then every finitely generated KH-submodule $\Gamma \subset \mathrm{L}^\infty(\uX)$ over $\mathrm{L}^\infty(\uY)$ is closed and projective.
\end{lemma}
\begin{proof}
     Take a finitely generated KH-submodule $\Gamma \subset \mathrm{L}^\infty(\uX)$ with suborthonormal basis $e_1, \dots, e_m \in \Gamma$. One can deduce that $\Gamma$ is closed in $\mathrm{L}^\infty(\uX)$ from the general fact that Kaplansky--Hilbert modules are complete with respect to their canonical norm (see \cite[Theorems 2.2.3 and 7.1.3]{Kusr2000}). We recall the argument in our framework: $(f_n)_{n \in \N}$ is a sequence in $\Gamma$ converging to some $f \in \mathrm{L}^\infty(\uX)$, then $(\mathbb{E}_{\uY}(f_n \overline{e_i}))_{n \in \N}$ converges to  $\mathbb{E}_{\uY}(f\overline{e_i})$ for $i \in \{1, \dots, m\}$. Hence, using the expansion with respect to the suborthonormal basis $e_1, \dots, e_m$ (see \cite[Lemma 2.10]{EHK2021}), we obtain 
        \begin{align*}
            \lim_{n \rightarrow \infty} f_n = \lim_{n \rightarrow \infty} \sum_{i=1}^m \mathbb{E}_{\uY}(f_n\overline{e_i})e_i = \sum_{i=1}^m \mathbb{E}_{\uY}(f\overline{e_i})e_i \in \Gamma.
        \end{align*}
    To see that $\Gamma$ is projective, observe that 
        \begin{align*}
            \Gamma \oplus \bigoplus_{i=1}^m \mathrm{L}^\infty(\uY)(\mathbbm{1} - \mathbb{E}_{\uY}|e_i|^2)
        \end{align*}
    is a free $\mathrm{L}^\infty(\uY)$-module with a basis $e_1', \dots, e_m'$ given by
        \begin{align*}
            e_i' \coloneqq (e_i, 0, \dots, 0, \mathbbm{1} - \mathbb{E}_{\uY}|e_i|^2, 0, \dots,0) \in  \Gamma \oplus \bigoplus_{i=1}^n \mathrm{L}^\infty(\uY)(\mathbbm{1} - \mathbb{E}_{\uY}|e_i|^2)
        \end{align*}
   , where $\mathbbm{1} - \mathbb{E}_{\uY}|e_i|^2$ is at the $(i+1)$th position for $i \in \{1, \dots, m\}$. Thus, $\Gamma$ is projective.
\end{proof}

With the help of Lemmas \ref{submodules} and \ref{proj} we can now build topological models for ergodic extensions with relative discrete spectrum (see also \cite[Theorem 3.5]{EdKr2022} for the case of extensions between ergodic systems).

\begin{definition}
    Let $J \colon (\uY;S) \rightarrow (\uX;T)$ be a relatively ergodic extension of measure-preserving systems with relative discrete spectrum. Let $\mathcal{B} \coloneqq \mathrm{L}^\infty(\uY)$ and let $\mathcal{A}$ be the unital C*-subalgebra of $\mathrm{L}^\infty(\uX)$ generated by all finitely generated, projective, closed, $T$-invariant $\mathrm{L}^\infty(\uY)$-submodules of $\mathrm{L}^\infty(\uX)$. We call the associated topological model  $(q;\mu;I^{\flat};I^{\sharp})$ (see Section \ref{measpres}) the \emph{natural model} of $J$.
\end{definition}

The following result shows that the natural model fits into the framework developed in Section \ref{sectop}.
\begin{proposition}\label{model}
    Let $J \colon (\uY;S) \rightarrow (\uX;T)$ be an ergodic extension with relative discrete spectrum and $(q;\mu;I^{\flat};I^{\sharp})$ its natural model. Then $q$ is an open topologically ergodic extension with relative discrete spectrum over a Stonean space and $q_*\mu$ is a normal measure.
\end{proposition}
\begin{proof}
    Write $q \colon (K;\varphi) \rightarrow (L;\psi)$. Since $(L;\psi;q_*\mu;V)$ is the Stone model (see Section \ref{measpres} above), we know that $L$ is Stonean and $q_*\mu$ is normal. Moreover, $q$ is open by \cite[Corollary 1.9]{ellis}. Finally, $q$ is topologically ergodic and has relative discrete spectrum by definition of the natural model.
\end{proof}
We can therefore apply Theorem \ref{reptheorem1} to represent relatively ergodic extensions with relative discrete spectrum as  relative skew-product extensions. However, before we prove the following extension property which will allow us to show that the closed subgroup bundle of the constructed skew-product extension is open (cf. Remark \ref{gap}). Recall that every bounded continuous function defined on a dense subset $D$ of a Stonean space $L$ can be extended uniquely to the whole space (see \cite[Exercise 24.2.11]{semadeni}), i.e., the map restriction map $\mathrm{C}(L) \rightarrow \mathrm{C}_{\mathrm{b}}(D), \, f \mapsto f|_D$ to the space of bounded continuous functions on $D$ is surjective (and hence an isomorphism of unital commutative C*-algebras).

In our natural model we obtain a similar property for closed, finitely generated, projective, invariant submodules. Observe here that for an extension  $q \colon (K;\varphi) \rightarrow (L;\psi)$ of topological dynamical systems and a subset $D \subset L_{\mathrm{\fix}}$ the preimage $(q \circ q_{\fix})^{-1}(D)$ is $\varphi$-invariant and hence we may consider the Koopman operator $T_\varphi$ on $\mathrm{C}_{\mathrm{b}}((q \circ q_{\fix})^{-1}(D))$. Moreover, $\mathrm{C}_{\mathrm{b}}((q \circ q_{\fix})^{-1}(D))$ is canonically a $\mathrm{C}(L)$-module.

\begin{lemma}\label{extensionlemma}
    Let $J \colon (\uY;S) \rightarrow (\uX;T)$ be a relatively ergodic extension with relative discrete spectrum and $(q;\mu;I^{\flat};I^{\sharp})$ with $q \colon (K;\varphi) \rightarrow (L;\psi)$ its natural model. Take
        \begin{enumerate}[(i)]
            \item a dense subset $D \subset L_{\fix}$, and
            \item a finitely generated, projective, closed, $T_\varphi$-invariant $\mathrm{C}(L)$-submodule $\Gamma \subset \mathrm{C}_{\mathrm{b}}((q_{\fix}\circ q)^{-1}(D))$.
        \end{enumerate} 
    Then there is a unique finitely generated, projective, closed, $T_\varphi$-invariant $\mathrm{C}(L)$-submodule $\Lambda \subset \mathrm{C}(K)$ with $\Lambda|_{(q_{\fix}\circ q)^{-1}(D)} = \Gamma$.
\end{lemma}
\begin{proof}
    Let $(K_{\mathrm{st}};\varphi_{\mathrm{st}})$ be the Stone model of $(\uX;T)$  and $q_{\mathrm{st}}\colon (K_{\mathrm{st}};\varphi_{\mathrm{st}}) \rightarrow (K;\varphi)$ the induced extension of topological dynamical systems  (see Sections  \ref{sectopdyn} and \ref{measpres}). We abbreviate 
        \begin{align*}
            D_{K} &\coloneqq (q_{\fix} \circ q)^{-1}(D) \subset K, \textrm{ and }\\
            D_{K_{\mathrm{st}}} &\coloneqq q_{\mathrm{st}}^{-1}(D_K) = (q_{\fix} \circ q \circ q_{\mathrm{st}})^{-1}(D) \subset K_{\mathrm{st}}
        \end{align*}
    and observe that these sets are dense in $K$ and $K_{\mathrm{st}}$, respectively, since the maps $q_{\fix} \circ q$, and $q_{\fix} \circ q \circ q_{\mathrm{st}}$ are open ($q$ and $q \circ q_{\mathrm{st}}$ are open by \cite[Corollary 1.9]{ellis} and $q_{\fix}$ is open by Lemma \ref{lemstonean}). 
    Observe now that, since $K_{\mathrm{st}}$ is Stonean, every bounded continuous function $f \colon D_{K_{\mathrm{st}}} \rightarrow \C$ has a unique continuous extension $\mathrm{ext}(f) \in \mathrm{C}(K_{\mathrm{st}})$.
	We take a finitely generated, projective, closed, and $T_\varphi$-invariant $\mathrm{C}(L)$-submodule $\Gamma \subset \mathrm{C}_{\mathrm{b}}(D_K)$. The maps
	    \begin{align*} 
	        I_{q_{\mathrm{st}}} &\colon \mathrm{C}_{\mathrm{b}}(D_K)  \rightarrow \mathrm{C}_{\mathrm{b}}(D_{K_{\mathrm{st}}}), \quad f \mapsto f \circ q_{\mathrm{st}} \textrm{ and }\\
	        \mathrm{ext} & \colon \mathrm{C}_{\mathrm{b}}(D_{K_{\mathrm{st}}})  \rightarrow \mathrm{C}(K_{\mathrm{st}}), \quad f \mapsto \mathrm{ext}(f)
	    \end{align*}
	  are $\mathrm{C}(L)$-linear embeddings of unital C*-algebras and intertwine the dynamics (the map $\mathrm{ext}$ is even an isomorphism and the inverse of the restriction map $\mathrm{C}(K_{\mathrm{st}}) \rightarrow \mathrm{C}_{\mathrm{b}}(D_{K_{\mathrm{st}}}),\, f \mapsto f|_{D_{K_{\mathrm{st}}}}$). We therefore obtain that the composition $\mathrm{ext} \circ  I_{q_{\mathrm{st}}}$ embeds $\Gamma$ into $\mathrm{C}(K_{\mathrm{st}})$ as a finitely generated, projective, closed and $T_\varphi$-invariant $\mathrm{C}(L)$-submodule $\Gamma'$. By definition of the natural model and the Stone model, $\Gamma'$ is contained in the image of $J_{q_{\mathrm{st}}}(\mathrm{C}(K))$, i.e., we find a finitely generated, projective, closed and $T_\varphi$-invariant $\mathrm{C}(L)$-submodule $\Lambda \subset \mathrm{C}(K)$ with $J_{q_{\mathrm{st}}}\Lambda = \Gamma'$.
	  Since the diagram
       \[
        	\xymatrix{
			\mathrm{C}(K) \ar[d]_{f \mapsto f|_{D_K}} \ar[r]^{J_{q_{\mathrm{st}}}} & \mathrm{C}(K_{\mathrm{st}})\\
			\mathrm{C}_{\mathrm{b}}(D_K) \ar[r]_{I_{q_{\mathrm{st}}}}   & \mathrm{C}_{\mathrm{b}}(D_{K_{\mathrm{st}}}) \ar[u]_{\mathrm{ext}}
		    }
	    \]	
	commutes, we obtain 
	    \begin{align*}
	        (\mathrm{ext} \circ I_{q_{\mathrm{st}}})(\Lambda|_{D_K}) = J_{q_{\mathrm{st}}}\Lambda = \Gamma' = (\mathrm{ext} \circ I_{q_{\mathrm{st}}})(\Gamma)
	    \end{align*}
	and therefore $\Lambda|_{D_K} = \Gamma$. This shows existence of the desired submodule. Since $D_K$ is dense in $K$, uniqueness is clear.
\end{proof}
\subsection{The Mackey--Zimmer representation theorem for relatively ergodic extensions}\label{mainres}
    We finally come to the main result of our article which is a converse of Theorem \ref{skewdiscr} for relatively ergodic extensions.
    \begin{theorem}\label{mz}
    Let $J \colon (\uY;S) \rightarrow (\uX;T)$ be a relatively ergodic extension of measure-preserving systems with relative discrete spectrum. Then there are
        \begin{enumerate}[(i)]
            \item a topological dynamical system $(L;\psi)$,
            \item an open compact group bundle $p \colon H \rightarrow L_{\fix}$,
            \item an open core-free closed subgroup bundle $p|_U \colon U \rightarrow L_{\fix}$ of $p$, 
            \item an ergodic relative cocycle $c \colon G \times L \rightarrow H$, and
            \item a normal invariant measure $\mu_L \in \mathrm{P}^\psi(L)$ on $L$,
        \end{enumerate}
            such that the induced relative skew-product extension 
                \begin{align*}
                    J_{q_c} \colon (L,\mu_L;T_\psi) \rightarrow \left(L \times_{L_{\fix}} H/U, \int_L \delta_L \otimes \lambda_{H/U} \, \mathrm{d}\mu_L;T_{\varphi_c}\right)
                \end{align*}
            is isomorphic to $J$.
    \end{theorem}
    \begin{remark}
        Our result is in the spirit of Ellis' version of the Mackey--Zimmer theorem representing an extension with relative discrete spectrum as a (relative) skew-product extension over the Stone model of the factor system (see \cite[Theorem 5.36]{ellis}, see also \cite[Theorem 5.3]{jamneshan2019fz}). When working with an at most countable group $G$ and standard Borel probability spaces $\uX$ and $\uY$ one can expect to deduce a representation over the original space $\uY$ from Theorem \ref{mz} similar to the one of Austin (see \cite[Proposition 5.7]{austin-fund}) by applying \cite[Proposition 3.2]{jt19} to obtain concrete measure-preserving maps  (cf. the discussion in \cite[Section 11]{jt20}). We leave the details to the interested reader. 
    \end{remark}
    We need the following auxiliary result.
    \begin{lemma}\label{boundedlemma}
        Let $H$ be a compact group and $[\pi] \in \hat{H}$. Then 
            \begin{align*}
                \|f\|_{\mathrm{L}^2(H)} \leq \|f\|_{\mathrm{C}(H)}\leq \dim(\pi)^2 \cdot \|f\|_{\mathrm{L}^2(H)}
            \end{align*}
        for all $f \in C_{[\pi]}$.
    \end{lemma}
    \begin{proof}
        Set $n \coloneqq \dim(\pi)$ and $f \in C_{[\pi]}$. The first inequality is trivial. For the second, let $(\pi_{jk})_{j,k}$ be a matrix representation of $[\pi]$. Then the elements $\pi_{jk}(x) \in \C$ for $j,k \in \{1, \dots, n\}$ form a unitary matrix for each $x \in H$. Thus,  $\|\pi_{jk}\|_{\mathrm{C}(H)} \leq 1$ for all $j,k \in \{1, \dots, n\}$. We therefore obtain
            \begin{align*}
               \|f\|_{\mathrm{C}(H)} &= \left\|\sum_{j,k=1}^n (f|\pi_{jk})_{\mathrm{L}^2(H)} \pi_{jk}\right\|_{\mathrm{C}(H)} \leq \sum_{j,k=1}^n \|f\|_{\mathrm{L}^2(H)} \cdot \|\pi_{jk}\|_{\mathrm{L}^2(H)} \|\pi_{jk}\|_{\mathrm{C}(H)}\\
               &\leq n^2\|f\|_{\mathrm{L}^2(H)}
            \end{align*}
       by the Cauchy--Schwarz inequality.
    \end{proof}
    \begin{proof}[Proof of Theorem \ref{mz}.]
       Let $(q;\mu;I^{\flat},I^{\sharp})$ with $q \colon (K;\varphi) \rightarrow (L;\psi)$ be the natural model of $J$. Moreover, let $\mu_L \coloneqq q_*\mu$ be the measure of the Stone model. By Proposition \ref{model} and Theorem \ref{reptheorem1}, there is an open compact group bundle $p \colon H \rightarrow L_{\fix}$, a core-free closed subgroup bundle $p|_U \colon U \rightarrow L_{\fix}$ of $p$, 
         and an ergodic relative cocycle $c \colon G \times L \rightarrow H$ such that $q$ is isomorphic as an extension of topological dynamical systems to the relative skew-product extension $q_c \colon (L \times_{L_{\fix}} H/U;\varphi_c) \rightarrow (L;\psi)$. By Proposition \ref{dynamicaldisint} and Theorem \ref{uniquerel}, $\int_L \delta_L \otimes \lambda_{H/U} \, \mathrm{d}\mu_L$ is the only invariant measure $\tilde{\mu} \in \mathrm{P}^{\varphi}(L \times_{L_{\fix}} H)$ with $(q_c)_*\tilde{\mu} = \mu_L$. Hence, we see (by considering the push-forward of $\mu$) that the isomorphism $\Phi \colon q \rightarrow q_c$ is measure-preserving and consequently is an isomorphism of extensions $(\mathrm{Id}_{\mathrm{L}^2(L,\mu_L)};J_{\Phi}) \colon J_{q_c} \rightarrow J_q$.\medskip\\
         We finally show that the core-free closed subgroup bundle $p|_U \colon U \rightarrow L_{\fix}$ is open. In view of Theorem \ref{loctrivrep2} it suffices to show that  $s|_{\mathrm{R}_{\hat{\sigma},U}}\colon \mathrm{R}_{\hat{\sigma},U} \rightarrow O$ is a locally trivial vector bundle for every local continuous section $\hat{\sigma} \colon O \rightarrow \hat{H}$. We take such a section $\hat{\sigma} \colon O \rightarrow \hat{H}$ and may assume, by restricting $\hat{\sigma}$, that $O$ is a clopen subset of $L_{\fix}$ and that $\dim \circ \hat{\sigma} \colon O \rightarrow \N$ is equal to some constant $n \in \N$ (see Proposition \ref{proptop} (\ref{proptop4})). By Lemma \ref{locconst} and Lemma \ref{subbundleopen2} it suffices to show that the dimension map
            \begin{align*}
                \dim \colon O \rightarrow \N_0, \quad m \mapsto \mathrm{R}_{\hat{\sigma}(m),U_m}
            \end{align*}
        is locally constant.
        Let $d$ be the maximal dimension among fibers of $\mathrm{R}_{\hat{\sigma},U}$ and consider the subset
            \begin{align*}
                O_{d} \coloneqq \{m \in O\mid \dim \mathrm{R}_{\hat{\sigma}(m),U_m} = d\}
            \end{align*}
        which is open by \cite[Theorem 18.2]{Gierz1982}. We show that it is also closed. Using Proposition \ref{locconst} $p^{-1}(O_{d}) \rightarrow O_{d}$ is a locally trivial bundle over $O_{d}$ and by Proposition \ref{denselydeforth} there is a densely defined local orthonormal basis  $\tau_1, \dots, \tau_{d} \colon V \rightarrow \mathrm{R}_{\hat{\sigma},U}$ for this bundle with respect to the continuous inner product induced by the Haar system of $H$ (cf. Proposition \ref{Renault}). By Lemma \ref{boundedlemma} we obtain $\|\tau_i(m)\|_{\mathrm{C}(H_m)} \leq n^2$ for all $i \in \{1, \dots, d\}$ and $m \in V$. Since bounded subsets of the locally trivial bundle $\mathrm{R}_{\hat{\sigma}}$ are relatively compact by \cite[Lemma 4.15]{EdKr2021} and $L$ is a Stonean space, every $\tau_i$ can be extended uniquely to a continuous section $\tau_i\colon \overline{V} \rightarrow \mathrm{R}_{\hat{\sigma}}$ (see again \cite[Exercise 24.2.11]{semadeni}) for $i \in \{1, \dots, d\}$. By continuity we still have $(\tau_i(m)|\tau_j(m)) = \delta_{ij}$ for all $i,j \in \{1, \dots, d\}$ and $m \in \overline{V} = \overline{O_{d}}$. Thus, if we can show that $\tau_i(m) \in \mathrm{R}_{\hat{\sigma},U}$ for all $m \in \overline{V}$, then $O_{d}$ is closed in $L_{\fix}$ as claimed.\medskip\\
        To see that $\tau_i$ indeed maps to $\mathrm{R}_{\hat{\sigma},U}$ for every $i \in \{1, \dots, d\}$ we apply Lemma \ref{extensionlemma}. To do so, we first have to construct a dense subset $D \subset L_{\fix}$ and finitely generated, projective, closed and $T_{\varphi_c}$-invariant $\mathrm{C}(L)$-submodule $\Gamma$ of $\mathrm{C}_{\mathrm{b}}((q_{\fix} \circ q)^{-1}(D))$.
        
        We start with the dense subset $D \coloneqq V \cup L_{\fix}\setminus \overline{V}$ of $L_{\fix}$. By Lemma \ref{boundedlemma},  there are bounded continuous functions $f_i\colon (q_{\fix} \circ q_c)^{-1}(W)\subset L \times_{L_{\fix}} H/U \rightarrow \C$ for $i \in \{1, \dots, d\}$ by setting 
            \begin{align*}
                f_i(l,xU) \coloneqq \begin{cases} \tau_i(q_{\fix}(l))(x) & \textrm{ if } (l,xU) \in (q_{\fix} \circ q_c)^{-1}(V),\\
                0 & \textrm{ if } (l,xU) \in (q_{\fix} \circ q_c)^{-1}(L_{\fix}\setminus \overline{V}).
                \end{cases}
            \end{align*}
            We claim that the $\mathrm{C}(L)$-submodule $\Gamma$ of $\mathrm{C}_{\mathrm{b}}((q_{\fix} \circ q_c)^{-1}(D))$ generated by $f_1, \dots, f_{d}$  is projective, closed and $T_{\varphi_c}$-invariant. In fact, if $g \in \Gamma$, then for every $l \in q_{\fix}^{-1}(V)$ and $m \coloneqq q_{\fix}(l) \in V$ the function $g_l \colon H_{m} \rightarrow \C, \, x \mapsto g(l,xU_{m})$ is contained in $\mathrm{R}_{\hat{\sigma}(m),U_m}$, and hence 
            \begin{align*}
                g_l = \sum_{i=1}^{d} (g_l|\tau_i(m))_{\mathrm{L}^2(H_m)} \tau_i(m)
            \end{align*}
        by choice of $\tau_1, \dots, \tau_{d}$. Therefore, the restriction $g_{\{l\} \times H_m/U_m} \in \mathrm{C}(\{l\} \times H_m/U_m)$ satisfies
            \begin{align*}
                g|_{\{l\} \times H_{m}/U_{m}} = \sum_{i=1}^{d} (g|f_i)_{\mathrm{L}^2(\{l\} \times H_{m}/U_{m})} f_i|_{\{l\} \times H_{m}/U_{m}}
            \end{align*}
       , where $\{l\} \times H_m/U_m$ is equipped with the measure $\delta_l \times \lambda_{H_m/U_m}$. 
        Both sides of this identity are equal to zero if $l \in L$ with $m \coloneqq q_{\fix}(l) \in L_{\fix}\setminus \overline{V}$, hence the identity even holds for every $l \in L_{\fix}$ with $m \coloneqq q_{\fix}(l) \in D$. Now observe that 
            \begin{align*}
                (g|f_i) \colon q_{\fix}^{-1}(W) \rightarrow \C, \quad l \mapsto (g|f_i)_{\mathrm{L}^2(\{l\} \times H_{q_{\fix}(l)}/U_{q_{\fix}(l)})}
            \end{align*}
        is a bounded continuous function defined on a dense subset of $L$, and hence has a unique continuous extension $\mathrm{ext}((g|f_i))\colon L \rightarrow \C$ for every $i \in \{1, \dots, d\}$. Then
            \begin{align*}
                g = \sum_{i=1}^d \mathrm{ext}((g|f_i)) f_i
            \end{align*}
        in the $\mathrm{C}(L)$-module $\mathrm{C}_{\mathrm{b}}((q_{\fix} \circ q_c)^{-1}(D))$. Using that $\mathrm{ext}((f_i|f_j)) = \delta_{ij}\mathbbm{1}_{\overline{q_{\fix}^{-1}(D)}}$ for all $i,j \in \{1, \dots, d\}$ we argue as in the proof of Proposition \ref{model} to obtain that $\Gamma$ is a projective and closed $\mathrm{C}(L)$-submodule of $\mathrm{C}_{\mathrm{b}}((q_{\fix} \circ q_c)^{-1}(D))$. 
        A moment's thought reveals that 
            \begin{align*}
                (T_{\varphi_c})_t f_i = \sum_{j=1}^{d} \mathrm{ext}( ((T_{\varphi_c})_t f_i|f_j)) f_j,
            \end{align*}
        hence $(T_{\varphi_c})_t f_i \in \Gamma$ for all $i \in \{1, \dots, d\}$ and $t \in G$. Thus, $\Gamma$ is also $T_{\varphi_c}$-invariant. We conclude from Lemma \ref{extensionlemma} that there is a unique finitely generated, projective, closed and $T_{\varphi_c}$-invariant $\mathrm{C}(L)$-submodule $\Lambda \subset \mathrm{C}(L \times_{L_{\fix}} H/U)$ with $\Lambda|_{(q_{\fix}\circ q)^{-1}(D)} = \Gamma$. In particular, $f_1, \dots, f_{d}$ can be extended to continuous functions on the whole space $L \times_{L_{\fix}} H/U$, and we use the same symbols for them. Now, if $x \in H$ we choose any $l \in L$ with $m \coloneqq p(x) = q_{\fix}(l) \in \overline{V}$. We then find, since $p$ and $q_{\fix}$ are open (see Lemma \ref{lemstonean}), nets $(x_\alpha)_{\alpha \in A}$ in $H$ converging to $x$ and $(l_\alpha)_{\alpha \in A}$ in $L$ converging to $l$ such that $m_\alpha \coloneqq p(x_\alpha) =  q_{\mathrm{fix}}(l_\alpha) \in V$ for every $\alpha \in A$. This implies
            \begin{align*}
                f_i(l,xU_m) = \lim_{\alpha} f_i(l_\alpha, x_\alpha U_{m_\alpha}) = \lim_\alpha \tau_i(m_\alpha)(x_\alpha) = \tau_i(m)(x)
            \end{align*}
        for every $i \in \{1, \dots, d\}$. This shows $\tau_i(m)(xu) = \tau_i(m)(x)$ for all $x \in H_m$, $u \in U_m$, hence $\tau_i(m) \in \mathrm{R}_{\hat{\sigma},U}$ for every $m \in \overline{V}$ as claimed.\medskip\\
        As stated above, this shows that $O_{d}$ is closed, hence clopen in $L_{\fix}$. Restricting the section to the clopen subset $O \setminus O_{d}$ and iterating the argument yields the claim.
    \end{proof}
    Note that in the case of ergodic systems, our proof of Theorem \ref{mz} simplifies considerably (it is basically a combination of the topological representation theorem Corollary \ref{corold} with Proposition \ref{model}, Proposition \ref{dynamicaldisint} and Theorem \ref{uniquerel}). In this case we obtain the following version of the Mackey--Zimmer representation theorem which is close to \cite[Theorem 5.36]{ellis}.
    \begin{theorem}
         Let $J \colon (\uY;S) \rightarrow (\uX;T)$ be an extension of ergodic measure-preserving systems with relative discrete spectrum. Then there are
           \begin{enumerate}[(i)]
            \item a topological dynamical system $(L;\psi)$,
            \item a compact group $H$,
            \item a core-free closed subgroup $U \subset H$, 
            \item an ergodic cocycle $c \colon G \times L \rightarrow H$, and
            \item a normal measure $\mu_L$ on $L$,
        \end{enumerate}
            such that the induced skew-product extension 
                \begin{align*}
                    J_{q_c} \colon (L,\mu_L;T_\psi) \rightarrow \left(L \times H/U, \mu_L \times \lambda_{H/U};T_{\varphi_c}\right)
                \end{align*}
            is isomorphic to $J$.
    \end{theorem}

\end{document}